\documentclass[3p,times]{elsarticle}

\usepackage[utf8]{inputenc}
\usepackage[T1]{fontenc}
\usepackage{longtable}
\usepackage{multirow}
\usepackage{color}

\usepackage{amssymb}
\usepackage{bbm,amssymb,amsfonts,amsmath,amsopn,amsthm}
\usepackage{enumerate}

\usepackage[figuresright]{rotating}

\usepackage[colorlinks=true,linkcolor=magenta,citecolor=blue]{hyperref}
\usepackage{color,tikz}
\usetikzlibrary{decorations.pathreplacing,calc}
\usepackage{rotating}
\tikzset{liltext/.style={font=\tiny}}
\usepackage{float}
\newcommand\Zhyp{Z_{{\text{hyp}}}}
\newcommand\Zpd{Z_{{\text{pd}}}}
\newcommand\Zred{Z_{{\text{red}}}}
\newcommand\Zell{Z_{{\text{ell}}}}
\DeclareMathOperator{\diag}{diag}
\DeclareMathOperator{\antidiag}{antidiag}

\begin{document}
	\newtheorem{theorem}{\quad Theorem}[section]
	\newtheorem{definition}[theorem]{\quad Definition}
	\newtheorem{corollary}[theorem]{\quad Corollary}
	\newtheorem{lemma}[theorem]{\quad Lemma}
	\newtheorem{proposition}[theorem]{\quad Proposition}
	\newtheorem{exam}[theorem]{\quad Example}
	\newtheorem{remark}[theorem]{\quad Remark}
\begin{frontmatter}
		
		

		
\title{Visco-elastic damped wave models with time-dependent coefficient}
		
\author[label1]{Halit Sevki Aslan}\author[label2]{Michael Reissig \corref{*}}
\address[label1]{Department of Computing and Mathematics (FFCLRP), University of São Paulo (USP), Ribeirão Preto, SP, 14040--901, Brazil \\}
\address[label2]{Faculty for Mathematics and Computer Science, Technical University Bergakademie Freiberg, Prüferstr. 9, 09596 Freiberg, Germany \\
halitsevkiaslan@gmail.com \qquad reissig@math.tu-freiberg.de}
\cortext[*]{Corresponding author}

\begin{abstract}
In this paper, we study the following Cauchy problem for linear visco-elastic damped wave models with a general time-dependent coefficient $g=g(t)$:
\begin{equation} \label{EqAbstract} \tag{$\star$}
\begin{cases}
u_{tt}- \Delta u + g(t)(-\Delta)u_t=0, &(t,x) \in (0,\infty) \times \mathbb{R}^n, \\
u(0,x)= u_0(x),\quad u_t(0,x)= u_1(x), &x \in \mathbb{R}^n.
\end{cases}
\end{equation}
We are interested to study the influence of the damping term $g(t)(-\Delta)u_t$ on qualitative properties of solutions to \eqref{EqAbstract} as decay estimates for energies of higher order and the parabolic effect. The main tools are related to WKB-analysis. We apply elliptic as well as hyperbolic WKB-analysis in different parts of the extended phase space.
\end{abstract}
		
	
\end{frontmatter}
	
	

\section{Introduction}\label{Sec_Intro}
\subsection{Historical remarks} \label{Section1.1}
It is well-known from the theory of visco-elasticity that visco-elastic materials undergoing deformation exhibit dual properties of viscosity and elasticity, which can keep the memory of their entire history and show natural damping. Therefore, the study of visco-elastic mechanical equations has wide application in the natural sciences and has become an important area of research.

Let us first recall some historical background for the following Cauchy problem with a time-dependent structural damping:
\begin{equation} \label{IntrHistorical}
\begin{cases}
u_{tt} - \Delta u + g(t)(-\Delta)^\delta u_t=0, &(t,x) \in (0,\infty) \times \mathbb{R}^n, \\
u(0,x)= u_0(x),\quad u_t(0,x)= u_1(x), &x \in \mathbb{R}^n,
\end{cases}
\end{equation}
with $\delta\in(0,1)$. This model was recently studied by several authors, in particular, see \cite{Wirth.non-eff, Wirth.eff.} for $\delta=0$. In \cite{ReissigLu2009}, the authors studied the Cauchy problem \eqref{IntrHistorical} with a decreasing time-dependent coefficient $g(t) = \mu(1+t)^{-\alpha}$, where $\mu>0$ and $\alpha\in[0,1]$. They studied the decay behavior of the energies of higher order of solutions and determined decay rates depending on the order of the energy. If higher order energies decay faster with increasing order, then the authors called this effect as ``parabolic effect''. Later, in \cite{Reissig} the author studied the Cauchy problem \eqref{IntrHistorical} with strictly increasing in time coefficients $g(t) = \mu(1+t)^{\alpha}$, where $\mu>0$ and $\alpha\in(0,1]$. In \cite{MatthesReissig2013}, the authors considered wave models of \eqref{IntrHistorical} with a scale-invariant coefficient $g=g(t)$, that is, $g(t)=\mu(1+t)^{2\delta-1}$ with $\mu>0$ and $\delta\in(0,1]$. They proved the optimality of their decay estimates and studied smoothing effects for solutions to structurally damped waves. \\
Let us mention briefly also some recent contributions related to the following Cauchy problem, the so-called \emph{structurally damped $\sigma$-evolution equation with time-dependent dissipation}:
\begin{equation}\label{Equation_sigma-delta}
\begin{cases}
u_{tt}+ (-\Delta)^\sigma u+ g(t) (-\Delta)^\delta u_t = 0, &(t,x) \in (0,\infty) \times \mathbb{R}^n, \\
u(0,x)= u_0(x),\quad u_t(0,x)= u_1(x), &x \in \mathbb{R}^n,
\end{cases}
\end{equation}
with $\sigma> 1$ and $\delta\in(0,\sigma)$. The authors in \cite{Kainane2014,KainaneReissig2015-1,KainaneReissig2015-2} developed WKB-analysis to derive an explicit representation formula based on Fourier multipliers for solutions when the dissipation coefficient $g=g(t)$ is supposed to be a monotonous function.
Using the obtained representation formula, they derived some $L^2-L^2$ decay estimates and $L^p-L^q$ estimates on the conjugate line for the energies of higher order. Moreover, some qualitative properties of energy solutions such as parabolic and smoothing effects were explained in detail. Coming back to the particular case of $g(t)= \mu(1+t)^{-\alpha}$ with a constant $\mu>0$ and $\alpha\in (-1,1)$, one can see that a classification between effective damping and noneffective damping, which strongly depends on parameters $\sigma$, $\delta$ and $\alpha$, is introduced in \cite{DAbbiccoEbert2016}. Their main goal is to study the asymptotic profile of solutions to \eqref{Equation_sigma-delta} and simultaneously to clarify that in the effective damping case, a diffusion phenomenon occurs. Physical motivations for the role played by a time-dependent structural damping term can be found in \cite{FangLuReissig2010}. For more details about the qualitative properties of solutions to wave models with visco-elastic damping and constant coefficients, we refer to Section 14.3 of \cite{EbertReissigBook} and pioneering works \cite{Ponce1985, Shibata2000}.
\subsection{Main purpose of the paper} \label{Section1.2}
This paper is concerned with studying the following Cauchy problem for the linear visco-elastic damped wave equation with a time-dependent coefficient $g=g(t)$:
\begin{equation} \label{modellinearviscoelasticdamped}
\begin{cases}
u_{tt}- \Delta u + g(t)(-\Delta)u_t=0, &(t,x) \in (0,\infty) \times \mathbb{R}^n, \\
u(0,x)= u_0(x),\quad u_t(0,x)= u_1(x), &x \in \mathbb{R}^n.
\end{cases}
\end{equation}
Here $g=g(t)$ is a time-dependent coefficient satisfying some appropriate conditions like positivity, continuity, monotonicity, and some behavior of derivatives. The role played by the visco-elastic damping in \eqref{modellinearviscoelasticdamped} varies with the choices of the coefficient $g=g(t)$ and different approaches are required to study its influence on the asymptotic profile of the solution as $t\rightarrow\infty$. To the best of the authors' knowledge, it seems that wave models with general time-dependent coefficient in the visco-elastic damping term have not been studied before. Our main goal is to study decay rates of energies of higher order for solutions to the Cauchy problem \eqref{modellinearviscoelasticdamped} with a general time-dependent coefficient in the visco-elastic damping. These estimates rely on the structural properties of representations of solutions. Our approach is based on asymptotic representations combined with an extended phase space analysis under some assumptions, mostly adapted from the treatment of problems in WKB-analysis.\medskip

We divide our considerations into the following cases of the time-dependent coefficients:
\begin{itemize}
\item Models with increasing time-dependent coefficient $g=g(t)$ in Section \ref{Section4.1}.
\item Models with integrable and decreasing time-dependent coefficient $g=g(t)$ in Section \ref{Section4.2}.
\item Models with non-integrable and decreasing time-dependent coefficient $g=g(t)$ in Section \ref{Section4.3}.
\item Models with non-integrable and slowly increasing time-dependent coefficient $g=g(t)$ in Section \ref{Section4.4}.
\end{itemize}
We will treat separately all four cases. After explaining how WKB-analysis should be applied to each of the cases, we sometimes shorten calculations and refer to older papers where the necessary calculations are given in detail to treat other models. In this way we try to shorten the paper without loosing readability.\medskip

\textbf{Notations}
\begin{itemize}
\item We write $f\lesssim g$ when there exists a constant $C>0$ such that $f\leq Cg$, and $f \approx g$ when $g\lesssim f\lesssim g$.
\item By $f\sim g$ we denote $\lim_{t\rightarrow\infty}\dfrac{f(t)}{g(t)}=1$, that is, $f$ and $g$ have the same asymptotic behavior.
\item As usual, the spaces $H^a$ and $\dot{H}^a$ with $a \geq 0$ stand for Bessel and Riesz potential spaces based on the $L^2$ space. Here $\big<D\big>^a$ and $|D|^a$ denote the pseudo-differential operator with symbol $\big<\xi\big>^a$ and the fractional Laplace operator with symbol $|\xi|^a$, respectively.
\item $(|\,\cdot\,|)$ denotes the matrix of absolute values of its entries for a given matrix.
\end{itemize}

\subsection{Our approach} \label{Section1.3}
We apply the partial Fourier transformation with respect to spatial variables to the Cauchy problem \eqref{modellinearviscoelasticdamped}. So we get that $\hat{u}(t,\xi)=\mathcal{F}_{x\rightarrow\xi}\left( u(t,x) \right)$ solves
\begin{equation} \label{modellinearviscoelasticdampedFourier}
\begin{cases}
\hat{u}_{tt}+|\xi|^2\hat{u} + g(t)|\xi|^2\hat{u}_t=0, &(t,\xi) \in (0,\infty) \times \mathbb{R}^n, \\
\hat{u}(0,\xi) = \hat{u}_0(\xi),\quad \hat{u}_t(0,\xi)= \hat{u}_1(\xi), &\xi \in \mathbb{R}^n.
\end{cases}
\end{equation}
We apply the change of variables
\[ \hat{u}(t,\xi)=\exp\bigg(-\frac{1}{2} \int_0^t g(\tau)|\xi|^2 d\tau\bigg) v(t,\xi) \]
to arrive at
\begin{equation} \label{modellinearauxiliary}
\begin{cases}
v_{tt} + |\xi|^2\bigg(1-\dfrac{g(t)^2|\xi|^2}{4}-\dfrac{g'(t)}{2}\bigg) v=0, &(t,\xi) \in (0,\infty) \times \mathbb{R}^n, \\
v(0,\xi)= v_0(\xi),\quad v_t(0,\xi)= v_1(\xi), &\xi \in \mathbb{R}^n,
\end{cases}
\end{equation}
where
\[ v_0(\xi) = \hat{u}_0(\xi) \quad \text{and} \quad v_1(\xi) = \frac{g(0)}{2}|\xi|^2\hat{u}_0(\xi) + \hat{u}_1(\xi). \]
Examples for $g=g(t)$ are
\begin{align*}
& \bullet \quad g(t)=e^t, & \bullet &\quad g(t)=e^{-t}, & \bullet \quad g(t)=(C_d+t)^d \,\, \text{ with } \,\, d \in \mathbb{R}, \\
& \bullet \quad g(t)=(1+t)\log(e+t), & \bullet &\quad g(t)=\mu (1+t) \,\, \text{ with } \,\, \mu >0, \\
& \bullet \quad g(t)=e^{e^t}, & \bullet &\quad g(t)=e^{-e^t}.
\end{align*}
To study (\ref{modellinearauxiliary}) we analyze
\[ 1-\frac{g(t)^2|\xi|^2}{4}-\frac{g'(t)}{2}=0. \]
This equation divides the extended phase space $(0,\infty) \times \mathbb{R}^n$ into two regions, the hyperbolic region $\Pi_{hyp}$ and the elliptic region $\Pi_{ell}$, as follows:
\[ \Pi_{hyp}=\bigg\{(t,\xi): 1-\frac{g(t)^2|\xi|^2}{4}-\frac{g'(t)}{2}>0\bigg\}, \qquad \Pi_{ell}=\bigg\{(t,\xi): 1-\frac{g(t)^2|\xi|^2}{4}-\frac{g'(t)}{2}<0\bigg\}.\]
The assumptions for $g=g(t)$ are organized in such a way that we can define a separating line
\[ t_\xi=t(|\xi|)=\bigg\{ (t,\xi) \in (0,\infty) \times \mathbb{R}^n: 1-\frac{g(t)^2|\xi|^2}{4}-\frac{g'(t)}{2}=0 \bigg\}.\]
We are going to consider the following two cases:\medskip

{\it \underline{Application of elliptic WKB-analysis}}:\medskip

\noindent We assume that
\[ 1-\frac{g'(t)}{2}=-h(t)^2 g(t)^2 \,\,\,\mbox{for all}\,\,\, t>0,\]
where $h=h(t)$ is a positive function for all $t>0$. This function has to satisfy usual properties to carry out steps from WKB-analysis. In this case the equation (\ref{modellinearauxiliary}) becomes
\begin{eqnarray} \label{modellinearauxiliary1}
v_{tt} - g(t)^2 |\xi|^2\bigg( h(t)^2+\frac{|\xi|^2}{4} \bigg)v=0.
\end{eqnarray}
So, we have only to apply tools from elliptic WKB-analysis to get WKB-representations of solutions.
\begin{exam}  \label{Example1}
Let us choose $g(t)=3e^t$. We study (\ref{modellinearauxiliary1}) with the function
$h=h(t)=\frac{\sqrt{3e^t-2}}{3\sqrt{2}e^t}$.
\end{exam}
\begin{exam}  \label{Example3}
Let us choose $g(t)=3e^{e^t}$. We study (\ref{modellinearauxiliary1}) with the function $h=h(t)=\frac{\sqrt{3e^{e^t}e^t-2}}{3\sqrt{2}e^{e^t}}$.
\end{exam}
\begin{exam}  \label{Example6}
Let us choose $g(t)=(C_d+t)^d$ with $d > 1$ and $C_d=4^{\frac{1}{d}}$ for example. We study (\ref{modellinearauxiliary1}) with the function $h=h(t)=\frac{\sqrt{(C_d+t)^d-2}}{\sqrt{2}(C_d+t)^d}$.
\end{exam}
\begin{exam}  \label{Example7}
Let us choose $g(t)=3(1+t)\log(e+t)$. We study (\ref{modellinearauxiliary1}) with the function $h=h(t)=\frac{\sqrt{3\log(e+t)+3 \frac{1+t}{e+t}-2}}{3\sqrt{2}(1+t)\log(e+t)}$.
\end{exam}
\begin{exam}  \label{Example9}
Let us choose $g(t)=\mu(1+t)$ with $\mu>2$. We study (\ref{modellinearauxiliary1}) with the function $h=h(t)=\sqrt{\frac{\mu-2}{2\mu^2}} \frac{1}{1+t}$.
\end{exam}
{\it \underline{Application of elliptic and hyperbolic WKB-analysis}}:\medskip

\noindent We assume that
\[ 1-\frac{g'(t)}{2}=h(t)^2 g(t)^2 \,\,\,\mbox{for all} \,\,\, t>0, \]
where $h=h(t)$ is a positive function for all $t>0$. This function has to satisfy usual properties to carry out steps from WKB-analysis. In this case the equation (\ref{modellinearauxiliary}) becomes
\begin{equation} \label{modellinearauxiliary2}
v_{tt} +g(t)^2  |\xi|^2\bigg(h(t)^2-\frac{|\xi|^2}{4}\bigg) v=0.
\end{equation}
So, we have to apply tools from elliptic WKB-analysis and from hyperbolic WKB-analysis as well to get WKB-representations of solutions.
The hyperbolic and elliptic regions can be easily described by
\[ \Pi_{hyp}=\bigg\{ (t,\xi):  h(t) > \frac{|\xi|}{2} \bigg\}, \qquad \Pi_{ell}=\bigg\{ (t,\xi): h(t) < \frac{|\xi|}{2} \bigg\}.\]
\begin{exam}  \label{Example2}
Let us choose $g(t)=e^{-t}$. We study (\ref{modellinearauxiliary2}) with the function
$h=h(t)= e^t\frac{\sqrt{e^{-t}+2}}{\sqrt{2}}$. The separating line is defined by $|\xi|=2h(t)$ for all $t>0$.
So, the hyperbolic region is very large. The elliptic region is large.
\end{exam}
\begin{exam}  \label{Example4}
Let us choose $g(t)=e^{-e^t}$. We study (\ref{modellinearauxiliary2}) with the function
$h=h(t)=e^{e^t}\frac{\sqrt{e^{-e^t}e^t+2}}{\sqrt{2}}$.
The separating line is defined by $|\xi|=2h(t)$ for all $t>0$.
So, the hyperbolic region is very large. The elliptic region is large.
\end{exam}
\begin{exam}  \label{Example10}
Let us choose $g(t)=\mu(1+t)$ with $\mu \in (0,2)$. We study (\ref{modellinearauxiliary2}) with the function $h=h(t)=\sqrt{\frac{2-\mu}{2\mu^2}} \frac{1}{1+t}$.
\end{exam}
\begin{exam}  \label{Example5}
Let us choose $g(t)=(1+t)^d$ with $d < 1$. We study (\ref{modellinearauxiliary2}) with the function
$h=h(t)=\frac{\sqrt{2-d(1+t)^{d-1}}}{\sqrt{2}(1+t)^d}$.
The separating line is defined by $|\xi|=2h(t)$ for all $t>0$.
If $d \in (0,1]$, then the hyperbolic region is small, the elliptic region is very large. If $d<0$, then
the hyperbolic region is very large. The elliptic region is large.
\end{exam}
\begin{exam}  \label{Example8}
We choose $g(t)=((1+t)\log(e+t))^{-1}$ and study (\ref{modellinearauxiliary2}) with $h=h(t)=\frac{\sqrt{\frac{\log(e+t)+\frac{1+t}{e+t}}{((1+t)\log(e+t))^2}+2}}{\sqrt{2}((1+t)\log(e+t))^{-1}}$. The separating line is defined by $|\xi|=2h(t)$ for all $t>0$.
So, the hyperbolic region is very large. The elliptic region is large.
\end{exam}

\section{Models with increasing time-dependent coefficient $g=g(t)$} \label{Section4.1}
We assume the following properties of the function $g=g(t)$:
\begin{enumerate}
\item[\textbf{(A1)}] $g(t)>0$ and $g'(t)>0$ for all $t \in [0,\infty)$,
\item[\textbf{(A2)}] $\dfrac{1}{g} \in L^1(0,\infty)$,
\item[\textbf{(A3)}] $|d_t^kg(t)|\leq C_kg(t)\Big( \dfrac{g(t)}{G(t)} \Big)^k$ for all $t \in [0,\infty)$, $k=1,2$, where $G(t):=\dfrac{1}{2}\displaystyle\int_0^t g(\tau)d\tau$ and $C_1$, $C_2$ are positive constants.
\end{enumerate}
\begin{theorem} \label{Theorem1}
Let us consider the Cauchy problem
\begin{equation*}
\begin{cases}
u_{tt}- \Delta u + g(t)(-\Delta)u_t=0, &(t,x) \in (0,\infty) \times \mathbb{R}^n, \\
u(0,x)= u_0(x),\quad u_t(0,x)= u_1(x), &x \in \mathbb{R}^n.
\end{cases}
\end{equation*}
We assume that the coefficient $g=g(t)$ satisfies the conditions \textbf{(A1)} to \textbf{(A3)} and $(u_0,u_1)\in \dot{H}^{|\beta|} \times \dot{H}^{|\beta|-2}$ with $|\beta|\geq 2$. Then, we have the following estimates for Sobolev solutions:
\begin{align*}
\big\||D|^{|\beta|} u(t,\cdot)\big\|_{L^2} & \lesssim  \|u_0\|_{\dot{H}^{|\beta|}} + \|u_1\|_{\dot{H}^{|\beta|-2}},\\
\big\||D|^{|\beta|-2} u_t(t,\cdot)\big\|_{L^2} & \lesssim g(t)\Big( \|u_0\|_{\dot{H}^{|\beta|}} + \|u_1\|_{\dot{H}^{|\beta|-2}} \Big).
\end{align*}
\end{theorem}
\begin{remark}
The statements of Theorem \ref{Theorem1} imply that we do not have any parabolic effect, that is, higher order energies do not decay faster with
increasing order. Theorem \ref{Theorem1} can be applied to Examples \ref{Example1}, \ref{Example3} and \ref{Example6}.
\end{remark}
\begin{proof}[Proof of Theorem \ref{Theorem1}]
We write equation (\ref{modellinearauxiliary1}) in the form
\begin{equation} \label{auxiliaryproblem3}
D_t^2 v + \frac{g(t)^2}{4}|\xi|^4 v + \bigg( \frac{g'(t)}{2}-1 \bigg)|\xi|^2 v=0.
\end{equation}
The influence of the term $\dfrac{g'(t)}{2}|\xi|^2 v$ is dominant to the influence of the term $-|\xi|^2 v$. \\
\medskip

We divide the extended phase space $[0,\infty)\times \mathbb{R}^n$ into zones as follows:
\begin{itemize}
\item pseudodifferential zone:
\begin{align*} \label{zonesellipticcase}
\Zpd(N)=\Big\{ (t,\xi)\in [0,\infty)\times\mathbb{R}^n: G(t)|\xi|^2 \leq N \Big\},
\end{align*}
\item elliptic zone:
\begin{align*} \Zell(N)=\Big\{ (t,\xi)\in [0,\infty)\times \mathbb{R}^n: G(t)|\xi|^2 \geq N \Big\},
\end{align*}
\end{itemize}
where $N>0$ is sufficiently large.
The separating line $t_\xi=t(|\xi|)$ is defined by
\[ t_\xi=\Big\{ (t,\xi) \in [0,\infty) \times \mathbb{R}^n: G(t)|\xi|^2=N \Big\}.\]

\begin{figure}[H]
\begin{center}
\begin{tikzpicture}[>=latex,xscale=1.1]
	\draw[->] (0,0) -- (4,0)node[below]{$|\xi|$};
	\draw[->] (0,0) -- (0,4)node[left]{$t$};
    \node[below left] at(0,0){$0$};
    \node  at (.4,3.5) {$\textcolor{blue}{t_{\xi}}$};
	\draw[domain=0.2:3.8,color=blue,variable=\t] plot ({4.5*exp(-\t/1)},\t);
	\node[color=black] at (2.5, 1.7){{\footnotesize $Z_{\text{ell}}$}};
    \node[color=black] at (2.5, 1.3){{\footnotesize $\big( G(t)|\xi|^2\geq N \big)$}};
	\node[color=black] at (.9,0.8){{\footnotesize $Z_{\text{pd}}$}};
    \node[color=black] at (1,0.4){{\footnotesize $\big( G(t)|\xi|^2\leq N \big)$}};
\end{tikzpicture}
\caption{Sketch of the zones for the case $g=g(t)$ is increasing}
\label{fig.zone1}
\end{center}
\end{figure}
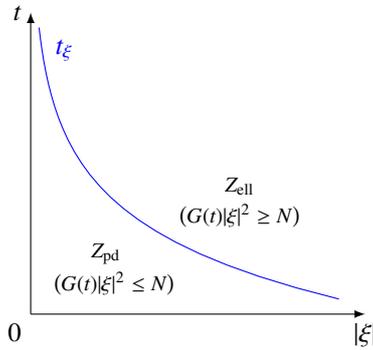

\subsection{Considerations in the elliptic zone $\Zell(N)$} \label{Section2.1}
Let us introduce the following family of symbol classes in the elliptic zone $\Zell(N)$.
\begin{definition} \label{Definition2.1}
A function $f=f(t,\xi)$ belongs to the elliptic symbol class $S_{\text{ell}}^\ell\{m_1,m_2\}$ if it holds
\begin{equation*} \label{Def.Symbol.Class.Ell.}
|D_t^kf(t,\xi)|\leq C_{k}\big( |\xi|^2g(t) \big)^{m_1}\Big( \frac{g(t)}{G(t)} \Big)^{m_2+k}
\end{equation*}
for all $(t,\xi)\in \Zell(N)$ and all $k\leq \ell$.
\end{definition}
Some useful rules of the symbolic calculus are collected in the following proposition.
\begin{proposition} \label{Prop.Symbol.Ell.}  The following statements are true:
\begin{itemize}
\item $S_{\text{ell}}^\ell\{m_1,m_2\}$ is a vector space for all nonnegative integers $\ell$;
\item $S_{\text{ell}}^\ell\{m_1,m_2\}\cdot S_{\text{ell}}^{\ell}\{m_1',m_2'\}\hookrightarrow S^{\ell}_{\text{ell}}\{m_1+m_1',m_2+m_2'\}$;
\item $D_t^kS_{\text{ell}}^\ell\{m_1,m_2\}\hookrightarrow S_{\text{ell}}^{\ell-k}\{m_1,m_2+k\}$
for all nonnegative integers $\ell$ with $k\leq \ell$;
\item $S_{\text{ell}}^{0}\{-1,2\}\hookrightarrow L_{\xi}^{\infty}L_t^1\big( \Zell(N) \big)$.
\end{itemize}
\end{proposition}
\begin{proof}
We only verify the integrability statement. Indeed, if $f=f(t,\xi)\in S_{\text{ell}}^{0}\{-1,2\}$, then it holds
\begin{align*}
\int_{t_\xi}^{\infty}|f(\tau,\xi)|d\tau  & \lesssim \int_{t_\xi}^{\infty}\frac{1}{|\xi|^2g(\tau)}\Big( \frac{g(\tau)}{G(\tau)} \Big)^2 d\tau
\leq \frac{C}{G(t_{\xi})|\xi|^2} = \frac{C}{N},
\end{align*}
where we used the definition of the separating line $t_\xi$.
\end{proof}
With $\gamma=\gamma(t,\xi):=\dfrac{g(t)}{2}|\xi|^2$, we define the micro-energy $V(t,\xi):=\big( \gamma(t,\xi)v(t,\xi),D_tv(t,\xi) \big)^\text{T}$. Then, by \eqref{auxiliaryproblem3} we obtain that $V=V(t,\xi)$ satisfies the following system of first order:
\begin{equation} \label{Systemwithh}
D_tV=\underbrace{\left[ \left( \begin{array}{cc}
0 & \dfrac{g(t)}{2}|\xi|^2 \\
-\dfrac{g(t)}{2}|\xi|^2 & 0
\end{array} \right) + \left( \begin{array}{cc}
\dfrac{D_tg(t)}{g(t)} & 0 \\
-\dfrac{g'(t)-2}{g(t)} & 0
\end{array} \right)\right]}_{A_V}V
\end{equation}
with the initial condition $V(0,\xi)=\big( \gamma(0,\xi)v(0,\xi),D_tv(0,\xi) \big)^\text{T}$. We want to estimate the fundamental solution $E_V=E_V(t,s,\xi)$ to the system \eqref{Systemwithh}, namely, the solution to
\begin{equation*}
D_tE_V(t,s,\xi)=A_V(t,\xi)E_V(t,s,\xi), \quad E_V(s,s,\xi)=I \quad \mbox{for any} \quad t\geq s\geq t_\xi.
\end{equation*}

\noindent\textbf{Step 1.} \emph{Diagonalization procedure} \medskip

\noindent We denote by $M$ the matrix consisting of eigenvectors of the first matrix on the right-hand side and its inverse matrix
\[ M = \left( \begin{array}{cc}
i & -i \\
1 & 1
\end{array} \right), \qquad M^{-1}=\frac{1}{2}\left( \begin{array}{cc}
-i & 1 \\
i & 1
\end{array} \right). \]
Then, defining $V^{(0)}:=M^{-1}V$ we get the system
\begin{equation*}
D_tV^{(0)}=\big( \mathcal{D}(t,\xi)+\mathcal{R}(t) \big)V^{(0)},
\end{equation*}
where
\begin{align*}
\mathcal{D}(t,\xi) = \left( \begin{array}{cc}
-i\dfrac{g(t)}{2}|\xi|^2 & 0 \\
0 & i\dfrac{g(t)}{2}|\xi|^2
\end{array} \right) \qquad \text{and} \qquad \mathcal{R}(t) = \frac{1}{2} \left( \begin{array}{cc}
\dfrac{D_tg(t)}{2g(t)}-i\dfrac{g'(t)-2}{2g(t)} & -\dfrac{D_tg(t)}{2g(t)}+i\dfrac{g'(t)-2}{2g(t)} \\
-\dfrac{D_tg(t)}{2g(t)}-i\dfrac{g'(t)-2}{2g(t)} & \dfrac{D_tg(t)}{2g(t)}+i\dfrac{g'(t)-2}{2g(t)}
\end{array} \right),
\end{align*}
where $\mathcal{D}(t,\xi)\in S_{\text{ell}}^2\{1,0\}$ and $\mathcal{R}(t)\in S_{\text{ell}}^1\{0,1\}$.  Let us introduce $F_0(t)=\diag\mathcal{R}(t)$. Now we carry out the next step of diagonalization procedure. The difference of the diagonal entries of the matrix $\mathcal{D}(t,\xi)+F_0(t)$ is
\begin{equation*}
i\delta(t,\xi):=g(t)|\xi|^2 + \frac{g'(t)-2}{g(t)}\sim g(t)|\xi|^2
\end{equation*}
for $t\geq t_\xi$ if we choose the zone constant $N$ sufficiently large and apply condition \textbf{(A3)}. Now we choose a matrix $N^{(1)}=N^{(1)}(t,\xi)$ such that
\[ N^{(1)}(t,\xi)=\left( \begin{array}{cc}
0 & -\dfrac{\mathcal{R}_{12}}{\delta(t,\xi)} \\
\dfrac{\mathcal{R}_{21}}{\delta(t,\xi)} & 0
\end{array} \right)\sim \left( \begin{array}{cc}
0 & i\dfrac{D_tg(t)}{4g^2(t)|\xi|^2}-\dfrac{g'(t)-2}{4g^2(t)|\xi|^2} \\
i\dfrac{D_tg(t)}{4g^2(t)|\xi|^2}+\dfrac{g'(t)-2}{4g^2(t)|\xi|^2} & 0
\end{array} \right). \]
Taking into consideration the rules of the symbolic calculus we have
\[ N^{(1)}(t,\xi)\in S_{\text{ell}}^{1}\{-1,1\} \qquad \text{and} \qquad  N_1(t,\xi)=I+N^{(1)}(t,\xi)\in S_{\text{ell}}^{1}\{0,0\}. \]
For a sufficiently large zone constant $N$ and all $t\geq t_\xi$ the matrix $N_1=N_1(t,\xi)$ is invertible with uniformly bounded inverse $N_1^{-1}=N_1^{-1}(t,\xi)$. Indeed, in the elliptic zone $\Zell(N)$ it holds
\[ |N_1(t,\xi)-I| \leq \frac{C}{g(t)|\xi|^2}\frac{g(t)}{G(t)} = \frac{C}{G(t)|\xi|^2}\leq \frac{C}{N}. \]
Let
\begin{align*}
B^{(1)}(t,\xi) &= D_tN^{(1)}(t,\xi)-( \mathcal{R}(t)-F_0(t,\xi))N^{(1)}(t,\xi), \\
\mathcal{R}_1(t,\xi) &= -N_1^{-1}(t,\xi)B^{(1)}(t,\xi)\in S_{\text{ell}}^{0}\{-1,2\},
\end{align*}
where $N_{1}(t,\xi)=I+N^{(1)}(t,\xi)$. Then, we have the following operator identity:
\begin{equation*}
\big( D_t-\mathcal{D}(t,\xi)-\mathcal{R}(t) \big)N_1(t,\xi)=N_1(t,\xi)\big( D_t-\mathcal{D}(t,\xi)-F_0(t)-\mathcal{R}_1(t,\xi) \big).
\end{equation*}
\noindent\textbf{Step 2.} \emph{Construction of the fundamental solution}
\begin{proposition} \label{Lem.Est.Ell.Zone}
The fundamental solution $E_{\text{ell}}^{V}=E_{\text{ell}}^{V}(t,s,\xi)$ to the transformed operator
\[ D_t-\mathcal{D}(t,\xi)-F_0(t)-\mathcal{R}_1(t,\xi) \]
can be estimated by
\begin{equation*}
\big( |E_{\text{ell}}^{V}(t,s,\xi)| \big) \lesssim \frac{g(t)}{g(s)}\exp\bigg( \frac{|\xi|^2}{2}\int_{s}^{t}g(\tau)d\tau \bigg)
\left( \begin{array}{cc}
1 & 1 \\
1 & 1
\end{array} \right),
\end{equation*}
with $(t,\xi),(s,\xi)\in \Zell(N)$, $t_\xi\leq s\leq t$.
\end{proposition}
\begin{proof}
To prove this proposition we can follow the proof to Theorem 15 of \cite{Wirth.eff.}.
\end{proof}
Now let us come back to
\begin{align} \label{IncZellBack}
V(t,\xi) = E_V(t,s,\xi)V(s,\xi),\quad \text{that is}, \quad \left( \begin{array}{cc}
\gamma(t,\xi)v(t,\xi) \\
D_t v(t,\xi)
\end{array} \right) = E_V(t,s,\xi)\left( \begin{array}{cc}
\gamma(s,\xi)v(s,\xi) \\
D_tv(s,\xi)
\end{array} \right) \quad \text{for} \quad t_\xi\leq s\leq t.
\end{align}
Therefore, from Proposition \ref{Lem.Est.Ell.Zone}, the backward transformations to the above used transformations and \eqref{IncZellBack} we may conclude the following estimates for $t_\xi \leq s \leq t$:
\begin{align*}
\gamma(t,\xi)|v(t,\xi)| & \lesssim \frac{g(t)}{g(s)}\exp\bigg( \frac{|\xi|^2}{2}\int_{s}^{t}g(\tau)d\tau \bigg)\Big( \gamma(s,\xi)|v(s,\xi)| + |v_t(s,\xi)| \Big), \\
|v_t(t,\xi)| & \lesssim \frac{g(t)}{g(s)}\exp\bigg( \frac{|\xi|^2}{2}\int_{s}^{t}g(\tau)d\tau \bigg)\Big( \gamma(s,\xi)|v(s,\xi)| + |v_t(s,\xi)| \Big).
\end{align*}
Using the backward transformation $v(t,\xi)=\exp\Big( \frac{|\xi|^2}{2} \int_0^t g(\tau) d\tau \Big)\hat{u}(t,\xi)$, we arrive immediately at the following result.
\begin{corollary} \label{IncreasingCorZell}
We have the following estimates in the elliptic zone $\Zell(N)$ for $t_\xi \leq s \leq t$:
\begin{align*}
|\xi|^{|\beta|}|\hat{u}(t,\xi)| & \lesssim |\xi|^{|\beta|}|\hat{u}(s,\xi)| + \frac{1}{g(s)}|\xi|^{|\beta|-2}|\hat{u}_t(s,\xi)| \quad \mbox{for} \quad |\beta|\geq 2, \\
|\xi|^{|\beta|}|\hat{u}_t(t,\xi)| & \lesssim g(t)|\xi|^{|\beta|+2}|\hat{u}(s,\xi)| + \frac{g(t)}{g(s)}|\xi|^{|\beta|}|\hat{u}_t(s,\xi)| \quad \mbox{for} \quad |\beta|\geq 0.
\end{align*}
\end{corollary}
\subsection{Considerations in the pseudo-differential zone $\Zpd(N)$} \label{Section2.2}
We define the micro-energy $U=\big( \gamma(t,\xi)\hat{u},D_t\hat{u} \big)^\text{T}$ with $\gamma(t,\xi):=\dfrac{g(t)}{2}|\xi|^2$. Then, the Cauchy problem \eqref{modellinearviscoelasticdampedFourier} leads to the system of first order
\begin{equation} \label{system.diss.zone}
D_tU=\underbrace{\left( \begin{array}{cc}
\dfrac{D_t\gamma(t,\xi)}{\gamma(t,\xi)} & \gamma(t,\xi) \\
\dfrac{|\xi|^2}{\gamma(t,\xi)} & ig(t)|\xi|^2
\end{array} \right)}_{A(t,\xi)}U.
\end{equation}
We are interested in the fundamental solution $E_{\text{pd}}=E_{\text{pd}}(t,s,\xi)$ to the system \eqref{system.diss.zone}, that is, the solution of
\[ D_tE_{\text{pd}}(t,s,\xi)=A(t,\xi)E_{\text{pd}}(t,s,\xi), \quad E_{\text{pd}}(s,s,\xi)=I, \]
for all $0\leq s \leq t$ and $(t,\xi), (s,\xi) \in \Zpd(N)$. Thus, the solution $U=U(t,\xi)$ is represented as
\[ U(t,\xi)=E_{\text{pd}}(t,s,\xi)U(s,\xi). \]
We will use the auxiliary function
\[ \delta=\delta(t,\xi)=\exp\bigg( \frac{|\xi|^2}{2}\int_{0}^{t}g(\tau)d\tau \bigg) = \exp\Big( |\xi|^2G(t) \Big) \lesssim 1. \]
The entries $E_{\text{pd}}^{(k\ell)}(t,s,\xi)$, $k,\ell=1,2,$ of the fundamental solution $E_{\text{pd}}(t,s,\xi)$ satisfy the following system for $\ell=1,2$:
\begin{eqnarray*}
D_tE_{\text{pd}}^{(1\ell)}(t,s,\xi) &=& \frac{D_t\gamma(t,\xi)}{\gamma(t,\xi)}E_{\text{pd}}^{(1\ell)}(t,s,\xi)+\gamma(t,\xi)E_{\text{pd}}^{(2\ell)}(t,s,\xi), \\
D_tE_{\text{pd}}^{(2\ell)}(t,s,\xi) &=& \frac{|\xi|^2}{\gamma(t,\xi)}E_{\text{pd}}^{(1\ell)}(t,s,\xi)+ig(t)|\xi|^2E_{\text{pd}}^{(2\ell)}(t,s,\xi).
\end{eqnarray*}
Then, by straight-forward calculations (with $\delta_{k \ell}=1$ if $k=\ell$ and $\delta_{k \ell}=0$ otherwise), we get
\begin{align*}
E_{\text{pd}}^{(1\ell)}(t,s,\xi) & = \frac{\gamma(t,\xi)}{\gamma(s,\xi)}\delta_{1\ell}+i\gamma(t,\xi)\int_{s}^{t}E_{\text{pd}}^{(21)}(\tau,s,\xi)d\tau, \\
E_{\text{pd}}^{(2\ell)}(t,s,\xi) & = \frac{\delta^2(s,\xi)}{\delta^2(t,\xi)}\delta_{2\ell}+\frac{i|\xi|^2}{\delta^2(t,\xi)}\int_{s}^{t}\frac{1}{\gamma(\tau,\xi)}\delta^2(\tau,\xi)
E_{\text{pd}}^{(12)}(\tau,s,\xi)d\tau.
\end{align*}
To complete the proof of Proposition \ref{CorEstPDZone} the following lemma is useful.
\begin{lemma} [Gronwall's inequality] \label{Appendix.Gronwall}
Let $f$ and $h$ be continuous and nonnegative functions defined on $J=[a,b]$ and let $d$ be a continuous, positive and nondecreasing function defined on $J$. Then, the inequality
\[ f(t)\leq d(t)+\int_{a}^{t}h(r)f(r)dr, \quad t\in J, \]
implies that
\[ f(t)\leq d(t)\exp\left( \int_{a}^{t}h(r) dr \right), \quad t\in J. \]
\end{lemma}
\begin{proposition} \label{CorEstPDZone}
We have the following estimates in the pseudo-differential zone:
\begin{equation*}
(|E_{\text{pd}}(t,s,\xi)|) \lesssim \frac{g(t)}{g(s)}
\left( \begin{array}{cc}
1 & 1 \\
1 & 1
\end{array} \right)
\end{equation*}
with $(s,\xi),(t,\xi)\in\Zpd(N)$ and $0\leq s\leq t\leq t_\xi$.
\end{proposition}
\begin{proof}
To prove this proposition we can follow the proof to Lemma 3.10 of \cite{ReissigLu2009}.
\end{proof}
Now let us come back to
\begin{equation} \label{EqPseudoDiffZone}
U(t,\xi) = E(t,0,\xi)U(0,\xi) \quad \text{for all} \quad 0\leq t\leq t_{\xi}.
\end{equation}
Because of \eqref{EqPseudoDiffZone} and Proposition \ref{CorEstPDZone}, the following statement can be concluded.
\begin{corollary} \label{IncreasingCorZpd}
In the pseudo-differential zone $\Zpd(N)$ the following estimates hold for all $0\leq t\leq t_{\xi}$:
\begin{align*}
|\xi|^{|\beta|}|\hat{u}(t,\xi)| &\lesssim |\xi|^{|\beta|}|\hat{u}_0(\xi)| +|\xi|^{|\beta|-2}|\hat{u}_1(\xi)| \quad \mbox{for} \quad |\beta|\geq 2, \\
|\xi|^{|\beta|}|\hat{u}_t(t,\xi)| &\lesssim g(t)|\xi|^{|\beta|+2}|\hat{u}_0(\xi)| + g(t)|\xi|^{|\beta|}|\hat{u}_1(\xi)| \quad \mbox{for} \quad |\beta|\geq 0.
\end{align*}
\end{corollary}
\subsection{Conclusion} \label{Section2.3}
From the statements of Corollaries \ref{IncreasingCorZell} and \ref{IncreasingCorZpd} we derive the following estimates for $t>0$:
\begin{align*}
|\xi|^{|\beta|}|\hat{u}(t,\xi)| & \lesssim |\xi|^{|\beta|}|\hat{u}_0(\xi)| + |\xi|^{|\beta|-2}|\hat{u}_1(\xi)| \quad \mbox{for} \quad |\beta|\geq2, \\
|\xi|^{|\beta|}|u_t(t,\xi)| & \lesssim g(t)|\xi|^{|\beta|+2}|\hat{u}_0(\xi)| + g(t)|\xi|^{|\beta|}|\hat{u}_1(\xi)| \quad \mbox{for} \quad |\beta|\geq0.
\end{align*}
This completes the proof of Theorem \ref{Theorem1}.
\end{proof}

\section{Models with integrable and decaying time-dependent coefficient $g=g(t)$} \label{Section4.2}
We write (\ref{modellinearauxiliary2}) in the form
\begin{align*} \label{auxiliaryproblem4}
D_t^2 v + \frac{g(t)^2}{4}|\xi|^4 v - \bigg( 1-\frac{g'(t)}{2} \bigg)|\xi|^2 v=0.
\end{align*}
The influence of the term $-|\xi|^2 v$ is dominant to the influence of the term $\dfrac{g'(t)}{2}|\xi|^2 v$.
Examples for this case are given in Examples \ref{Example2}, \ref{Example4} and \ref{Example5} with $d<-1$. Here let us recall that damped wave models with integrable and decaying in time speed of propagation have been studied in \cite{EbertReissigJHDE}. For this reason, in the language of the paper \cite{EbertReissigJHDE}, Example \ref{Example4} is non-effectively damped and Example \ref{Example5} with $d<-1$ is effectively damped.
\medskip

In this case the extended phase space $[0,\infty)\times \mathbb{R}^n$ is divided into zones as follows:
\begin{itemize}
\item hyperbolic zone:
\[ \Zhyp(\varepsilon) = \Big\{ (t,\xi)\in[0,\infty)\times\mathbb{R}^n : g(t)|\xi|\leq \varepsilon \Big\}, \]
\item reduced zone:
\[ \Zred(\varepsilon,N) = \Big\{ (t,\xi)\in[0,\infty)\times\mathbb{R}^n : \varepsilon \leq g(t)|\xi| \leq N  \Big\}, \]
\item elliptic zone:
\[ \Zell(N) = \Big\{ (t,\xi)\in[0,\infty)\times\mathbb{R}^n : g(t)|\xi|\geq N \Big\}, \]
\end{itemize}
where $\varepsilon>0$ is sufficiently small and $N>0$ is sufficiently large.
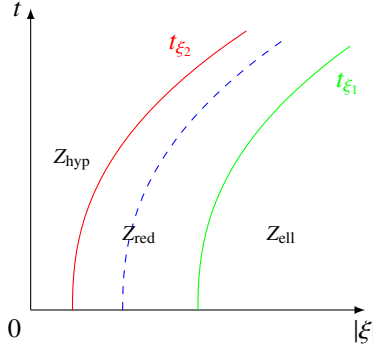
\begin{figure}[H]
\begin{center}
\begin{tikzpicture}[>=latex,xscale=1.1]
	\draw[->] (7,0) -- (11,0)node[below]{$|\xi|$};
	\draw[->] (7,0) -- (7,4)node[left]{$t$};
	\node[below left] at(7,0){$0$};
	\node  at (8.8,3.5) {$\textcolor{red}{t_{\xi_2}}$};
	\node  at (10.8,3) {$\textcolor{green}{t_{\xi_1}}$};
    \draw[dashed, domain=0:3.6,color=blue,variable=\t] plot ({8.1 + 0.09*pow(\t,2.4)},\t);
	\node[color=black] at (10, 1){{\footnotesize $Z_{\text{ell}}$}};
	\node[color=black] at (8.3, 1){{\footnotesize $Z_{\text{red}}$}};
	\draw[domain=0:3.5,color=green,variable=\t] plot ({9 + 0.09*pow(\t,2.4)},\t);
	\draw[domain=0:3.7,color=red,variable=\t] plot ({7.5 + 0.09*pow(\t,2.4)},\t);
	\node[color=black] at (7.5,2){{\footnotesize $Z_{\text{hyp}}$}};
\end{tikzpicture}
\caption{Sketch of the zones for the case $g\in L^1(0,\infty)$}
\label{fig.zone2}
\end{center}
\end{figure}
We denote the separating line between elliptic and reduced zone as $t_{\xi_1}$ and that between hyperbolic zone and reduced zone as $t_{\xi_2}$.
The blue dashed line denotes the separating line between the hyperbolic and the elliptic region.
\subsection{Sub-decaying integrable time-dependent coefficient $g=g(t)$} \label{SectSubDecreasing}
We assume that the coefficient $g=g(t)$ satisfies the following conditions:
\begin{enumerate}
\item[\textbf{(B1)}] $g(t)>0$ and $g'(t)<0$ for all $t \in [0,\infty)$,
\item[\textbf{(B2)}] $g \in L^1(0,\infty)$,
\item[\textbf{(B3)}] $|g'(t)|\leq C_1g(t)$ and $|g''(t)|\leq C_2g(t)$ for all $t \in [0,\infty)$ with positive constants $C_1$ and $C_2$.
\end{enumerate}

Examples for this case are $g(t)=(1+t)^{-d}$ with $d>1$, and $g(t)=e^{-Ct}$ with $C>0$.
\begin{theorem} \label{TheoremIntegrableCase}
Let us consider the Cauchy problem
\begin{equation*}
\begin{cases}
u_{tt}- \Delta u + g(t)(-\Delta)u_t=0, &(t,x) \in (0,\infty) \times \mathbb{R}^n, \\
u(0,x)= u_0(x),\quad u_t(0,x)= u_1(x), &x \in \mathbb{R}^n.
\end{cases}
\end{equation*}
We assume that the coefficient $g=g(t)$ satisfies the conditions \textbf{(B1)} to \textbf{(B3)} and $(u_0,u_1)\in \dot{H}^{|\beta|} \times \dot{H}^{|\beta|-2}$ with $|\beta|\geq 2$. Then, we have the following estimates for Sobolev solutions:
\begin{align*}
\big\||D|^{|\beta|} u(t,\cdot)\big\|_{L^2} & \lesssim \|u_0\|_{\dot{H}^{|\beta|}} + \|u_1\|_{\dot{H}^{|\beta|-2}}, \\
\big\||D|^{|\beta|-2} u_t(t,\cdot)\big\|_{L^2} & \lesssim \|u_0\|_{\dot{H}^{|\beta|}} + \|u_1\|_{\dot{H}^{|\beta|-2}}.
\end{align*}
\end{theorem}
\begin{remark}
Theorem \ref{TheoremIntegrableCase} implies that we do not have any parabolic effect.
\end{remark}

\begin{proof}[Proof of Theorem \ref{TheoremIntegrableCase}]
$ $
\subsubsection{Considerations in the hyperbolic zone $\Zhyp(\varepsilon)$} \label{Section3.1}
Let us turn to the equation \eqref{modellinearviscoelasticdampedFourier} in the following form:
\begin{equation} \label{EqIntegrableCasewithD}
D_t^2 \hat u - |\xi|^2 \hat u - ig(t)|\xi|^2D_t \hat u = 0.
\end{equation}
We define the micro-energy $U=\big( |\xi|\hat{u},D_t\hat{u} \big)^\text{T}$. Then, the equation \eqref{EqIntegrableCasewithD} leads to the system of first order
\begin{equation*}
D_tU=\underbrace{\left( \begin{array}{cc}
0 & |\xi| \\
|\xi| & ig(t)|\xi|^2
\end{array} \right)}_{A(t,\xi)}U.
\end{equation*}
\begin{proposition} \label{LemEstHypZoneIntegrable}
The fundamental solution $E_U=E_U(t,s,\xi)$ corresponding to this system of first order satisfies the following estimate in $\Zhyp(\varepsilon)$:
\begin{equation*}
\big( |E_U(t,s,\xi)| \big) \lesssim \exp\bigg( -\frac{3}{8}|\xi|^2\int_s^tg(\tau)d\tau \bigg)
\left( \begin{array}{cc}
1 & 1 \\
1 & 1
\end{array} \right)
\end{equation*}
with $(t,\xi),(s,\xi)\in \Zhyp(\varepsilon)$ and $s \leq t$.
\end{proposition}
\begin{proof}
\noindent\textbf{Step 1.} \emph{Diagonalization procedure} \medskip

Let us carry out the first step of diagonalization. For this reason we set
\[ M = \left( \begin{array}{cc}
1 & -1 \\
1 & 1
\end{array} \right) \qquad \text{and} \qquad  M^{-1} = \frac{1}{2}\left( \begin{array}{cc}
1 & 1 \\
-1 & 1
\end{array} \right). \]
We define $U^{(0)}:=M^{-1}U$. Then, we arrive at the system
\begin{equation*}
D_tU^{(0)}=\big( \mathcal{D}(\xi)+\mathcal{R}(t,\xi) \big)U^{(0)},
\end{equation*}
where
\begin{align*}
\mathcal{D}(\xi) = \left( \begin{array}{cc}
\tau_1 & 0 \\
0 & \tau_2
\end{array} \right) =\left( \begin{array}{cc}
-|\xi| & 0 \\
0 & |\xi|
\end{array} \right) \qquad \mbox{and} \qquad \mathcal{R}(t,\xi) = \frac{1}{2}\left( \begin{array}{cc}
ig(t)|\xi|^2 & -ig(t)|\xi|^2 \\
-ig(t)|\xi|^2 & ig(t)|\xi|^2
\end{array} \right).
\end{align*}
Let $F_0=F_0(t,\xi)$ be the diagonal part of $\mathcal{R}=\mathcal{R}(t,\xi)$. To carry out the second step of diagonalization procedure, we introduce
\begin{align*}
N^{(1)}(t,\xi) = \left( \begin{array}{cc}
0 & \dfrac{\mathcal{R}_{12}}{\tau_1-\tau_2} \\
\dfrac{\mathcal{R}_{21}}{\tau_2-\tau_1} & 0
\end{array} \right) = \frac{1}{4}\left( \begin{array}{cc}
0 & ig(t)|\xi| \\
-ig(t)|\xi| & 0
\end{array} \right),
\end{align*}
and $N_1(t,\xi)=I+N^{(1)}(t,\xi)$. For all $(t,\xi) \in \Zhyp(\varepsilon)$ the matrix $N_1=N_1(t,\xi)$ is invertible with uniformly bounded inverse $N_1^{-1}=N_1^{-1}(t,\xi)$. Indeed, in $\Zhyp(\varepsilon)$ it holds
\[  |N^{(1)}(t,\xi)| \leq \frac{g(t)|\xi|}{4}\leq \frac{\varepsilon}{4}. \]
We set
\begin{align*}
B^{(1)}(t,\xi) = D_tN^{(1)}(t,\xi)-\big( \mathcal{R}(t,\xi)-F_0(t,\xi) \big)N^{(1)}(t,\xi) = \frac{1}{8}\left( \begin{array}{cc}
g^2(t)|\xi|^3 & 2g'(t)|\xi| \\
-2g'(t)|\xi| & -g^2(t)|\xi|^3
\end{array} \right),
\end{align*}
and, then
\[ \mathcal{R}_1(t,\xi)=-N_1^{-1}(t,\xi)B^{(1)}(t,\xi). \]
Thus, we may conclude
\begin{equation*}
\big( D_t-\mathcal{D}(\xi)-\mathcal{R}(t,\xi) \big)N_1(t,\xi)=N_1(t,\xi)\big( D_t-\mathcal{D}(\xi)-F_0(t,\xi)-\mathcal{R}_1(t,\xi) \big).
\end{equation*}
\noindent\textbf{Step 2.} \emph{Construction of the fundamental solution} \medskip

We turn to  $U^{(1)}=U^{(1)}(t,\xi)$ as the solution to the system
\[ \big( D_t-\mathcal{D}(\xi)-F_0(t,\xi)-\mathcal{R}_1(t,\xi) \big)U^{(1)}(t,\xi) = 0. \]
We can write $U^{(1)}(t,\xi) = E_{U,1}(t,s,\xi)U^{(1)}(s,\xi)$. Here $E_{U,1}=E_{U,1}(t,s,\xi)$ is the fundamental solution to the system
\begin{align*}
\big( D_t-\mathcal{D}(\xi)-F_0(t,\xi)-\mathcal{R}_1(t,\xi) \big)E_{U,1}(t,s,\xi) = 0, \quad E_{U,1}(s,s,\xi) = I.
\end{align*}
The solution $E_0 = E_0(t,s,\xi)$ of the ``principal diagonal part'' satisfies
\begin{align*}
D_tE_0(t,s,\xi) = \big( \mathcal{D}(\xi)+F_0(t,\xi) \big)E_0(t,s,\xi), \quad E_{0}(s,s,\xi) = I,
\end{align*}
with $t\geq s$ and $(t,\xi), (s,\xi) \in \Zhyp(\varepsilon)$. Consequently, we have
\[ E_0(t,s,\xi) = \exp \bigg( i\int_s^t\big( \mathcal{D}(\xi)+F_0(\tau,\xi) \big)d\tau \bigg), \]
and we can estimate
\[ |E_0(t,s,\xi)| \lesssim \exp \bigg( -\frac{1}{2}|\xi|^2\int_s^tg(\tau)d\tau \bigg). \]
Let us set
\begin{align*}
\mathcal{R}_2(t,s,\xi) &= E_0^{-1}(t,s,\xi)\mathcal{R}_1(t,\xi)E_0(t,s,\xi), \\
Q(t,s,\xi) &= I + \sum_{k=1}^\infty i^k\int_s^t\mathcal{R}_2(t_1,s,\xi)\int_s^{t_1}\mathcal{R}_2(t_2,s,\xi)\cdots\int_s^{t_{k-1}}\mathcal{R}_2(t_k,s,\xi)dt_k\cdots dt_2dt_1.
\end{align*}
Then, $Q=Q(t,s,\xi)$ solves the Cauchy problem
\begin{align*}
D_tQ(t,s,\xi) = \mathcal{R}_2(t,s,\xi)Q(t,s,\xi), \quad Q(s,s,\xi) = I.
\end{align*}
The fundamental solution $E_{U,1}=E_{U,1}(t,s,\xi)$ is representable in the form $E_{U,1}(t,s,\xi) = E_0(t,s,\xi)Q(t,s,\xi)$. Furthermore, we see that
\[ |Q(t,s,\xi)|\leq \exp \bigg( \int_s^t|\mathcal{R}_1(\tau,\xi)|d\tau \bigg) \lesssim \exp \bigg( \frac{1}{8}\int_s^t\big( g^2(\tau)|\xi|^3 - 2g'(\tau)|\xi| \big)d\tau \bigg) \lesssim \exp \bigg( \frac{1}{8}\int_s^t g^2(\tau)|\xi|^3 d\tau \bigg) \]
by using the definition of $\Zhyp(\varepsilon)$. Therefore, we get
\begin{align*}
|E_{U,1}(t,s,\xi)| &\leq |E_0(t,s,\xi)|\,|Q(t,s,\xi)| \\
& \leq \exp \bigg( -\frac{1}{2}|\xi|^2\int_s^tg(\tau)d\tau+\frac{1}{8}\int_s^t g^2(\tau)|\xi|^3 d\tau \bigg) \\
& \leq \exp \bigg( -\frac{1}{2}|\xi|^2\int_s^tg(\tau)\left( 1-\frac{1}{4}g(\tau)|\xi| \right)d\tau \bigg) \\
& \leq \exp\bigg( -\frac{4-\varepsilon}{8}|\xi|^2\int_s^tg(\tau)d\tau \bigg),
\end{align*}
where we used the definition of $\Zhyp(\varepsilon)$. The backward transformation leads to the same estimate for $E_U=E_U(t,s,\xi)$ for all $(t,\xi), (s,\xi) \in \Zhyp(\varepsilon)$ and $t \geq s$. This completes the proof.
\end{proof}
\begin{corollary} \label{IntegrableCorZhyp}
We have the following estimates for $s\leq t$, $(s,\xi), (t,\xi) \in \Zhyp(\varepsilon)$:
\begin{align*}
|\xi|^{|\beta|}|\hat{u}(t,\xi)| & \lesssim \exp\bigg( -\frac{3}{8}|\xi|^2\int_s^tg(\tau)d\tau \bigg)\big( |\xi|^{|\beta|}|\hat{u}(s,\xi)| + |\xi|^{|\beta|-1}|\hat{u}_t(s,\xi)| \big) \quad \mbox{for} \quad |\beta|\geq 1, \\
|\xi|^{|\beta|}|\hat{u}_t(t,\xi)| & \lesssim \exp\bigg( -\frac{3}{8}|\xi|^2\int_s^tg(\tau)d\tau \bigg)\big( |\xi|^{|\beta|+1}|\hat{u}(s,\xi)| + |\xi|^{|\beta|}|\hat{u}_t(s,\xi)| \big) \quad \mbox{for} \quad |\beta|\geq 0.
\end{align*}
\end{corollary}
\subsubsection{Considerations in the reduced zone $\Zred(\varepsilon,N)$} \label{Section3.3}
We define the micro-energy $U=\big( |\xi|\hat{u},D_t\hat{u} \big)^\text{T}$. Then, the Cauchy problem \eqref{modellinearviscoelasticdampedFourier} leads to the system of first order
\begin{equation} \label{SystemPseudoIntegrable}
D_tU=\underbrace{\left( \begin{array}{cc}
0 & |\xi| \\
|\xi| & ig(t)|\xi|^2
\end{array} \right)}_{A(t,\xi)}U.
\end{equation}
We can define for all $(t,\xi) \in \Zred$ the following energy of the solutions to \eqref{SystemPseudoIntegrable}:
\begin{equation*}
\mathcal{E}(t,\xi) := \frac{1}{2}\big( |\xi|^2|\hat{u}(t,\xi)|^2 + |\hat{u}_t(t,\xi)|^2 \big).
\end{equation*}
If we differentiate the energy $\mathcal{E}=\mathcal{E}(t,\xi)$ with respect to $t$ and use our equation \eqref{modellinearviscoelasticdampedFourier}, it follows
\[ \frac{d}{dt}\mathcal{E}(t,\xi) = -g(t)|\xi|^2|\hat{u}_t(t,\xi)|^2 \leq 0. \]
Consequently, $\mathcal{E}=\mathcal{E}(t,\xi)$ is monotonically decreasing in $t$. Therefore, we have $\mathcal{E}(t,\xi)\leq \mathcal{E}(s,\xi)$ for $s\leq t$ and $(s,\xi), (t,\xi) \in \Zred(\varepsilon,N)$. This implies
\begin{align*}
|\xi||\hat{u}(t,\xi)| &\leq \sqrt{\mathcal{E}(s,\xi)} \leq |\xi||\hat{u}(s,\xi)| + |\hat{u}_t(s,\xi)|, \\
|\hat{u}_t(t,\xi)| &\leq \sqrt{\mathcal{E}(s,\xi)} \leq |\xi||\hat{u}(s,\xi)| + |\hat{u}_t(s,\xi)|.
\end{align*}
Taking into consideration, that the solution $U=U(t,\xi)$ is represented as $U(t,\xi)=E_{\text{red}}(t,s,\xi)U(s,\xi)$, we arrive at the following result.
\begin{corollary} \label{CorEstPDZoneIntegrable}
We have the following estimates in $\Zred(\varepsilon,N)$ with $(s,\xi), (t,\xi)\in\Zred(\varepsilon,N)$ and $s\leq t$ :
\begin{align*}
|\xi|^{|\beta|}|\hat{u}(t,\xi)| &\lesssim |\xi|^{|\beta|}|\hat{u}(s,\xi)| + |\xi|^{|\beta|-1}|\hat{u}_t(s,\xi)|\,\,\,\mbox{for}\,\,\,|\beta| \geq 1, \\
|\xi|^{|\beta|}|\hat{u}_t(t,\xi)| &\lesssim |\xi|^{|\beta|+1}|\hat{u}(s,\xi)| + |\xi|^{|\beta|}|\hat{u}_t(s,\xi)| \,\,\,\mbox{for}\,\,\,|\beta| \geq 0.
\end{align*}
\end{corollary}
\subsubsection{Considerations in the elliptic zone $\Zell(N)$} \label{Section3.2}

We write the equation (\ref{modellinearauxiliary}) in the following form:
\begin{equation} \label{modellinearauxiliary111}
D_t^2 v + \Big( \underbrace{\dfrac{g(t)^2}{4}|\xi|^4-|\xi|^2}_{=:d^2(t,\xi)} \Big)v + \underbrace{\frac{g'(t)}{2}|\xi|^2}_{=:m(t,\xi)}v=0.
\end{equation}
\begin{remark} \label{RemarkIntegrableCase}
We have the following inequalities:
\begin{align} \label{EqEstimatesDerivativeIntegrable}
d^2(t,\xi) \leq \frac{1}{4}g^2(t)|\xi|^4 \qquad \text{and} \qquad d^2(t,\xi)\geq \Big( \frac{1}{4}-\frac{1}{N^2} \Big)g^2(t)|\xi|^4
\end{align}
with $N$ sufficiently large. Therefore, we get $d(t,\xi)\approx g(t)|\xi|^2$. Moreover, it holds
\begin{align*}
|d_t(t,\xi)| =  \bigg|\frac{1}{2}\frac{g'(t)g(t)|\xi|^4}{\sqrt{\frac{g^2(t)}{4}|\xi|^4-|\xi|^2}} \bigg| \lesssim - \frac{g'(t)g(t)|\xi|^4}{g(t)|\xi|^2} = -Cg'(t)|\xi|^2.
\end{align*}
On the other hand, we have
\begin{align*}
d_t^2(t,\xi) &=  \frac{1}{2}\frac{\Big( g''(t)g(t)|\xi|^4+(g'(t))^2|\xi|^4 \Big)\sqrt{\frac{g^2(t)}{4}|\xi|^4-|\xi|^2}}{\frac{g^2(t)}{4}|\xi|^4-|\xi|^2} - \frac{1}{8}\frac{\big( g'(t)g(t)|\xi|^4 \big)^2}{\Big( \frac{g^2(t)}{4}|\xi|^4-|\xi|^2 \Big)\sqrt{\frac{g^2(t)}{4}|\xi|^4-|\xi|^2}}.
\end{align*}
Employing condition \textbf{(B3)} and estimates \eqref{EqEstimatesDerivativeIntegrable}, we arrive at
\begin{align*}
|d_t^2(t,\xi)| &\leq  \frac{1}{2}\frac{|g''(t)|g(t)|\xi|^4+|g'(t)|^2|\xi|^4}{d(t,\xi)} + \frac{1}{8}\frac{\big( |g'(t)|g(t)|\xi|^4 \big)^2}{d^3(t,\xi)} \leq \Bigg( \frac{C_2+C_1^2}{\Big( 1-\frac{4}{N^2} \Big)^{\frac{1}{2}}} + \frac{C_1^2}{\Big( 1-\frac{4}{N^2} \Big)^{\frac{3}{2}}}  \Bigg)g(t)|\xi|^2 \leq C_N g(t)|\xi|^2 .
\end{align*}
\end{remark}
It is reasonable to introduce the micro-energy
\[ V=V(t,\xi):=\big( d(t,\xi)v,D_t v \big)^{\text{T}} \qquad \mbox{with}  \qquad d(t,\xi) := \sqrt{\frac{g^2(t)}{4}|\xi|^4-|\xi|^2}. \]
Thus, we have to apply tools from elliptic WKB-analysis. Transformation to a system of first order from \eqref{modellinearauxiliary111} leads to
\begin{equation*}
D_tV=\left( \begin{array}{cc}
0 & d(t,\xi) \\
-d(t,\xi) & 0
\end{array} \right)V + \left( \begin{array}{cc}
\dfrac{D_td(t,\xi)}{d(t,\xi)} & 0 \\
-\dfrac{m(t,\xi)}{d(t,\xi)} & 0
\end{array} \right)V.
\end{equation*}
Using $V=MV^{(0)}$, $M=\begin{pmatrix} i & 1 \\ -i & 1\end{pmatrix}$, then after the first step of diagonalization we obtain
\[ D_tV^{(0)} = \big( \mathcal{D}(t,\xi) + \mathcal{R}(t,\xi) \big)V^{(0)}, \]
where
\begin{align*}
\mathcal{D}(t,\xi) = \left( \begin{array}{cc}
-id(t,\xi) & 0 \\
0 & id(t,\xi)
\end{array} \right) \qquad \text{and} \qquad \mathcal{R}(t,\xi) = \frac{1}{2} \left( \begin{array}{cc}
\dfrac{D_td(t,\xi)}{d(t,\xi)}-i\dfrac{m(t,\xi)}{d(t,\xi)} & -\dfrac{D_td(t,\xi)}{d(t,\xi)}+i\dfrac{m(t,\xi)}{d(t,\xi)} \\
-\dfrac{D_td(t,\xi)}{d(t,\xi)}-i\dfrac{m(t,\xi)}{d(t,\xi)} & \dfrac{D_td(t,\xi)}{d(t,\xi)}+i\dfrac{m(t,\xi)}{d(t,\xi)}
\end{array} \right).
\end{align*}
From Remark \ref{RemarkIntegrableCase} we find the estimates
\begin{align*}
\Big| \frac{D_td(t,\xi)}{d(t,\xi)} \Big| \lesssim -\frac{g'(t)}{g(t)} \qquad \mbox{and} \qquad \Big| \frac{m(t,\xi)}{d(t,\xi)} \Big| \lesssim -\frac{g'(t)}{g(t)}.
\end{align*}
Let us introduce $F_0(t,\xi)=\diag\mathcal{R}(t,\xi)$. Now we carry out the next step of diagonalization procedure. The difference of the diagonal entries of the matrix $\mathcal{D}(t,\xi)+F_0(t,\xi)$ is
\begin{equation*}
i\delta(t,\xi):=2d(t,\xi) + \frac{m(t,\xi)}{d(t,\xi)}\sim d(t,\xi)
\end{equation*}
for $t\leq t_{\xi_1}$ if we choose the zone constant $N$ sufficiently large and apply condition \textbf{(B3)}. Now we choose a matrix $N^{(1)}=N^{(1)}(t,\xi)$ such that
\[ N^{(1)}(t,\xi)=\left( \begin{array}{cc}
0 & -\dfrac{\mathcal{R}_{12}}{\delta(t,\xi)} \\
\dfrac{\mathcal{R}_{21}}{\delta(t,\xi)} & 0
\end{array} \right)\sim \left( \begin{array}{cc}
0 & i\dfrac{D_td(t,\xi)}{2d^2(t)}-\dfrac{m(t,\xi)}{2d^2(t,\xi)} \\
i\dfrac{D_td(t,\xi)}{2d^2(t,\xi)}+\dfrac{m(t,\xi)}{2d^2(t,\xi)} & 0
\end{array} \right). \]
We put $N_1=N_1(t,\xi):=I+N^{(1)}(t,\xi)$. For a sufficiently large zone constant $N$ and all $t\leq t_{\xi_1}$ the matrix $N_1=N_1(t,\xi)$ is invertible with uniformly bounded inverse $N_1^{-1}=N_1^{-1}(t,\xi)$. Indeed, in the elliptic zone $\Zell(N)$, for large $N$ it holds
\[ |N_1(t,\xi)-I| \leq \frac{1}{|\xi|^2g(t)}\frac{-g'(t)}{g(t)} \leq -\frac{1}{N^2}g'(t) \leq \frac{1}{2}. \]
Let
\begin{align*}
B^{(1)}(t,\xi) &= D_tN^{(1)}(t,\xi)-( \mathcal{R}(t,\xi)-F_0(t,\xi))N^{(1)}(t,\xi), \\
\mathcal{R}_1(t,\xi) &= -N_1^{-1}(t,\xi)B^{(1)}(t,\xi).
\end{align*}
Consequently, we have the following operator identity:
\begin{equation*}
\big( D_t-\mathcal{D}(t,\xi)-\mathcal{R}(t,\xi) \big)N_1(t,\xi)=N_1(t,\xi)\big( D_t-\mathcal{D}(t,\xi)-F_0(t,\xi)-\mathcal{R}_1(t,\xi) \big).
\end{equation*}
\noindent\textbf{Step 2.} \emph{Construction of the fundamental solution}
\begin{proposition} \label{Lem.Est.Ell.Zone.Integrable}
The fundamental solution $E_{\text{ell}}^{V}=E_{\text{ell}}^{V}(t,s,\xi)$ to the transformed operator
\[ D_t-\mathcal{D}(t,\xi)-F_0(t,\xi)-\mathcal{R}_1(t,\xi) \]
can be estimated by
\begin{equation*}
(|E_{\text{ell}}^{V}(t,s,\xi)|) \lesssim \frac{g(t)}{g(s)}\exp\bigg( \frac{|\xi|^2}{2}\int_{s}^{t}g(\tau)d\tau \bigg)
\left( \begin{array}{cc}
1 & 1 \\
1 & 1
\end{array} \right),
\end{equation*}
with $(t,\xi),(s,\xi)\in \Zell(N)$ and $s\leq t$.
\end{proposition}
\begin{proof}
We transform the system for $E_{\text{ell}}^{V}=E_{\text{ell}}^{V}(t,s,\xi)$ to an integral equation for a new matrix-valued function $\mathcal{Q}_{\text{ell}}=\mathcal{Q}_{\text{ell}}(t,s,\xi)$. Following the same idea as in the proof of Proposition \ref{Lem.Est.Ell.Zone}, we obtain that $E_{\text{ell}}^{V}=E_{\text{ell}}^{V}(t,s,\xi)$ satisfies the following integral equation:
\begin{align*}
E_{\text{ell}}^{V}(t,s,\xi) & = \exp\bigg\{ i\int_{s}^{t}\big( \mathcal{D}(\tau,\xi)+F_0(\tau,\xi) \big)d\tau \bigg\}E_{\text{ell}}^{V}(s,s,\xi)\\
& \quad + i\int_{s}^{t} \exp \bigg\{ i\int_{\theta}^{t}\big( \mathcal{D}(\tau,\xi)+F_0(\tau,\xi) \big)d\tau \bigg\}\mathcal{R}_1(\theta,\xi)E_{\text{ell}}^{V}(\theta,s,\xi)\,d\theta.
\end{align*}
We define
\[ \mathcal{Q}_{\text{ell}}(t,s,\xi):=\exp\bigg\{ -\int_{s}^{t}\beta(\tau,\xi)d\tau \bigg\} E_{\text{ell}}^{V}(t,s,\xi), \]
where $\beta=\beta(t,\xi)$ is chosen from the main entries of the diagonal matrix $i\mathcal{D}(t,\xi)+iF_0(t,\xi)$ as follows:
\[ \beta(t,\xi)=d(t,\xi)+\frac{d_t(t,\xi)}{2d(t,\xi)}+\frac{m(t,\xi)}{2d(t,\xi)}. \]
It satisfies the new integral equation
\begin{align*}
\mathcal{Q}_{\text{ell}}(t,s,\xi)=&\exp \bigg\{ \int_{s}^{t}\big( i\mathcal{D}(\tau,\xi)+iF_0(\tau,\xi)-\beta(\tau,\xi)I \big)d\tau \bigg\}\\
& \quad + \int_{s}^{t} \exp \bigg\{ \int_{\theta}^{t}\big( i\mathcal{D}(\tau,\xi)+iF_0(\tau,\xi)-\beta(\tau,\xi)I \big)d\tau \bigg\}\mathcal{R}_1(\theta,\xi)\mathcal{Q}_{\text{ell}}(\theta,s,\xi)\,d\theta.
\end{align*}
Using our conditions \textbf{(B1)} to \textbf{(B3)} and Remark \ref{RemarkIntegrableCase}, one may see that  $\mathcal{R}_1=\mathcal{R}_1(\theta,\xi)$ is uniformly integrable over the elliptic zone. It follows
\begin{align*}
H(t,s,\xi) & =\exp \bigg\{ \int_{s}^{t}\big( i\mathcal{D}(\tau,\xi)+iF_0(\tau,\xi)-\beta(\tau,\xi)I \big)d\tau \bigg\}\\
& = \diag \bigg( 1, \exp \bigg\{ \int_{s}^{t}\bigg( -2d(\tau,\xi)-\frac{m(\tau,\xi)}{d(\tau,\xi)} \bigg)d\tau \bigg\} \bigg)\rightarrow \left( \begin{array}{cc}
1 & 0 \\
0 & 0
\end{array} \right)
\end{align*}
as $t\rightarrow \infty$. Hence, the matrix $H=H(t,s,\xi)$ is uniformly bounded for $(s,\xi),(t,\xi)\in \Zell(N)$. So, the representation of $\mathcal{Q}_{\text{ell}}=\mathcal{Q}_{\text{ell}}(t,s,\xi)$ by a Neumann series gives
\begin{align*}
\mathcal{Q}_{\text{ell}}(t,s,\xi)=H(t,s,\xi)+&\sum_{k=1}^{\infty}i^k\int_{s}^{t}H(t,t_1,\xi)\mathcal{R}_1(t_1,\xi)\int_{s}^{t_1}H(t_1,t_2,\xi)\mathcal{R}_1(t_2,\xi)\cdots \int_{s}^{t_{k-1}}H(t_{k-1},t_k,\xi)\mathcal{R}_1(t_k,\xi)dt_k\cdots dt_2dt_1.
\end{align*}
Then, convergence of this series is obtained from the symbol estimates, since $\mathcal{R}_1=\mathcal{R}_1(t,\xi)$ is uniformly integrable over $\Zell(N)$. Hence, from the last considerations we may conclude
\begin{align*}
E_{\text{ell}}^{V}(t,s,\xi)&=\exp \bigg\{ \int_{s}^{t}\beta(\tau,\xi)d\tau \bigg\}\mathcal{Q}_{\text{ell}}(t,s,\xi)\\
& = \exp \bigg\{ \int_{s}^{t}\bigg( d(\tau,\xi)+\frac{\partial_\tau d(\tau,\xi)}{2d(\tau,\xi)}+\frac{m(\tau,\xi)}{2d(\tau,\xi)} \bigg)d\tau \bigg\}\mathcal{Q}_{\text{ell}}(t,s,\xi) \\
& \leq \frac{d(t,\xi)}{d(s,\xi)} \exp\bigg( \int_{s}^{t} d(\tau,\xi)\,d\tau \bigg)\mathcal{Q}_{\text{ell}}(t,s,\xi),
\end{align*}
where we used $m(t,\xi)\leq \partial_td(t,\xi)$ and $\mathcal{Q}_{\text{ell}}=\mathcal{Q}_{\text{ell}}(t,s,\xi)$ is a uniformly bounded matrix. Then, it follows
\begin{align*}
(|E_{\text{ell}}^{V}(t,s,\xi)|) & \lesssim \frac{g(t)}{g(s)} \exp\bigg( |\xi|^2\int_{s}^{t} \frac{g(\tau)}{2}d\tau \bigg)\left( \begin{array}{cc}
1 & 1 \\
1 & 1
\end{array} \right).
\end{align*}
This completes the proof.	
\end{proof}

Using the backward transformation we arrive at the following result.
\begin{corollary} \label{IntegrableCorZell}
In $\Zell(N)$ we have the following estimates for $(s,\xi), (t,\xi) \in \Zell(N)$ and $0 \leq s \leq t$:
\begin{align*}
\frac{g(t)}{2}|\xi|^{|\beta|}|\hat{u}(t,\xi)| & \lesssim \frac{g(t)}{g(s)} \Big( g(s)|\xi|^{|\beta|}|\hat{u}(s,\xi)| + |\xi|^{|\beta|-2}|\hat{u}_t(s,\xi)| \Big)\quad \mbox{for} \quad |\beta| \geq 2, \\
|\xi|^{|\beta|}|\hat{u}_t(t,\xi)| & \lesssim \frac{g(t)}{g(s)} \Big( g(s)|\xi|^{|\beta|+2}|\hat{u}(s,\xi)| + |\xi|^{|\beta|}|\hat{u}_t(s,\xi)| \Big) \quad \mbox{for} \quad |\beta| \geq 0.
\end{align*}
\end{corollary}
\subsubsection{Conclusions} \label{Section3.4}
From the statements of Corollaries \ref{IntegrableCorZhyp}, \ref{CorEstPDZoneIntegrable} and \ref{IntegrableCorZell} we derive our desired statements.

\noindent \textit{Case 1:} $t\leq t_{\xi_1}.$ Due to Corollary \ref{IntegrableCorZell} we have
\begin{align*}
\frac{g(t)}{2}|\xi|^{|\beta|}|\hat{u}(t,\xi)| & \lesssim g(t)\Big(|\xi|^{|\beta|}|\hat{u}_0(\xi)| + |\xi|^{|\beta|-2}|\hat{u}_1(\xi)| \Big), \\
|\xi|^{|\beta|}|\hat{u}_t(t,\xi)| &\lesssim g(t)\Big(|\xi|^{|\beta|+2}|\hat{u}_0(\xi)| + |\xi|^{|\beta|}|\hat{u}_1(\xi)| \Big).
\end{align*}
\noindent \textit{Case 2:} $t_{\xi_1}\leq t\leq t_{\xi_2}.$ In this case we apply Corollaries \ref{CorEstPDZoneIntegrable} and \ref{IntegrableCorZell} to get
\begin{align*}
|\xi|^{|\beta|}|\hat{u}(t,\xi)| & \leq |\xi|^{|\beta|}|\hat{u}(t_{\xi_1},\xi)| + |\xi|^{|\beta|-1}|\hat{u}_t(t_{\xi_1},\xi)|
\lesssim |\xi|^{|\beta|}|\hat{u}_0(\xi)| + |\xi|^{|\beta|-2}|\hat{u}_1(\xi)|, \\
|\xi|^{|\beta|}|\hat{u}_t(t,\xi)| &\leq |\xi|^{|\beta|+1}|\hat{u}(t_{\xi_1},\xi)| + |\xi|^{|\beta|}|\hat{u}_t(t_{\xi_1},\xi)|
\lesssim |\xi|^{|\beta|+1}|\hat{u}_0(\xi)| + |\xi|^{|\beta|-1}|\hat{u}_1(\xi)|.
\end{align*}
\noindent \textit{Case 3:} $t\geq t_{\xi_2}$. In this case we use Corollaries \ref{IntegrableCorZhyp}, \ref{CorEstPDZoneIntegrable} and \ref{IntegrableCorZell}. Then, it holds
\begin{align*}
|\xi|^{|\beta|}|\hat{u}(t,\xi)| & \leq \exp\bigg( -\frac{3}{8}|\xi|^2\int_{t_{\xi_2}}^tg(\tau)d\tau \bigg)\Big( |\xi|^{|\beta|}|\hat{u}(t_{\xi_2},\xi)| + |\xi|^{|\beta|-1}|\hat{u}_t(t_{\xi_2},\xi)| \Big) \\
& \leq\exp \bigg( -\frac{3}{8}|\xi|^2\int_{t_{\xi_2}}^tg(\tau)d\tau \bigg)\Big( |\xi|^{|\beta|}|\hat{u}(t_{\xi_1},\xi)| + |\xi|^{|\beta|-1}|\hat{u}_t(t_{\xi_1},\xi)| \Big) \\
& \lesssim \exp \bigg( -\frac{3}{8}|\xi|^2\int_{t_{\xi_2}}^tg(\tau)d\tau \bigg)\Big( |\xi|^{|\beta|}|\hat{u}_0(\xi)| + |\xi|^{|\beta|-2}|\hat{u}_1(\xi)| \Big), \\
|\xi|^{|\beta|}|\hat{u}_t(t,\xi)| &\leq \exp \bigg( -\frac{3}{8}|\xi|^2\int_{t_{\xi_2}}^tg(\tau)d\tau \bigg)\Big( |\xi|^{|\beta|+1}|\hat{u}(t_{\xi_2},\xi)| + |\xi|^{|\beta|}|\hat{u}_t(t_{\xi_2},\xi)| \Big) \\
& \lesssim \exp \bigg( -\frac{3}{8}|\xi|^2\int_{t_{\xi_2}}^tg(\tau)d\tau \bigg)\Big( |\xi|^{|\beta|+1}|\hat{u}(t_{\xi_1},\xi)| + |\xi|^{|\beta|}|\hat{u}_t(t_{\xi_1},\xi)| \Big) \\
& \lesssim \exp \bigg( -\frac{3}{8}|\xi|^2\int_{t_{\xi_2}}^tg(\tau)d\tau \bigg)\Big( |\xi|^{|\beta|+1}|\hat{u}_0(\xi)| + |\xi|^{|\beta|-1}|\hat{u}_1(\xi)| \Big).
\end{align*}
Consequently, the proof of Theorem \ref{TheoremIntegrableCase} is completed.
\end{proof}

\subsection{Super-decaying integrable time-dependent coefficient $g=g(t)$}
We assume that the coefficient $g=g(t)$ satisfies the following conditions:
\begin{enumerate}
\item[\textbf{(C1)}] $g(t)>0$ and $g'(t)<0$ for all $t \in [0,\infty)$,
\item[\textbf{(C2)}] $g \in L^1(0,\infty)$.
\end{enumerate}
Examples for this case are $g(t)=e^{-e^{t}}$ and $g(t)=e^{-e^{e^{t}}}$.
\begin{theorem} \label{TheoremIntegrableCaseSuper}
Let us consider the Cauchy problem
\begin{equation*}
\begin{cases}
u_{tt}- \Delta u + g(t)(-\Delta)u_t=0, &(t,x) \in (0,\infty) \times \mathbb{R}^n, \\
u(0,x)= u_0(x),\quad u_t(0,x)= u_1(x), &x \in \mathbb{R}^n.
\end{cases}
\end{equation*}
We assume that the coefficient $g=g(t)$ satisfies the conditions \textbf{(C1)} and \textbf{(C2)}, and $(u_0,u_1)\in \dot{H}^{|\beta|+\kappa+1} \times \dot{H}^{|\beta|+\kappa-1}$ with $|\beta|\geq 1$ and arbitrarily small $\kappa>0$. Then, we have the following estimates for Sobolev solutions:
\begin{align*}
\big\||D|^{|\beta|} u(t,\cdot)\big\|_{L^2} & \lesssim \|u_0\|_{\dot{H}^{|\beta|+\kappa+1}} + \|u_1\|_{\dot{H}^{|\beta|+\kappa-1}}, \\
\big\||D|^{|\beta|-1} u_t(t,\cdot)\big\|_{L^2} & \lesssim \|u_0\|_{\dot{H}^{|\beta|+\kappa+1}} + \|u_1\|_{\dot{H}^{|\beta|+\kappa-1}}.
\end{align*}
\end{theorem}
\begin{remark}
Theorem \ref{TheoremIntegrableCaseSuper} implies that we do not have any parabolic effect.
\end{remark}
\begin{proof}[Proof of Theorem \ref{TheoremIntegrableCaseSuper}]
$ $
\subsubsection{Considerations in the hyperbolic zone $\Zhyp(\varepsilon)$} \label{Section3.1Super}
We have the same statement from Proposition \ref{LemEstHypZoneIntegrable} in Subsection \ref{SectSubDecreasing}.
\begin{corollary} \label{IntegrableCorZhypSuper}
We have the following estimates for $s\leq t$, $(s,\xi), (t,\xi) \in \Zhyp(\varepsilon)$:
\begin{align*}
|\xi|^{|\beta|}|\hat{u}(t,\xi)| & \lesssim \exp\bigg( -\frac{3}{8}|\xi|^2\int_s^tg(\tau)d\tau \bigg)\Big( |\xi|^{|\beta|}|\hat{u}(s,\xi)| + |\xi|^{|\beta|-1}|\hat{u}_t(s,\xi)| \Big) \quad \mbox{for} \quad |\beta|\geq 1, \\
|\xi|^{|\beta|}|\hat{u}_t(t,\xi)| & \lesssim \exp\bigg( -\frac{3}{8}|\xi|^2\int_s^tg(\tau)d\tau \bigg)\Big( |\xi|^{|\beta|+1}|\hat{u}(s,\xi)| + |\xi|^{|\beta|}|\hat{u}_t(s,\xi)| \Big) \quad \mbox{for} \quad |\beta|\geq 0.
\end{align*}
\end{corollary}
\subsubsection{Considerations in the reduced zone $\Zred(\varepsilon,N)$} \label{Section3.3Super}
We have the same estimates from Corollary \ref{CorEstPDZoneIntegrable} in Subsection \ref{SectSubDecreasing}.
\begin{corollary} \label{CorEstPDZoneIntegrableSuper}
The following estimates hold in the reduced zone $\Zred(\varepsilon,N)$ with $(s,\xi), (t,\xi) \in \Zred(\varepsilon,N)$ and $s \leq t$:
\begin{align*}
|\xi|^{|\beta|}|\hat{u}(t,\xi)| &\lesssim |\xi|^{|\beta|}|\hat{u}(s,\xi)| + |\xi|^{|\beta|-1}|\hat{u}_t(s,\xi)|\,\,\,\mbox{for}\,\,\,|\beta| \geq 1, \\
|\xi|^{|\beta|}|\hat{u}_t(t,\xi)| &\lesssim |\xi|^{|\beta|+1}|\hat{u}(s,\xi)| + |\xi|^{|\beta|}|\hat{u}_t(s,\xi)| \,\,\,\mbox{for}\,\,\,|\beta| \geq 0.
\end{align*}
\end{corollary}
\subsubsection{Considerations in the elliptic zone $\Zell(N)$} \label{Section3.2Super}
Let us write the equation (\ref{modellinearauxiliary}) in the following form:
\begin{equation*} \label{modellinearauxiliary111Super}
D_t^2 v + \bigg( \underbrace{\dfrac{g(t)^2}{4}|\xi|^4-|\xi|^2}_{=:d^2(t,\xi)} \bigg)v + \underbrace{\frac{g'(t)}{2}|\xi|^2}_{=:m(t,\xi)}v=0.
\end{equation*}
\begin{remark} \label{RemarkIntegrableCaseSuper}
We have the following inequalities with sufficiently large $N$:
\begin{align*} \label{EqEstimatesDerivativeIntegrableSuper}
d^2(t,\xi) \leq \frac{1}{4}g^2(t)|\xi|^4 \qquad \text{and} \qquad d^2(t,\xi)\geq \Big( \frac{1}{4}-\frac{1}{N^2} \Big)g^2(t)|\xi|^4.
\end{align*}
Therefore, we get $d(t,\xi)\approx g(t)|\xi|^2$. Furthermore, it holds
\begin{align*}
|d_t(t,\xi)| =  \bigg| \frac{1}{4}\frac{g'(t)g(t)|\xi|^4}{\sqrt{\frac{g^2(t)}{4}|\xi|^4-|\xi|^2}} \bigg| \leq  -\frac{1}{2\sqrt{1-\frac{4}{N^2}}}g'(t)|\xi|^2.
\end{align*}
Finally, we have
\begin{align*}
-\frac{g'(t)}{g(t)}d(t,\xi) \leq -m(t,\xi) \leq -\frac{1}{\sqrt{1-\frac{4}{N^2}}}\frac{g'(t)}{g(t)}d(t,\xi).
\end{align*}
\end{remark}
We introduce the micro-energy
\[ V=V(t,\xi):=\big( d(t,\xi)v,D_t v \big)^{\text{T}} \qquad \mbox{with} \qquad d(t,\xi) := \sqrt{\frac{g^2(t)}{4}|\xi|^4-|\xi|^2}. \]
Thus, we have to apply tools from elliptic WKB-analysis as we did in Subsection \ref{SectSubDecreasing} with only one step of diagonalization procedure. Transforming \eqref{modellinearauxiliary111} to a system of first order for $V=V(t,\xi)$  gives
\begin{equation*}
D_tV=\left( \begin{array}{cc}
0 & d(t,\xi) \\
-d(t,\xi) & 0
\end{array} \right)V + \left( \begin{array}{cc}
\dfrac{D_td(t,\xi)}{d(t,\xi)} & 0 \\
-\dfrac{m(t,\xi)}{d(t,\xi)} & 0
\end{array} \right)V.
\end{equation*}
Using $V=MV^{(0)}$, $M=\begin{pmatrix} i & 1 \\ -i & 1\end{pmatrix}$, then after the first step of diagonalization we obtain
\[ D_tV^{(0)} = \big( \mathcal{D}(t,\xi) + \mathcal{R}(t,\xi) \big)V^{(0)}, \]
where
\begin{align*}
\mathcal{D}(t,\xi) &= \left( \begin{array}{cc}
-id(t,\xi) & 0 \\
0 & id(t,\xi)
\end{array} \right) \qquad \mbox{and} \qquad \mathcal{R}(t,\xi) = \frac{1}{2} \left( \begin{array}{cc}
\dfrac{D_td(t,\xi)}{d(t,\xi)}-i\dfrac{m(t,\xi)}{d(t,\xi)} & -\dfrac{D_td(t,\xi)}{d(t,\xi)}+i\dfrac{m(t,\xi)}{d(t,\xi)} \\
-\dfrac{D_td(t,\xi)}{d(t,\xi)}-i\dfrac{m(t,\xi)}{d(t,\xi)} & \dfrac{D_td(t,\xi)}{d(t,\xi)}+i\dfrac{m(t,\xi)}{d(t,\xi)}
\end{array} \right).
\end{align*}
From Remark \ref{RemarkIntegrableCaseSuper} we find the estimates
\begin{align} \label{EstimateRemaindersSuper}
\Big|\frac{D_td(t,\xi)}{d(t,\xi)}\Big| \leq -\frac{1}{\sqrt{1-\frac{4}{N^2}}}\frac{g'(t)}{g(t)} \qquad \mbox{and} \qquad -\frac{g'(t)}{g(t)} \leq -\frac{m(t,\xi)}{d(t,\xi)} \leq -\frac{1}{\sqrt{1-\frac{4}{N^2}}}\frac{g'(t)}{g(t)}.
\end{align}
Let us introduce $F_0(t,\xi):=\diag\mathcal{R}(t,\xi)$ and $\mathcal{R}_1(t,\xi):=\antidiag\mathcal{R}(t,\xi)$.\medskip

\noindent\textbf{Step 2.} \emph{Construction of the fundamental solution}
\begin{proposition} \label{Lem.Est.Ell.Zone.IntegrableSuper}
The fundamental solution $E_{\text{ell}}^{V}=E_{\text{ell}}^{V}(t,s,\xi)$ to the transformed operator
\[ D_t-\mathcal{D}(t,\xi)-F_0(t,\xi)-\mathcal{R}_1(t,\xi) \]
can be estimated by
\begin{equation*}
(|E_{\text{ell}}^{V}(t,s,\xi)|) \lesssim \Big( \frac{g(s)}{g(t)} \Big)^{\frac{1}{\sqrt{1-\frac{4}{N^2}}}-1}\exp\bigg( \frac{|\xi|^2}{2}\int_{s}^{t}g(\tau)d\tau \bigg)
\left( \begin{array}{cc}
1 & 1 \\
1 & 1
\end{array} \right),
\end{equation*}
with $s \leq t$ and $(t,\xi),(s,\xi)\in \Zell(N)$.
\end{proposition}
\begin{proof}
We transform the system for $E_{\text{ell}}^{V}=E_{\text{ell}}^{V}(t,s,\xi)$ to an integral equation for a new matrix-valued function $\mathcal{Q}_{\text{ell}}=\mathcal{Q}_{\text{ell}}(t,s,\xi)$ as in the proof of Proposition \ref{Lem.Est.Ell.Zone}. We define
\[ \mathcal{Q}_{\text{ell}}(t,s,\xi):=\exp\bigg\{ -\int_{s}^{t}\beta(\tau,\xi)d\tau \bigg\} E_{\text{ell}}^{V}(t,s,\xi), \]
where $\beta=\beta(t,\xi)$ is chosen from the main entries of the diagonal matrix $i\mathcal{D}(t,\xi)+iF_0(t,\xi)$ as follows:
\[ \beta(t,\xi)=d(t,\xi)+\frac{d_t(t,\xi)}{2d(t,\xi)}+\frac{m(t,\xi)}{2d(t,\xi)}. \]
The standard construction of $\mathcal{Q}_{\text{ell}}=\mathcal{Q}_{\text{ell}}(t,s,\xi)$ in terms of a Peano-Baker series implies
\begin{align} \label{EstimateQSuperExponential}
|\mathcal{Q}_{\text{ell}}(t,s,\xi)| \leq \exp\bigg\{ \int_s^t|\mathcal{R}_1(\tau,\xi)|d\tau  \bigg\} & \leq \exp\bigg\{ \int_s^t \bigg( \Big|\frac{D_\tau d(\tau,\xi)}{2d(\tau,\xi)} \Big|+\Big|\frac{m(\tau,\xi)}{2d(\tau,\xi)}\Big| \bigg)d\tau \bigg\} \nonumber \\
& \leq \exp\bigg\{ \int_s^t \bigg( -\frac{1}{\sqrt{1-\frac{4}{N^2}}}\frac{g'(\tau)}{g(\tau)} \bigg)d\tau \bigg\}\leq \Big(\frac{g(s)}{g(t)}\Big)^{\frac{1}{\sqrt{1-\frac{4}{N^2}}}},
\end{align}
where we used estimates in \eqref{EstimateRemaindersSuper}. Hence, from the last considerations we may conclude
\begin{align*}
E_{\text{ell}}^{V}(t,s,\xi)&=\exp \bigg\{ \int_{s}^{t}\beta(\tau,\xi)d\tau \bigg\}\mathcal{Q}_{\text{ell}}(t,s,\xi) \\
& = \exp \bigg\{ \int_{s}^{t}\bigg( d(\tau,\xi)+\frac{\partial_\tau d(\tau,\xi)}{2d(\tau,\xi)}+\frac{m(\tau,\xi)}{2d(\tau,\xi)} \bigg)d\tau \bigg\}\mathcal{Q}_{\text{ell}}(t,s,\xi) \\
& \leq  \exp \bigg\{ \int_{s}^{t}\bigg( d(\tau,\xi)+\frac{\partial_\tau d(\tau,\xi)}{2d(\tau,\xi)}+ \frac{g'(\tau)}{2g(\tau)} \bigg)d\tau \bigg\}\mathcal{Q}_{\text{ell}}(t,s,\xi) \\
& \leq \sqrt{\frac{d(t,\xi)g(t)}{d(s,\xi)g(s)}} \exp \bigg( \int_{s}^{t}d(\tau,\xi)d\tau \bigg)\mathcal{Q}_{\text{ell}}(t,s,\xi).
\end{align*}
Then, using the estimate of $\mathcal{Q}_{\text{ell}}=\mathcal{Q}_{\text{ell}}(t,s,\xi)$ from \eqref{EstimateQSuperExponential}, it follows
\begin{align*}
(|E_{\text{ell}}^{V}(t,s,\xi)|) & \lesssim \frac{g(t)}{g(s)}\exp \bigg( |\xi|^2\int_{s}^{t} \frac{g(\tau)}{2}d\tau \bigg) \left( \begin{array}{cc}
1 & 1 \\
1 & 1
\end{array} \right)|\mathcal{Q}_{\text{ell}}(t,s,\xi)| \\
& \lesssim \Big( \frac{g(s)}{g(t)} \Big)^{\frac{1}{\sqrt{1-\frac{4}{N^2}}}-1}\exp \bigg( |\xi|^2\int_{s}^{t} \frac{g(\tau)}{2}d\tau \bigg)\left( \begin{array}{cc}
1 & 1 \\
1 & 1
\end{array} \right).
\end{align*}
This completes the proof.
\end{proof}
Using the backward transformation we arrive at the following result.
\begin{corollary} \label{IntegrableCorZellSuper}
In $\Zell(N)$ we have the following estimates for $(s,\xi), (t,\xi) \in \Zell(N)$ and $0 \leq s \leq t$:
\begin{align*}
\frac{g(t)}{2}|\xi|^{|\beta|}|\hat{u}(t,\xi)| & \lesssim \Big( \frac{g(s)}{g(t)} \Big)^\kappa \Big( g(s)|\xi|^{|\beta|}|\hat{u}(s,\xi)| + |\xi|^{|\beta|-2}|\hat{u}_t(s,\xi)| \Big) \quad \mbox{for} \quad |\beta| \geq 2, \\
|\xi|^{|\beta|}|\hat{u}_t(t,\xi)| & \lesssim \Big( \frac{g(s)}{g(t)} \Big)^\kappa \Big( g(s)|\xi|^{|\beta|+2}|\hat{u}(s,\xi)| + |\xi|^{|\beta|}|\hat{u}_t(s,\xi)| \Big) \quad \mbox{for} \quad |\beta| \geq 0,
\end{align*}
where $\kappa=\frac{1}{\sqrt{1-\frac{4}{N^2}}}-1$ is an arbitrarily small exponent for arbitrarily large $N=N(\kappa)$ for all $\kappa>0$.
\end{corollary}
\subsubsection{Conclusion} \label{Section3.4Super}
From the statements of Corollaries \ref{IntegrableCorZhypSuper}, \ref{CorEstPDZoneIntegrableSuper} and \ref{IntegrableCorZellSuper} we derive our desired statements.

\noindent \textit{Case 1:} $t\leq t_{\xi_1}.$ Due to Corollary \ref{IntegrableCorZellSuper} and $g(t_{\xi_1})|\xi|=N$, we have
\begin{align*}
|\xi|^{|\beta|}|\hat{u}(t,\xi)| & \lesssim \frac{1}{g(t)^{\kappa+1}}\Big( |\xi|^{|\beta|}|\hat{u}_0(\xi)| + |\xi|^{|\beta|-2}|\hat{u}_1(\xi)| \Big) \lesssim |\xi|^{|\beta|+\kappa+1}|\hat{u}_0(\xi)| + |\xi|^{|\beta|+\kappa-1}|\hat{u}_1(\xi)|, \\
|\xi|^{|\beta|}|\hat{u}_t(t,\xi)| &\lesssim \frac{1}{g(t)^{\kappa}}\Big( |\xi|^{|\beta|+2}|\hat{u}_0(\xi)| + |\xi|^{|\beta|}|\hat{u}_1(\xi)| \Big) \lesssim |\xi|^{|\beta|+\kappa+2}|\hat{u}_0(\xi)| + |\xi|^{|\beta|+\kappa}|\hat{u}_1(\xi)|.
\end{align*}
\noindent \textit{Case 2:} $t_{\xi_1}\leq t\leq t_{\xi_2}.$ In this case we apply Corollary \ref{CorEstPDZoneIntegrableSuper} and use $g(t_{\xi_1})|\xi|=N$. Then, we get
\begin{align*}
|\xi|^{|\beta|}|\hat{u}(t,\xi)| & \leq |\xi|^{|\beta|}|\hat{u}(t_{\xi_1},\xi)| + |\xi|^{|\beta|-1}|\hat{u}_t(t_{\xi_1},\xi)| \lesssim |\xi|^{|\beta|+\kappa+1}|\hat{u}_0(\xi)| + |\xi|^{|\beta|+\kappa-1}|\hat{u}_1(\xi)|, \\
|\xi|^{|\beta|}|\hat{u}_t(t,\xi)| &\leq |\xi|^{|\beta|+1}|\hat{u}(t_{\xi_1},\xi)| + |\xi|^{|\beta|}|\hat{u}_t(t_{\xi_1},\xi)| \lesssim |\xi|^{|\beta|+\kappa+2}|\hat{u}_0(\xi)| + |\xi|^{|\beta|+\kappa}|\hat{u}_1(\xi)|.
\end{align*}
\noindent \textit{Case 3:} $t\geq t_{\xi_2}$. In this case we apply Corollary \ref{IntegrableCorZhypSuper}. It holds
\begin{align*}
|\xi|^{|\beta|}|\hat{u}(t,\xi)| & \leq \exp\bigg( -\frac{3}{8}|\xi|^2\int_{t_{\xi_2}}^tg(\tau)d\tau \bigg)\Big( |\xi|^{|\beta|}|\hat{u}(t_{\xi_2},\xi)| + |\xi|^{|\beta|-1}|\hat{u}_t(t_{\xi_2},\xi)| \Big) \\
& \leq\exp\bigg( -\frac{3}{8}|\xi|^2\int_{t_{\xi_2}}^tg(\tau)d\tau \bigg)\Big( |\xi|^{|\beta|}|\hat{u}(t_{\xi_1},\xi)| + |\xi|^{|\beta|-1}|\hat{u}_t(t_{\xi_1},\xi)| \Big) \\
& \leq \exp\bigg( -\frac{3}{8}|\xi|^2\int_{t_{\xi_2}}^tg(\tau)d\tau \bigg)\Big( |\xi|^{|\beta|+\kappa+1}|\hat{u}_0(\xi)| + |\xi|^{|\beta|+\kappa-1}|\hat{u}_1(\xi)| \Big), \\
|\xi|^{|\beta|}|\hat{u}_t(t,\xi)| &\leq \exp\bigg( -\frac{3}{8}|\xi|^2\int_{t_{\xi_2}}^tg(\tau)d\tau \bigg)\Big( |\xi|^{|\beta|+1}|\hat{u}(t_{\xi_2},\xi)| + |\xi|^{|\beta|}|\hat{u}_t(t_{\xi_2},\xi)| \Big) \\
& \lesssim \exp\bigg( -\frac{3}{8}|\xi|^2\int_{t_{\xi_2}}^tg(\tau)d\tau \bigg)\Big( |\xi|^{|\beta|+1}|\hat{u}(t_{\xi_1},\xi)| + |\xi|^{|\beta|}|\hat{u}_t(t_{\xi_1},\xi)| \Big) \\
& \lesssim \exp\bigg( -\frac{3}{8}|\xi|^2\int_{t_{\xi_2}}^tg(\tau)d\tau \bigg)\Big( |\xi|^{|\beta|+\kappa+2}|\hat{u}_0(\xi)| + |\xi|^{|\beta|+\kappa}|\hat{u}_1(\xi)| \Big).
\end{align*}
Thus, the proof of Theorem \ref{TheoremIntegrableCaseSuper} is completed.
\end{proof}


\section{Models with non-integrable and decreasing time-dependent coefficient $g=g(t)$} \label{Section4.3}
We assume that function $g=g(t)$ satisfies the following conditions:
\begin{enumerate}
\item[\textbf{(D1)}] $g(t)>0$, $g'(t)\leq 0$ and $g''(t)\geq 0$ for all $t\in[0,\infty)$,
\item[\textbf{(D2)}] $g\notin L^1(0,\infty)$,
\item[\textbf{(D3)}] $|d_t^kg(t)|\leq C_kg(t)\Big( \dfrac{1}{1+t} \Big)^k$ for all $t\in[0,\infty)$, $k=1,2$ and $C_1$, $C_2$ are positive constants.
\end{enumerate}
\begin{theorem} \label{ThmDecreasing}
Let us consider the Cauchy problem \eqref{modellinearviscoelasticdamped}, where the coefficient $g=g(t)$ satisfies the conditions \textbf{(D1)} to \textbf{(D3)}. Then, we have the following estimates for Sobolev solutions:
\begin{align*}
\big\||D|^{|\beta|}u(t,\cdot)\big\|_{L^2} &\lesssim \bigg( 1+\int_0^tg(\tau)d\tau \bigg)^{-\frac{|\beta|}{2}}\|u_0\|_{H^{|\beta|}} + \bigg( 1+\int_0^tg(\tau)d\tau \bigg)^{-\frac{|\beta|-1}{2}}\|u_1\|_{H^{|\beta|-1}} \quad \mbox{for} \quad |\beta|\geq 1, \\
\big\||D|^{|\beta|}u_t(t,\cdot)\big\|_{L^2} &\lesssim \bigg( 1+\int_0^tg(\tau)d\tau \bigg)^{-\frac{|\beta|+1}{2}}\|u_0\|_{H^{|\beta|+1}} + \bigg( 1+\int_0^tg(\tau)d\tau \bigg)^{-\frac{|\beta|}{2}}\|u_1\|_{H^{|\beta|}} \quad \mbox{for} \quad |\beta|\geq 0.
\end{align*}
\end{theorem}
\begin{remark} \label{Remarkparabolicdecreasing}
The statements of Theorem \ref{ThmDecreasing} show that we have the parabolic effect. This means that the energies of higher order of Sobolev solutions decay faster and faster with increasing order.
\end{remark}
\begin{exam} \label{Example4.1}
Let us choose $g(t)=(1+t)^{-\gamma}$, $\gamma\in (0,1]$. Then, $g=g(t)$ satisfies the assumptions of Theorem \ref{ThmDecreasing}. Consequently, the following estimates for Sobolev solutions hold: \\
$\gamma\in (0,1)$:
\begin{align*}
\big\||D|^{|\beta|}u(t,\cdot)\big\|_{L^2} &\lesssim (1+t)^{-\frac{|\beta|(1-\gamma)}{2}}\|u_0\|_{H^{|\beta|}} + (1+t)^{-\frac{(|\beta|-1)(1-\gamma)}{2}}\|u_1\|_{H^{|\beta|-1}} \quad \mbox{for} \quad |\beta|\geq 1, \\
\big\||D|^{|\beta|}u_t(t,\cdot)\big\|_{L^2} &\lesssim (1+t)^{-\frac{(|\beta|+1)(1-\gamma)}{2}}\|u_0\|_{H^{|\beta|+1}} + (1+t)^{-\frac{|\beta|(1-\gamma)}{2}}\|u_1\|_{H^{|\beta|}} \quad \mbox{for} \quad |\beta|\geq 0;
\end{align*}
$\gamma=1$:
\begin{align*}
\big\||D|^{|\beta|}u(t,\cdot)\big\|_{L^2} &\lesssim \Big( \log(e+t) \Big)^{-\frac{|\beta|}{2}}\|u_0\|_{H^{|\beta|}} + \Big( \log(e+t) \Big)^{-\frac{|\beta|-1}{2}}\|u_1\|_{H^{|\beta|-1}} \quad \mbox{for} \quad |\beta|\geq 1, \\
\big\||D|^{|\beta|}u_t(t,\cdot)\big\|_{L^2} &\lesssim \Big( \log(e+t) \Big)^{-\frac{|\beta|+1}{2}}\|u_0\|_{H^{|\beta|+1}} + \Big( \log(e+t) \Big)^{-\frac{|\beta|}{2}}\|u_1\|_{H^{|\beta|}} \quad \mbox{for} \quad |\beta|\geq 0.
\end{align*}
\end{exam}
\begin{exam} \label{Example4.2}
Let us consider $g(t)=\left( (e^2+t)\log(e^2+t) \right)^{-1}$. Then, $g=g(t)$ satisfies the assumptions of Theorem \ref{ThmDecreasing} and the following estimates for Sobolev solutions are given as follows:
\begin{align*}
\big\||D|^{|\beta|}u(t,\cdot)\big\|_{L^2} &\lesssim \Big( \log\big( \log(e^2+t) \big) \Big)^{-\frac{|\beta|}{2}}\|u_0\|_{H^{|\beta|}} + \Big( \log\big( \log(e^2+t) \big) \Big)^{-\frac{|\beta|-1}{2}}\|u_1\|_{H^{|\beta|-1}} \quad \mbox{for} \quad |\beta|\geq 1, \\
\big\||D|^{|\beta|}u_t(t,\cdot)\big\|_{L^2} &\lesssim \Big( \log\big( \log(e^2+t) \big) \Big)^{-\frac{|\beta|+1}{2}}\|u_0\|_{H^{|\beta|+1}} +  \Big( \log\big( \log(e^2+t) \big) \Big)^{-\frac{|\beta|}{2}}\|u_1\|_{H^{|\beta|}} \quad \mbox{for} \quad |\beta|\geq 0.
\end{align*}
\end{exam}
\begin{exam} \label{Example4.3}
Let us consider $g(t)=\dfrac{\log(e^2+t)}{e^2+t}$. Then, $g=g(t)$ satisfies the assumptions of Theorem \ref{ThmDecreasing} and the following estimates for Sobolev solutions are given as follows:
\begin{align*}
\big\||D|^{|\beta|}u(t,\cdot)\big\|_{L^2} &\lesssim \Big( \big( \log(e^2+t) \big)^2 \Big)^{-\frac{|\beta|}{2}}\|u_0\|_{H^{|\beta|}} + \Big( \big( \log(e^2+t) \big)^2 \Big)^{-\frac{|\beta|-1}{2}}\|u_1\|_{H^{|\beta|-1}} \quad \mbox{for} \quad |\beta|\geq 1, \\
\big\||D|^{|\beta|}u_t(t,\cdot)\big\|_{L^2} &\lesssim \Big( \big( \log(e^2+t) \big)^2 \Big)^{-\frac{|\beta|+1}{2}}\|u_0\|_{H^{|\beta|+1}} +  \Big( \big( \log(e^2+t) \big)^2 \Big)^{-\frac{|\beta|}{2}}\|u_1\|_{H^{|\beta|}} \quad \mbox{for} \quad |\beta|\geq 0.
\end{align*}
\end{exam}\begin{remark} \label{Remark4.3}
Let us consider the visco-elastic damped case in the paper \cite{DAbbiccoEbert2016} with $g(t)=(1+t)^{-\gamma}$, $\gamma\in(-1,1)$. Then, the dissipation is non-effective and there exist is a parabolic effect in this case. One may also see that the regularity of the data in this paper is given by $(u_0,u_1)\in H^{|\beta|} \times H^{|\beta|-2}$ (see the decay estimate and more details in Remark 20 of \cite{DAbbiccoEbert2016}).
\end{remark}

\begin{proof}[Proof of Theorem \ref{ThmDecreasing}]
We write the equation in (\ref{modellinearauxiliary2}) in the following form:
\begin{equation*} \label{auxiliaryproblem5}
D_t^2 v + \frac{g(t)^2}{4}|\xi|^4 v - \bigg( 1-\frac{g'(t)}{2} \bigg)|\xi|^2v=0.
\end{equation*}
We introduce $h=h(t)=1-\dfrac{g'(t)}{2} \geq 1$. The function $h$ is bounded if $t$ tends to infinity. We consider
\begin{equation*} \label{auxiliaryproblem6}
D_t^2 v + \frac{g(t)^2}{4}|\xi|^4 v - h(t)|\xi|^2v=0.
\end{equation*}
Due to condition \textbf{(D1)} the function $f=f(t)=\dfrac{g(t)^2}{h(t)}$ is monotonically decreasing.
Thus, we have a separating line $t_\xi$ as solution of the implicit equation $\dfrac{g(t)^2}{4h(t)}|\xi|^2=1$.
\medskip

Let us divide the extended phase space into the following zones:
\begin{itemize}
\item hyperbolic zone:
\[ \Zhyp = \Big\{ (t,\xi)\in[0,\infty)\times\mathbb{R}^n : 1- \frac{g(t)^2|\xi|^2}{4h(t)}\geq \frac{1}{4} \Big\}, \]
\item reduced zone:
\[ \Zred = \Big\{ (t,\xi)\in[0,\infty)\times\mathbb{R}^n : -\frac{1}{4} \leq 1- \frac{g(t)^2|\xi|^2}{4h(t)}\leq \frac{1}{4} \Big\}, \]
\item elliptic zone:
\[ \Zell = \Big\{ (t,\xi)\in[0,\infty)\times\mathbb{R}^n : 1- \frac{g(t)^2|\xi|^2}{4h(t)}\leq -\frac{1}{4} \Big\}. \]
\end{itemize}

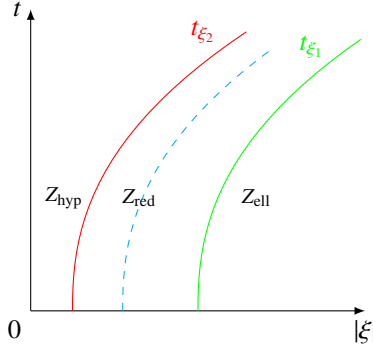
\begin{figure}[H]
\begin{center}
\begin{tikzpicture}[>=latex,xscale=1.1]
    \draw[->] (0,0) -- (4,0)node[below]{$|\xi|$};
    \draw[->] (0,0) -- (0,4)node[left]{$t$};
    \node[below left] at(0,0){$0$};
    \node[right] at (1.8,3.7) {$\textcolor{red}{t_{\xi_2}}$};
    \node[right] at (3.1,3.5) {$\textcolor{green}{t_{\xi_1}}$};
    \draw[dashed, domain=0:3.5,color=cyan,variable=\t] plot ({1.1 + 0.09*pow(\t,2.4)},\t);
	\node[color=black] at (2.7, 1.5){{\footnotesize $Z_{\text{ell}}$}};
	\node[color=black] at (1.3, 1.5){{\footnotesize $Z_{\text{red}}$}};
	\draw[domain=0:3.6,color=green,variable=\t] plot ({2 + 0.09*pow(\t,2.4)},\t);
	\draw[domain=0:3.7,color=red,variable=\t] plot ({0.5 + 0.09*pow(\t,2.4)},\t);
	\node[color=black] at (0.4,1.5){{\footnotesize $Z_{\text{hyp}}$}};	
\end{tikzpicture}
\caption{Sketch of the zones for the case $g=g(t)$ is non-integrable and decreasing}
\label{fig.zone.decreasing}
\end{center}
\end{figure}

We denote the separating line between elliptic and reduced zone as $t_{\xi_1}$ and that between hyperbolic zone and reduced zone as $t_{\xi_2}$.
The blue dashed line denotes the separating line between the hyperbolic and the elliptic region.

\subsection{Considerations in the hyperbolic zone $\Zhyp$} \label{SecDecresingZhyp}
\begin{proposition} \label{PropDecreasingHyp}
The following estimates hold for all $t_{\xi_2}\leq s \leq t$, where $t_{\xi_2}=0$ for small frequencies:
\begin{align*}
|\xi|^{|\beta|}|\hat{u}(t,\xi)| \lesssim \exp\bigg( -\frac{|\xi|^2}{2}\int_s^tg(\tau)d\tau \bigg)\Big( |\xi|^{|\beta|}|\hat{u}(s,\xi)| + |\xi|^{|\beta|-1}|\hat{u}_t(s,\xi)| \Big) \quad \mbox{for} \quad |\beta|\geq 1, \\
|\xi|^{|\beta|}|\hat{u}_t(t,\xi)| \lesssim \exp\bigg( -\frac{|\xi|^2}{2}\int_s^tg(\tau)d\tau \bigg)\Big( |\xi|^{|\beta|+1}|\hat{u}(s,\xi)| + |\xi|^{|\beta|}|\hat{u}_t(s,\xi)| \Big) \quad \mbox{for} \quad |\beta|\geq 0.
\end{align*}
\end{proposition}
\begin{proof}
Let us consider the equation
\begin{equation*}
v_{tt} + \bigg( \underbrace{h(t)|\xi|^2-\dfrac{g(t)^2}{4}|\xi|^4}_{:=p^2(t,\xi)} \bigg)v=0.
\end{equation*}
We define the micro-energy
\[ V(t,\xi) = \big( p(t,\xi)v, D_tv \big)^{\text{T}}, \qquad p(t,\xi) := \sqrt{h(t)|\xi|^2-\frac{g^2(t)}{4}|\xi|^4}, \qquad p(t,\xi)\approx \sqrt{h(t)}|\xi| \,\,\,\mbox{for}\,\,\,(t,\xi) \in \Zhyp. \]
Thus, we have to apply tools from hyperbolic WKB-analysis. Transformation to a system of first order from \eqref{modellinearauxiliary} gives
\begin{equation*}
D_tV=\left( \begin{array}{cc}
0 & p(t,\xi) \\
p(t,\xi) & 0
\end{array} \right)V + \left( \begin{array}{cc}
\dfrac{D_tp(t,\xi)}{p(t,\xi)} & 0 \\
0 & 0
\end{array} \right)V.
\end{equation*}
Using $V=MV^{(0)}$ with $M=\begin{pmatrix} 1 & -1 \\ 1 &1\end{pmatrix}$, then after the first step of diagonalization we obtain
\[ D_tV^{(0)} = \big( \mathcal{D}(t,\xi) + \mathcal{R}(t,\xi) \big)V^{(0)}, \]
where
\begin{align*}
\mathcal{D}(t,\xi) &= \left( \begin{array}{cc}
p(t,\xi) & 0 \\
0 & p(t,\xi)
\end{array} \right) \qquad \mbox{and} \qquad \mathcal{R}(t,\xi) = \frac{1}{2} \left( \begin{array}{cc}
\dfrac{D_tp(t,\xi)}{p(t,\xi)} & -\dfrac{D_tp(t,\xi)}{p(t,\xi)} \\
-\dfrac{D_tp(t,\xi)}{p(t,\xi)} & \dfrac{D_tp(t,\xi)}{p(t,\xi)}
\end{array} \right).
\end{align*}
After the first step of diagonalization procedure the entries of the matrix $\mathcal{R}(t,\xi)$ are uniformly integrable over the hyperbolic zone $\Zhyp$. We can write $V^{(1)}(t,\xi) = E_1(t,s,\xi)V^{(1)}(s,\xi)$, where $E_1=E_1(t,s,\xi)$ is the fundamental solution, that is, the solution of the system
\begin{align*}
D_tE_1(t,s,\xi) = \big( \mathcal{D}(t,\xi) + \mathcal{R}(t,\xi) \big)E_1(t,s,\xi), \quad E_1(s,s,\xi) = I,
\end{align*}
for all $t\geq s$ and $(s,\xi)\in\Zhyp$. Straightforward calculations imply (see Proposition 3.1 of \cite{KainaneReissig2015-1})
\[ |E_1(t,s,\xi)| \leq C \quad \text{for all} \quad t\geq s \quad \text{and} \quad (s,\xi)\in\Zhyp. \]
Finally, we obtain the following estimate for the transformed micro-energy $V^{(1)}(t,\xi)$ in the hyperbolic zone:
\[ |V^{(1)}(t,\xi)| \lesssim |V^{(1)}(s,\xi)|,  \qquad \bigg|\left( \begin{array}{cc}
p(t,\xi)v(t,\xi)| \\
D_tv(t,\xi)
\end{array} \right)\bigg|\lesssim \bigg|\left( \begin{array}{cc}
p(s,\xi)v(s,\xi) &  \\
D_tv(s,\xi)
\end{array} \right)\bigg| \]
uniformly for all $t\geq s$ and $(t,\xi),(s,\xi)\in \Zhyp$. From the backward transformation
\[ \hat{u}(t,\xi) = \exp\bigg( -\frac{|\xi|^2}{2}\int_0^tg(\tau)d\tau \bigg)v(t,\xi), \]
and the equivalence $p(t,\xi)\approx \sqrt{h(t)}|\xi|$, where $h=h(t)$ is bounded for large time $t$, gives the desired estimates.
\end{proof}
\subsection{Considerations in the reduced zone $\Zred$} \label{SecDecresingZred}
\begin{proposition} \label{PropDecreasingRed}
The following estimates hold for all $(t,\xi), (s,\xi) \in \Zred$ with $s \leq t$:
\begin{align*}
|\xi|^{|\beta|}|\hat{u}(t,\xi)| \lesssim \exp\bigg( -\frac{|\xi|^2}{6}\int_s^tg(\tau)d\tau \bigg)\Big( |\xi|^{|\beta|}|\hat{u}(s,\xi)| + |\xi|^{|\beta|-1}|\hat{u}_t(s,\xi)| \Big) \quad \mbox{for} \quad |\beta|\geq 1, \\
|\xi|^{|\beta|}|\hat{u}_t(t,\xi)| \lesssim \exp\bigg( -\frac{|\xi|^2}{6}\int_s^tg(\tau)d\tau \bigg)\Big( |\xi|^{|\beta|+1}|\hat{u}(s,\xi)| + |\xi|^{|\beta|}|\hat{u}_t(s,\xi)| \Big) \quad \mbox{for} \quad |\beta|\geq 0.
\end{align*}
\end{proposition}
\begin{proof}
In the reduced zone we have $\sqrt{h(t)}|\xi|\approx \dfrac{g(t)}{2}|\xi|^2$. Employing the transformed equation
\begin{equation*}
v_{tt} + \bigg( h(t)|\xi|^2-\dfrac{g^2(t)}{4}|\xi|^4\bigg)v=0,
\end{equation*}
we can estimate
\[ \Big| h(t)|\xi|^2-\frac{g^2(t)}{4}|\xi|^4 \Big|\leq \frac{g^2(t)}{12}|\xi|^4 \qquad \text{taking account of} \qquad \frac{g^2(t)}{5}|\xi|^4 \leq h(t)|\xi|^2 \leq \frac{g^2(t)}{3}|\xi|^4. \]
Thus, we define the micro-energy
\[ V(t,\xi) = \Big( \frac{g(t)}{4}|\xi|^2v, D_tv \Big)^{\text{T}} \quad \mbox{for all} \quad t\geq t_{\xi_2}  \quad \mbox{and} \quad (t,\xi)\in\Zred.  \]
Then, we get the following system of first order:
\begin{equation} \label{DecreasingSystemZred}
D_tV(t,\xi)=\underbrace{\left( \begin{array}{cc}
\dfrac{D_t g(t)}{g(t)} & \dfrac{g(t)}{4}|\xi|^2 \\
\dfrac{h(t)|\xi|^2-\frac{g^2(t)}{4}|\xi|^4}{\frac{g(t)}{4}|\xi|^2} & 0 \end{array} \right)}_{A_V(t,\xi)} V(t,\xi).
\end{equation}
To estimate the entries of this matrix we will use
\[ \frac{\Big| h(t)|\xi|^2-\frac{g^2(t)}{4}|\xi|^4 \Big|}{\frac{g(t)}{4}|\xi|^2}\leq \frac{g(t)}{3}|\xi|^2. \]
\begin{corollary} \label{Corollary4.2}
The fundamental solution $E=E(t,s,\xi)$ to \eqref{DecreasingSystemZred} for all $t\geq s$ and $(t,\xi), (s,\xi)\in\Zred$ satisfies
\[ |E(t,s,\xi)|\leq \exp\bigg( \frac{|\xi|^2}{3} \int_s^t g(\tau)d\tau \bigg). \]
\end{corollary}
From the backward transformation and the equivalence $\dfrac{g(t)}{2}|\xi|^2\approx h(t)|\xi|$ in $\Zred$, we may conclude the desired statements of the proposition.
\end{proof}

\subsection{Considerations in the elliptic zone $\Zell$} \label{SecDecresingZell}

To estimate $-g'(t)$ we use the definition of the elliptic zone. We get
\begin{equation*} \label{estimategprime}
-\frac{g'(t)}{2} \leq 1-\frac{g'(t)}{2}=h(t)\leq \frac{g(t)^2|\xi|^2}{5}.
\end{equation*}
To estimate  $g''(t)$ we use assumption \textbf{(D3)}.

Let us define the following classes of symbols related to the properties of $g=g(t)$ and $\Zell$.
\begin{definition} \label{Definition4.1}
A function $f=f(t,\xi)$ belongs to the elliptic symbol class $S_{\text{ell}}^\ell\{m_1,m_2,m_3\}$ if it holds
\begin{equation*} \label{DecreasingSymbolClassEll}
|D_t^kf(t,\xi)|\leq C_{k}|\xi|^{m_1}g(t)^{m_2}\bigg( \frac{1}{1+t} \bigg)^{m_3+k}
\end{equation*}
for all $(t,\xi)\in \Zell$ and all $k\leq \ell$.
\end{definition}
The further considerations are basing on the following rules of the symbolic calculus.
\begin{proposition} \label{DecreasingPropSymbolEll}
The following statements are true:
\begin{itemize}
\item $S_{\text{ell}}^\ell\{m_1,m_2,m_3\}$ is a vector space for all nonnegative integers $\ell$;
\item $S_{\text{ell}}^\ell\{m_1,m_2+k,m_3\}\hookrightarrow S_{\text{ell}}^{\ell}\{m_1,m_2,m_3\}$ for $k\geq 0$;
\item $S_{\text{ell}}^\ell\{m_1,m_2,m_3\}\cdot S_{\text{ell}}^{\ell}\{m_1',m_2',m_3'\}\hookrightarrow S^{\ell}_{\text{ell}}\{m_1+m_1',m_2+m_2',m_3+m_3'\}$;
\item $D_t^kS_{\text{ell}}^\ell\{m_1,m_2,m_3\}\hookrightarrow S_{\text{ell}}^{\ell-k}\{m_1,m_2,m_3+k\}$
for all nonnegative integers $\ell$ with $k\leq \ell$.
\end{itemize}
\end{proposition}
Let us turn to the equation \eqref{modellinearviscoelasticdampedFourier} in the following form:
\begin{equation} \label{EqNonIntegrableCasewithD}
D_t^2u - |\xi|^2u - ig(t)|\xi|^2D_tu = 0.
\end{equation}
If we introduce the micro-energy $U=U(t,\xi)$ in $\Zell$ by $U=(|\xi|\hat{u},D_t\hat{u})^\text{T}$, then the corresponding first-order system of \eqref{EqNonIntegrableCasewithD} leads to
\begin{equation} \label{EqDecreasingSzstemU}
D_tU=\underbrace{\left( \begin{array}{cc}
0 & |\xi| \\
|\xi| & ig(t)|\xi|^2
\end{array} \right)}_{A(t,\xi)}U.
\end{equation}
\begin{proposition} \label{PropDecreasingEll}
The following estimates hold for the solutions to (\ref{EqDecreasingSzstemU}) for all $t\in[0,t_{\xi_1}]$:
\begin{align*}
|\xi|^{|\beta|}|\hat{u}(t,\xi)| &\lesssim \exp\bigg( -C\int_0^t\frac{1}{g(\tau)}d\tau \bigg)\Big( |\xi|^{|\beta|}|\hat{u}_0(\xi)| + |\xi|^{|\beta|-1|}|\hat{u}_1(\xi)| \Big) \quad \mbox{for} \quad |\beta|\geq1, \\
|\xi|^{|\beta|}|\hat{u}_t(t,\xi)| &\lesssim \exp\bigg( -C\int_{0}^t\frac{1}{g(\tau)}d\tau \bigg)\Big( |\xi|^{|\beta|+1}|\hat{u}_0(\xi)| + |\xi|^{|\beta|}|\hat{u}_1(\xi)| \Big) + \exp\bigg( -\frac{|\xi|^2}{2}\int_{0}^tg(\tau)d\tau \bigg)|\xi|^{|\beta|}|\hat{u}_1(\xi)| \quad \mbox{for} \quad |\beta|\geq0.
\end{align*}
\end{proposition}
\begin{proof}
The proof is divided into two steps.\medskip

\noindent \textbf{Step 1.} \emph{A straight-forward estimate for the fundamental solution $E=E(t,s,\xi)$}
\begin{proposition} \label{propEstimateEjk}
The fundamental solution $E=E(t,s,\xi)$ to (\ref{EqDecreasingSzstemU}) satisfies for all $t\geq s$ and $(t,\xi)$, $(s,\xi)\in\Zell$ the following estimates:
\begin{align*} \label{EqPropEstEjk}
\left(\begin{array}{cc}
|E^{(11)}(t,s,\xi)| & |E^{(12)}(t,s,\xi)| \\
|E^{(21)}(t,s,\xi)| & |E^{(22)}(t,s,\xi)|
\end{array}\right) \lesssim \exp\bigg( -C\int_s^t\frac{1}{g(\tau)}d\tau \bigg)\left(\begin{array}{cc}
1 & \dfrac{1}{g(s)|\xi|} \\
g(t)|\xi| & \dfrac{g(t)}{g(s)}
\end{array}\right),
\end{align*}
where the constant $C$ is independent of $(s,\xi), (t,\xi)\in\Zell$.
\end{proposition}
\begin{proof}
Let us carry out the first step of diagonalization. The eigenvalues of the matrix $A=A(t,\xi)$ are
\[ \lambda_k(t,\xi) = \frac{ig(t)|\xi|^2 + (-1)^{k-1}i\sqrt{g^2(t)|\xi|^4-4|\xi|^2}}{2}, \quad k=1,2. \]
In the further calculations we use the following properties of $\lambda_1(t,\xi)$ and $\lambda_2(t,\xi)$.
\begin{lemma} \label{LemmaDecresingImaginary}
It holds
\begin{enumerate}
\item $\Im\lambda_1(t,\xi)+\Im\lambda_2(t,\xi) = g(t)|\xi|^2$, \,\,\, $\Im\lambda_1(t,\xi)\Im\lambda_2(t,\xi) =|\xi|^2$,
\item $\Im\lambda_1(t,\xi)\geq\Im\lambda_2(t,\xi) \geq0$, \,\,\, $|\lambda_1(t,\xi)|\geq |\lambda_2(t,\xi)|$,
\item $\dfrac{g(t)}{2}|\xi|^2 \leq \Im\lambda_1(t,\xi) \leq g(t)|\xi|^2$, \,\,\, $\dfrac{1}{g(t)}\leq \Im\lambda_2(t,\xi)\leq \dfrac{2}{g(t)}$.
\end{enumerate}
\end{lemma}
Then, we introduce the corresponding matrix of eigenvectors $M=M(t,\xi)$ and $M^{-1}=M^{-1}(t,\xi)$ as
\begin{align*}
M(t,\xi) := \left(\begin{array}{cc}
1 & 1 \\
\lambda_1(t,\xi)|\xi|^{-1} & \lambda_2(t,\xi)|\xi|^{-1}
\end{array}\right), \qquad M^{-1}(t,\xi) := \frac{i}{\sqrt{g^2(t)|\xi|^2-4}}\left(\begin{array}{cc}
\lambda_2(t,\xi)|\xi|^{-1} & -1 \\
-\lambda_1(t,\xi)|\xi|^{-1} & 1
\end{array}\right).
\end{align*}
Setting $U^{(1)}(t,\xi):=M^{-1}(t,\xi)U(t,\xi)$ for all $t\geq s$ and $(t,\xi)\in\Zell$ we obtain the following system:
\begin{equation} \label{EqDecreasingSystemU1}
D_tU^{(1)}(t,\xi)=M^{-1}(t,\xi)A(t,\xi)M(t,\xi)U^{(1)}(t,\xi) - M^{-1}(t,\xi)D_tM(t,\xi)U^{(1)}(t,\xi).
\end{equation}
Straight-forward calculations imply
\begin{align}
\label{EqDecreasingD}
\mathcal{D}(t,\xi) &= M^{-1}(t,\xi)A(t,\xi)M(t,\xi) = \left(\begin{array}{cc}
\lambda_1(t,\xi) & 0 \\
0 & \lambda_2(t,\xi)
\end{array} \right),\\
\label{EqDecreasingR}
\mathcal{R}(t,\xi) &= M^{-1}(t,\xi)D_tM(t,\xi) = -\frac{1}{2} \left(\begin{array}{cc}
a+b & a-b \\
-a-b & -a+b
\end{array} \right),
\end{align}
where
\begin{align*} \label{EqDecreasingab}
a := \frac{g'(t)|\xi|^2}{i\sqrt{g^2(t)|\xi|^4-4|\xi|^2}}, \qquad b := \frac{g(t)g'(t)|\xi|^4}{i(g^2(t)|\xi|^4-4|\xi|^2)}.
\end{align*}
The system \eqref{EqDecreasingSystemU1} has diagonal principal part $\mathcal{D}\in S_{\text{ell}}^2\{2,1,0\}$ with the remainder $\mathcal{R}\in S_{\text{ell}}^1\{0,0,1\}$. We carry out one more step of diagonalization procedure. Let
\begin{equation} \label{EqNonIntegrableDecMatrN1}
N^{(1)} := \left( \begin{array}{cc}
0 & \dfrac{R_{12}}{\lambda_2-\lambda_1} \\
\dfrac{R_{21}}{\lambda_1-\lambda_2} & 0
\end{array}\right)\in S_{\text{ell}}^1\{-2,-1,1\}.
\end{equation}
Now we set $N_1(t,\xi) := I + N^{(1)}(t,\xi)$. In order to prove the invertibility of $N_1$ we need to estimate the entries of $N^{(1)}$. We have
\begin{align*}
\dfrac{R_{21}}{\lambda_1-\lambda_2} &= -\frac{1}{2}\frac{g'(t)}{g^2(t)|\xi|^2-4} - \frac{1}{2}\frac{g(t)g'(t)|\xi|}{(g^2(t)|\xi|^2-4)^{\frac{3}{2}}}, \\
\dfrac{R_{12}}{\lambda_2-\lambda_1} &= -\frac{1}{2}\frac{g'(t)}{g^2(t)|\xi|^2-4} + \frac{1}{2}\frac{g(t)g'(t)|\xi|}{(g^2(t)|\xi|^2-4)^{\frac{3}{2}}}.
\end{align*}
From the definition of $\Zell$, we have $g^2(t)|\xi|^2 \geq 5h(t)$. This implies the following estimates:
\begin{align} \label{Eq:EstimatesforInvert}
\frac{1}{g^2(t)|\xi|^2} \leq \frac{1}{5h(t)} \qquad \text{and} \qquad \dfrac{1}{1-\frac{4}{g^2(t)|\xi|^2}} \leq \dfrac{1}{1-\frac{4}{5h(t)}}.
\end{align}
Thus, using the estimates in \eqref{Eq:EstimatesforInvert} we get the following estimates:
\begin{align*}
-\frac{1}{2}\frac{g'(t)}{g^2(t)|\xi|^2-4} &= -\frac{g'(t)}{2}\frac{1}{g^2(t)|\xi|^2\Big( 1-\frac{4}{g^2(t)|\xi|^2} \Big)} \leq -\frac{g'(t)}{2}\frac{1}{5h(t)\Big( 1-\frac{4}{5h(t)} \Big)} = -\frac{g'(t)}{2-5g'(t)} \leq \frac{1}{5}.
\end{align*}
Using the previous estimate, we have
\begin{align*}
-\frac{1}{2}\frac{g(t)g'(t)|\xi|}{(g^2(t)|\xi|^2-4)^{\frac{3}{2}}} &= -\frac{g'(t)}{2}\frac{1}{g^2(t)|\xi|^2\Big( 1-\frac{4}{g^2(t)|\xi|^2} \Big)^{\frac{3}{2}}} \leq -\frac{g'(t)}{2-5g'(t)}\frac{1}{\Big( 1-\frac{4}{5h(t)} \Big)^{\frac{1}{2}}} \leq \frac{1}{5}\frac{1}{\sqrt{1-\frac{4}{5}}} = \frac{1}{\sqrt{5}}.
\end{align*}
Therefore, the previous two inequalities guarantee that the matrix $N_1=N_1(t,\xi)$ is invertible.

Let $\mathcal{R}_1:=-N_1^{-1}\big( (D_t-\mathcal{R})N^{(1)}+N^{(1)}F^{(1)} \big)$. Then, in $\Zell$ we have $\mathcal{R}_1\in S_{\text{ell}}^0\{-2,1,2\}$ such that the following identity holds:
\[ \big( D_t-\mathcal{D}(t,\xi)-\mathcal{R}(t,\xi) \big)N_1(t,\xi) = N_1(t,\xi)\big( D_t-\mathcal{D}(t,\xi)-F^{(1)}(t,\xi)-\mathcal{R}_1(t,\xi) \big), \]
where $F^{(1)}=\diag \mathcal{R}$ and, $\mathcal{D}$ and $\mathcal{R}$ are defined in \eqref{EqDecreasingD} and \eqref{EqDecreasingR}, respectively.

The representation of $\mathcal{R}_1$ follows immediately from the fact that $[N^{(1)},\mathcal{D}]=\mathcal{R}-F^{(1)}$. By taking into consideration the matrices $F^{(1)}\in S_{\text{ell}}^1\{0,0,1\}$, $N^{(1)}\in S_{\text{ell}}^1\{-2,-1,1\}$ and $\mathcal{R}\in S_{\text{ell}}^1\{0,0,1\}$, and using the property of symbol classes from Proposition \ref{DecreasingPropSymbolEll} we may conclude that $\mathcal{R}_1\in S_{\text{ell}}^{0}\{-2,-1,2\}$.

Let $U^{(2)}(t,\xi):=N_1^{-1}(t,\xi)M^{-1}(t,\xi)U(t,\xi)$, then we obtain the following equivalent problem to \eqref{EqDecreasingSzstemU} for $U^{(2)}(t,\xi)$:
\begin{equation*}
\big( D_t-\mathcal{D}(t,\xi)-F^{(1)}(t,\xi)-\mathcal{R}_1(t,\xi) \big)U^{(2)}(t,\xi) = 0,
\end{equation*}
for all $t\geq s$ and $(t,\xi), (s,\xi)\in\Zell$. Thus, we have obtained in $\Zell$ the diagonalization of the system \eqref{EqDecreasingSzstemU} modulo remainder $\mathcal{R}_1\in S_{\text{ell}}^{0}\{-2,-1,2\}$. Namely, the entries of the matrix $\mathcal{R}_1=\mathcal{R}_1(t,\xi)$ are uniformly integrable over the elliptic zone $\Zell$. For this reason, the matrix $\mathcal{R}_1=\mathcal{R}_1(t,\xi)$ belongs to $L^1_{\text{loc}}(\Zell)$. Hence, we can find the solution to the system
\[ \big( D_t-\mathcal{D}(t,\xi)-F^{(1)}(t,\xi)-\mathcal{R}_1(t,\xi) \big)U^{(2)}(t,\xi) = 0. \]
We can write $U^{(2)}(t,\xi)=E_2(t,s,\xi)U^{(2)}(s,\xi)$, where $E_2=E_2(t,s,\xi)$ is the fundamental solution to the system
\[ D_tE_2(t,s,\xi) = \big( \mathcal{D}(t,\xi)+F^{(1)}(t,\xi)+\mathcal{R}_1(t,\xi) \big)E_2(t,s,\xi), \quad E_2(s,s,\xi) = I, \]
for all $t\geq s$ and $(t,\xi), (s,\xi)\in\Zell$. First, we estimate $E_d=E_d(t,s,\xi)$ as the fundamental solution of the diagonal part of this system, that is,
\[ D_tE_d(t,s,\xi) = \big( \mathcal{D}(t,\xi)+F^{(1)}(t,\xi) \big)E_d(t,s,\xi), \quad E_d(s,s,\xi) = I, \]
for all $t\geq s$ and $(t,\xi), (s,\xi)\in\Zell$. Thus,
\begin{align*}
E_d^{(11)}(t,s,\xi) &= \exp\bigg\{ -\frac{1}{2}\int_s^t\bigg( 1+\frac{g'(\tau)|\xi|^2}{g^2(\tau)|\xi|^4-4|\xi|^2} \bigg)\bigg( \sqrt{g^2(\tau)|\xi|^4-4|\xi|^2}+g(\tau)|\xi|^2 \bigg)d\tau \bigg\}, \\
E_d^{(22)}(t,s,\xi) &= \exp\bigg\{ \frac{1}{2}\int_s^t\bigg( 1+\frac{g'(\tau)|\xi|^2}{g^2(\tau)|\xi|^4-4|\xi|^2} \bigg)\bigg( \sqrt{g^2(\tau)|\xi|^4-4|\xi|^2}-g(\tau)|\xi|^2 \bigg)d\tau \bigg\} ,\\
E_d^{(12)}(t,s,\xi) &= E_d^{(21)}(t,s,\xi) = 0.
\end{align*}
\begin{proposition} \label{LemmaEstimateEd}
We have the following estimate for all $(t,\xi), (s,\xi) \in\Zell$:
\[ |E_d(t,s,\xi)|\lesssim \exp\bigg( -C\int_s^t\frac{1}{g(\tau)}d\tau \bigg) \]
with a positive constant $C$ which is independent of $(t,\xi), (s,\xi)\in\Zell$.
\end{proposition}
\begin{proof}
The estimate for $E_d=E_d(t,s,\xi)$ will be determined by the estimate of $E_d^{(22)}=E_d^{(22)}(t,s,\xi)$. By applying the definition of the elliptic zone and Lemma \ref{LemmaDecresingImaginary}, we get the following estimates:
\begin{align*}
\frac{1}{2}\bigg( \sqrt{g^2(t)|\xi|^4-4|\xi|^2}-g(t)|\xi|^2 \bigg)\leq -\frac{1}{g(t)} \qquad \mbox{and} \qquad \frac{|g'(\tau)||\xi|^2}{g^2(\tau)|\xi|^4-4|\xi|^2}\leq \frac{2}{5} \quad \mbox{by (\ref{Eq:EstimatesforInvert})}.
\end{align*}
This completes the proof.
\end{proof}
The fundamental solution $E_2=E_2(t,s,\xi)$ satisfies
\[ \big( D_t-\mathcal{D}(t,\xi)-F^{(1)}(t,\xi)-\mathcal{R}_1(t,\xi) \big)E_2(t,s,\xi) = 0, \quad E_2(s,s,\xi) = I \]
for all $t\geq s$ and $(t,\xi), (s,\xi)\in\Zell$. Following the same procedure as in the proof of Proposition \ref{Lem.Est.Ell.Zone}, we have
\begin{align*}
E_2(t,s,\xi) &= \exp\bigg( i\int_s^t\Big( \mathcal{D}(\tau,\xi)+F^{(1)}(\tau,\xi) \Big)d\tau \bigg)E_2(s,s,\xi) \\
& \quad + i\int_s^t\exp\bigg(i\int_\theta^t \Big( \mathcal{D}(\tau,\xi)+F^{(1)}(\tau,\xi) \Big)d\tau \bigg)\mathcal{R}_1(\theta,\xi)E_2(t,s,\xi)d\theta.
\end{align*}
Then, using the expression for  $E_d^{(22)}=E_d^{(22)}(t,s,\xi)$ we introduce the weight
\begin{equation} \label{Definition_Beta}
\beta(t,\xi) = \frac{1}{2}\bigg( 1+\frac{g'(t)|\xi|^2}{g^2(t)|\xi|^4-4|\xi|^2} \bigg)\bigg( \sqrt{g^2(t)|\xi|^4-4|\xi|^2}-g(t)|\xi|^2 \bigg).
\end{equation}
Let us define
\begin{equation} \label{Eq:NonintegrableDecDefQ}
Q_{\text{ell}}=Q_{\text{ell}}(t,s,\xi):=\exp\bigg( -\int_s^t\beta(\tau,\xi)d\tau \bigg)E_2(t,s,\xi).
\end{equation}
Then, we get
\begin{align*}
\mathcal{Q}_{\text{ell}}(t,s,\xi) &= \exp \bigg\{ \int_{s}^{t}\Big( i\mathcal{D}(\tau,\xi)+iF^{(1)}(\tau,\xi)-\beta(\tau,\xi)I \Big)d\tau \bigg\}\\
& \quad + \int_{s}^{t} \exp \bigg\{ \int_{\theta}^{t}\Big( i\mathcal{D}(\tau,\xi)+iF^{(1)}(\tau,\xi)-\beta(\tau,\xi)I \Big)d\tau \bigg\}\mathcal{R}_1(\theta,\xi)\mathcal{Q}_{\text{ell}}(\theta,s,\xi)\,d\theta.
\end{align*}
Furthermore,
\begin{align*}
H(t,s,\xi) & =\exp \bigg\{ \int_{s}^{t}\big( i\mathcal{D}(\tau,\xi)+iF^{(1)}(\tau,\xi)-\beta(\tau,\xi)I \big)d\tau \bigg\}\\
& = \diag \bigg( \exp \bigg\{ -\frac{1}{2}\int_{s}^{t}\Big( \sqrt{g^2(\tau)|\xi|^4-|\xi|^2}+\frac{g'(\tau)|\xi|^2}{\sqrt{g^2(\tau)|\xi|^4-4|\xi|^2}} \bigg)d\tau \bigg\}, 1 \bigg).
\end{align*}
Hence, the matrix $H=H(t,s,\xi)$ is uniformly bounded for $(s,\xi), (t,\xi)\in\Zell$. Taking account of $\mathcal{R}_1\in S_{\text{ell}}^{0}\{-2,-1,2\}$, the matrix $Q_{\text{ell}}=Q_{\text{ell}}(t,s,\xi)$ which is given by the matrix representation, is uniformly bounded in $\Zell$. From the last consideration we may conclude
\begin{align*}
\left(\begin{array}{cc}
|E_2^{(11)}(t,s,\xi)| & |E_2^{(12)}(t,s,\xi)| \\
|E_2^{(21)}(t,s,\xi)| & |E_2^{(22)}(t,s,\xi)|
\end{array}\right) \lesssim \exp\bigg( -C\int_s^t\frac{1}{g(\tau)}d\tau \bigg)\left(\begin{array}{cc}
1 & 1 \\
1 & 1
\end{array}\right)
\end{align*}
for all $t\geq s$ and $(s,\xi), (t,\xi)\in\Zell$. From $U^{(2)}(t,\xi)=N_1^{-1}(t,\xi)M^{-1}(t,\xi)U(t,\xi)$ the backward transformation gives the representation
\[ E(t,s,\xi) = M(t,\xi)N_1(t,\xi)E_2(t,s,\xi)N_1^{-1}(s,\xi)M^{-1}(s,\xi). \]
Due to Proposition \ref{LemmaEstimateEd} and the uniform bounded behavior of $Q_{\text{ell}}$ and $N_1$ we have
\begin{align*}
(|E(t,s,\xi)|) &\lesssim |M(t,\xi)|\left(\begin{array}{cc}
1 & 1 \\
1 & 1
\end{array}\right)\exp\bigg( -C\int_s^t\frac{1}{g(\tau)}d\tau \bigg)\left(\begin{array}{cc}
1 & 1 \\
1 & 1
\end{array}\right)|M^{-1}(s,\xi)| \\
&\lesssim \exp\bigg( -C\int_s^t\frac{1}{g(\tau)}d\tau \bigg)\left(\begin{array}{cc}
1 & \dfrac{1}{g(s)|\xi|} \\
g(t)|\xi| & \dfrac{g(t)}{g(s)}
\end{array}\right),
\end{align*}
where we used $|\lambda_1(t,\xi)|\approx g(t)|\xi|^2$, $|\lambda_2(t,\xi)|\approx \dfrac{1}{g(t)}$, $|\det M(t,\xi)|\approx g(t)|\xi|$, the definition of $\Zell$ and the fact that $g$ is decreasing for $s\leq t$. These estimates give the desired statements and the proof is completed.
\end{proof}
\noindent \textbf{Step 2.} \emph{A refined estimate for the fundamental solution $E=E(t,s,\xi)$}
\begin{proposition} \label{LemmaRefinedEstimatesZellDecreasing}
The fundamental solution $E=E(t,s,\xi)$ satisfies for all $(t,\xi)$, $(s,\xi)\in\Zell$ the following estimates:
\begin{align*}
& \left(\begin{array}{cc}
|E^{(11)}(t,s,\xi)| & |E^{(12)}(t,s,\xi)| \\
|E^{(21)}(t,s,\xi)| & |E^{(22)}(t,s,\xi)|
\end{array}\right) \\ & \qquad \qquad \lesssim \exp\bigg( -C\int_s^t\frac{1}{g(\tau)}d\tau \bigg)\left(\begin{array}{cc}
1 & \dfrac{1}{g(s)|\xi|} \\
\dfrac{1}{g(t)|\xi|} & \dfrac{1}{g(t)g(s)|\xi|^2}
\end{array}\right)
+ \exp\bigg( -|\xi|^2\int_s^tg(\tau)d\tau \bigg)\left(\begin{array}{cc}
0 & 0 \\
0 & 1
\end{array} \right),
\end{align*}
where the constant $C$ is independent of $(s,\xi), (t,\xi)\in\Zell$.
\end{proposition}
\begin{proof} Using conditions \textbf{(D1)}, \textbf{(D2)} and \textbf{(D3)}, we may follow the techniques of the proof of Lemma 3.9 in \cite{ReissigLu2009}.
\end{proof}
\begin{corollary}
We have the following representation of solutions to (\ref{EqDecreasingSzstemU}) for all $t\geq s$ and $(t,\xi), (s,\xi)\in \Zell$:
\begin{align*}
\left(\begin{array}{cc}
|\xi|\hat{u}(t,\xi) \\
D_t\hat{u}(t,\xi)
\end{array}\right) = \exp\bigg( \int_s^t\beta(s,\xi)ds \bigg)M(t,\xi)N_1(t,\xi)Q_{\text{ell}}(t,s,\xi) N_1^{-1}(s,\xi)M^{-1}(s,\xi)\left(\begin{array}{cc}
|\xi|\hat{u}(s,\xi) \\
D_t\hat{u}(s,\xi)
\end{array}\right),
\end{align*}
where from \eqref{Eq:NonintegrableDecDefQ} we used $E_2(t,s,\xi) = \exp\Big( \int_s^t\beta(\tau,\xi)d\tau \Big)Q_{\text{ell}}(t,s,\xi)$ with $\beta=\beta(t,\xi)$ in \eqref{Definition_Beta}.
\end{corollary}
Taking account of the representation of solutions from the previous corollary with $s=0$ and the refined estimates from Proposition \ref{LemmaRefinedEstimatesZellDecreasing}, it follows
\begin{align*}
|\xi||\hat{u}(t,\xi)| &\lesssim \exp\bigg( -C\int_0^t\frac{1}{g(\tau)}d\tau \bigg)\Big( |\xi||\hat{u}_0(\xi)| + |\hat{u}_1(\xi)| \Big), \\
|\hat{u}_t(t,\xi)| &\lesssim  \exp\bigg( -C\int_0^t\frac{1}{g(\tau)}d\tau \bigg)\Big( |\xi||\hat{u}_0(\xi)| + |\hat{u}_1(\xi)| \Big) + \exp\bigg( -|\xi|^2\int_0^tg(\tau)d\tau \bigg)|\hat{u}_1(\xi)|.
\end{align*}
This completes the proof of Proposition \ref{PropDecreasingEll}.
\end{proof}
\subsection{Gluing procedure} \label{Section4.44}
For large frequencies we have to glue the statements from Proposition \ref{PropDecreasingHyp}, Proposition \ref{PropDecreasingRed} and Proposition \ref{PropDecreasingEll}. We are able to extend the estimates from $\Zhyp$ in Proposition \ref{PropDecreasingHyp} to $\Zred$ in Proposition \ref{PropDecreasingRed}. For this reason, we obtain the following statement.
\begin{corollary} \label{CorDecreasingGlueHypRed}
The following estimates hold for all $t \geq s \geq t_{\xi_1}$:
\begin{align*}
|\xi|^{|\beta|}|\hat{u}(t,\xi)| \lesssim \exp\bigg( -\frac{|\xi|^2}{6}\int_s^tg(\tau)d\tau \bigg)\Big( |\xi|^{|\beta|}|\hat{u}(s,\xi)| + |\xi|^{|\beta|-1}|\hat{u}_t(s,\xi)| \Big) \quad \mbox{for} \quad |\beta|\geq 1, \\
|\xi|^{|\beta|}|\hat{u}_t(t,\xi)| \lesssim \exp\bigg( -\frac{|\xi|^2}{6}\int_s^tg(\tau)d\tau \bigg)\Big( |\xi|^{|\beta|+1}|\hat{u}(s,\xi)| + |\xi|^{|\beta|}|\hat{u}_t(s,\xi)| \Big) \quad \mbox{for} \quad |\beta|\geq 0.
\end{align*}
\end{corollary}
Finally, we have to glue estimates from Corollary \ref{CorDecreasingGlueHypRed} and estimates from Proposition \ref{PropDecreasingEll} for $t=t_{\xi_1}$.
\begin{corollary} \label{CorDecreasingGlueHypRedEll}
The following estimates hold for all $t\in[t_{\xi_1},\infty)$:
\begin{align*}
|\xi|^{|\beta|}|\hat{u}(t,\xi)| &\lesssim \exp\bigg( -\frac{|\xi|^2}{6}\int_{t_{\xi_1}}^tg(\tau)d\tau \bigg)\exp\bigg( -C\int_0^{t_{\xi_1}}\frac{1}{g(\tau)}d\tau \bigg)\Big( |\xi|^{|\beta|}|\hat{u}_0(\xi)| + |\xi|^{|\beta|-1}|\hat{u}_1(\xi)| \Big) \\
& \quad + \exp\bigg( -\frac{|\xi|^2}{6}\int_{0}^tg(\tau)d\tau \bigg)|\xi|^{|\beta|-1}|\hat{u}_1(\xi)| \quad \mbox{for} \quad |\beta|\geq1, \\
|\xi|^{|\beta|}|\hat{u}_t(t,\xi)| &\lesssim \exp\bigg( -\frac{|\xi|^2}{6}\int_{t_{\xi_1}}^tg(\tau)d\tau \bigg)\exp\bigg( -C\int_0^{t_{\xi_1}}\frac{1}{g(\tau)}d\tau \bigg)\Big( |\xi|^{|\beta|+1}|\hat{u}_0(\xi)| + |\xi|^{|\beta|}|\hat{u}_1(\xi)| \Big) \\
&\quad + \exp\bigg( -\frac{|\xi|^2}{6}\int_{0}^tg(\tau)d\tau \bigg)|\xi|^{|\beta|}|\hat{u}_1(\xi)| \quad \mbox{for} \quad |\beta|\geq0.
\end{align*}
\end{corollary}
\begin{proof}
Let us begin to estimate $|\xi|^{|\beta|}\hat{u}(t,\xi)$. The statement of Corollary \ref{CorDecreasingGlueHypRed} implies
\[ |\xi|^{|\beta|}|\hat{u}(t,\xi)| \lesssim \exp\bigg( -\frac{|\xi|^2}{6}\int_{t_{\xi_1}}^tg(\tau)d\tau \bigg)\Big( |\xi|^{|\beta|}|\hat{u}(t_{\xi_1},\xi)| + |\xi|^{|\beta|-1}|\hat{u}_t(t_{\xi_1},\xi)| \Big). \]
Using the estimates for $|\xi|^{|\beta|}|\hat{u}(t_{\xi_1},\xi)|$ and $|\xi|^{|\beta|}|\hat{u}_t(t_{\xi_1},\xi)|$ from Proposition \ref{PropDecreasingEll} we have
\begin{align*}
|\xi|^{|\beta|}|\hat{u}(t,\xi)| &\lesssim \exp\bigg( -\frac{|\xi|^2}{6}\int_{t_{\xi_1}}^tg(\tau)d\tau \bigg)\exp\bigg( -C\int_0^{t_{\xi_1}}\frac{1}{g(\tau)}d\tau \bigg)\Big( |\xi|^{|\beta|}|\hat{u}_0(\xi)| + |\xi|^{|\beta|-1}|\hat{u}_1(\xi)| \Big) \\
& \quad + \exp\bigg( -\frac{|\xi|^2}{6}\int_{0}^tg(\tau)d\tau \bigg)|\xi|^{|\beta|-1}|\hat{u}_1(\xi)|, \\
|\xi|^{|\beta|}|\hat{u}_t(t,\xi)| &\lesssim \exp\bigg( -\frac{|\xi|^2}{6}\int_{t_{\xi_1}}^tg(\tau)d\tau \bigg)\exp\bigg( -C\int_0^{t_{\xi_1}}\frac{1}{g(\tau)}d\tau \bigg)\Big( |\xi|^{|\beta|+1}|\hat{u}_0(\xi)| + |\xi|^{|\beta|}|\hat{u}_1(\xi)| \Big)\\
& \quad + \exp\bigg( -\frac{|\xi|^2}{6}\int_{0}^tg(\tau)d\tau \bigg)|\xi|^{|\beta|}|\hat{u}_1(\xi)|.
\end{align*}
This completes the proof.
\end{proof}
For small frequencies we may use the estimates from Proposition \ref{PropDecreasingHyp}, because $t_{\xi_2}=0$.
\begin{corollary} \label{CorDecreasingGlueHypN}
The following estimates hold for all $t\in[0,\infty)$ and $0<|\xi| \leq \frac{\sqrt{3h(0)}}{g(0)}$:
\begin{align*}
|\xi|^{|\beta|}|\hat{u}(t,\xi)| &\lesssim \exp\bigg( -\frac{|\xi|^2}{2}\int_{0}^tg(\tau)d\tau \bigg)\Big( |\xi|^{|\beta|}|\hat{u}_0(\xi)| + |\xi|^{|\beta|-1}|\hat{u}_1(\xi)| \Big) \quad \mbox{for} \quad |\beta|\geq 1, \\
|\xi|^{|\beta|}|\hat{u}_t(t,\xi)| &\lesssim \exp\bigg( -\frac{|\xi|^2}{2}\int_{0}^tg(\tau)d\tau \bigg)\Big( |\xi|^{|\beta|+1}|\hat{u}_0(\xi)| + |\xi|^{|\beta|}|\hat{u}_1(\xi)| \Big) \quad \mbox{for} \quad |\beta|\geq 0.
\end{align*}
\end{corollary}
\subsection{Energy estimates} \label{Section4.45}
To derive the corresponding energy estimates for large frequencies, using Corollary \ref{CorDecreasingGlueHypRedEll}, we need to estimate the term
\[ \exp\bigg( -C\int_0^{t_{\xi_1}}\frac{1}{g(\tau)}d\tau \bigg) \exp\bigg( -\frac{|\xi|^2}{6}\int_{t_{\xi_1}}^tg(\tau)d\tau \bigg). \]
\begin{lemma} \label{LemmaDecreasingAuxiliary}
To a given positive constant $C$ we find a sufficiently small positive constant $C_1$ such that the following estimate holds for $t\geq t_{\xi_1}$:
\[\exp\bigg( -C\int_0^{t_{\xi_1}}\frac{1}{g(\tau)}d\tau \bigg)\exp\bigg( -C_1\int_{t_{\xi_1}}^tg(\tau)d\tau \bigg)\lesssim \exp\Big( -C_1G(t) \Big), \]
where $G=G(t)$ is defined as follows:
\begin{equation} \label{Eq:DefinitionG}
G(t) := 1+\int_0^tg(\tau)d\tau.
\end{equation}
\end{lemma}
\begin{proof}
Using the decreasing behavior of $g=g(t)$ implies for all $t\geq 0$ with a suitable $C_1>0$ the relations
\[ -C\frac{1}{g(t)}\leq -C_1g(t), \qquad \mbox{hence,} \qquad  -C\int_0^{t_{\xi_1}}\frac{1}{g(\tau)}d\tau\leq -C_1\int_0^{t_{\xi_1}}g(\tau)d\tau. \]
Considering the definition of $G=G(t)$ we get
\[ -C\int_0^{t_{\xi_1}}\frac{1}{g(\tau)}d\tau\leq C_1\Big( 1-G(t_{\xi_1}) \Big) \qquad \mbox{and} \qquad -C_1\int_{t_{\xi_1}}^tg(\tau)d\tau =  C_1\Big( G(t_{\xi_1})-G(t) \Big). \]
This implies what we wanted to show.
\end{proof}
From Corollary \ref{CorDecreasingGlueHypRedEll} and Lemma \ref{LemmaDecreasingAuxiliary} we obtain for large frequencies the following statement about ``an exponential type decay'' for large frequencies.
\begin{corollary} \label{CorDecreasingGlueAllLarge}
The following estimates hold for all $t\in[0,\infty)$:
\begin{align*}
|\xi|^{|\beta|}|\hat{u}(t,\xi)| &\lesssim \exp\Big( -CG(t) \Big)\Big( |\xi|^{|\beta|}|\hat{u}_0(\xi)| + |\xi|^{|\beta|-1}|\hat{u}_1(\xi)| \Big) \quad \mbox{for} \quad |\beta|\geq 1, \\
|\xi|^{|\beta|}|\hat{u}_t(t,\xi)| &\lesssim \exp\Big( -CG(t) \Big)\Big( |\xi|^{|\beta|+1}|\hat{u}_0(\xi)| + |\xi|^{|\beta|}|\hat{u}_1(\xi)| \Big) \quad \mbox{for} \quad |\beta|\geq 0,
\end{align*}
where $G=G(t)$ is given in \eqref{Eq:DefinitionG}.
\end{corollary}
For small frequencies we may use the estimates from Corollary \ref{CorDecreasingGlueHypN}.
\subsection{Conclusion} \label{Section4.46}
By Corollary \ref{CorDecreasingGlueHypN} and Corollary \ref{CorDecreasingGlueAllLarge} the proof of Theorem \ref{ThmDecreasing} is completed.
\end{proof}

\section{Models with non-integrable and slowly increasing time-dependent coefficient $g=g(t)$} \label{Section4.4}
We assume that  the function $g=g(t)$ satisfies the following conditions:
\begin{enumerate}
\item[\textbf{(E1)}] $g(t)>0$, $0\leq g'(t)\leq 1$ and $g''(t)\leq 0$ for all $t\in [0,\infty)$,
\item[\textbf{(E2)}] $\dfrac{1}{g}\notin L^1(0,\infty)$,
\item[\textbf{(E3)}] $|g'(t)|\leq C_1g(t)\dfrac{g(t)}{G(t)}$ for all $t\in [0,\infty)$, where $G(t):=1+\displaystyle\int_0^tg(\tau)d\tau$ and $C_1>0$, moreover, $|g''(t)| \leq \dfrac{g'(t)}{g(t)}$ with $t\geq t_0$, where $t_0$ is large.
\end{enumerate}

\begin{theorem} \label{ThmIncreasing}
Let us consider the Cauchy problem \eqref{modellinearviscoelasticdamped}, where the coefficient $g=g(t)$ satisfies the conditions \textbf{(E1)} to \textbf{(E3)}. Then, the Sobolev solution $u=u(t,x)$ satisfies the following estimates:
\begin{align*}
\big\||D|^{|\beta|}u(t,\cdot)\big\|_{L^2} &\lesssim \max\bigg\{ \bigg( 1+\int_0^tg(\tau)d\tau \bigg)^{-\frac{|\beta|}{2}}, \exp\bigg( -C\int_0^t\frac{1}{g(\tau)}d\tau \bigg) \bigg\}\|u_0\|_{H^{|\beta|}} \\
& \quad + \max\bigg\{ \bigg( 1+\int_0^tg(\tau)d\tau \bigg)^{-\frac{|\beta|-1}{2}}, \exp\bigg( -C\int_0^t\frac{1}{g(\tau)}d\tau \bigg) \bigg\}\|u_1\|_{H^{|\beta|-1}} \quad \mbox{for} \quad |\beta|\geq 1, \\
\big\||D|^{|\beta|}u_t(t,\cdot)\big\|_{L^2} &\lesssim  \max\bigg\{ \bigg( 1+\int_0^tg(\tau)d\tau \bigg)^{-\frac{|\beta|+1}{2}}, \exp\bigg( -C\int_0^t\frac{1}{g(\tau)}d\tau \bigg) \bigg\} \|u_0\|_{H^{|\beta|+1}} \\
& \quad + \max\bigg\{ \bigg( 1+\int_0^tg(\tau)d\tau \bigg)^{-\frac{|\beta|}{2}}, \exp\bigg( -C\int_0^t\frac{1}{g(\tau)}d\tau \bigg) \bigg\}\|u_1\|_{H^{|\beta|}}  \quad \mbox{for} \quad |\beta|\geq 0.
\end{align*}
\end{theorem}

\begin{remark} \label{Remarkparaboliceffect2}
The statements of Theorem \ref{ThmIncreasing} show that we have the parabolic effect if the term $\exp\Big( -C\int_0^t\frac{1}{g(\tau)}d\tau \Big)$ does not determine the decay. Namely, higher order energies decay faster with increasing order. We see this property in the following examples Example \ref{Exam1NonIntIncreasing}, \ref{Exam2NonIntIncreasing} and \ref{Exam3NonIntIncreasing}. On the contrary, if the term $\exp\Big( -C\int_0^t\frac{1}{g(\tau)}d\tau \Big)$ determines the decay, then we do not have any parabolic effect.
\end{remark}

\begin{exam} \label{Exam1NonIntIncreasing}
Let us choose $g(t)= (1 + t)^\gamma$ with $\gamma\in(0,1)$ (see Example \ref{Example5}). Then, $g=g(t)$ satisfies the assumptions of Theorem \ref{ThmIncreasing}. Consequently, the Sobolev solution $u=u(t,x)$ satisfies the following estimates:
\begin{align*}
\big\||D|^{|\beta|}u(t,\cdot)\big\|_{L^2} &\lesssim (1+t)^{-\frac{|\beta|(1+\gamma)}{2}}\|u_0\|_{H^{|\beta|}} + (1+t)^{-\frac{(|\beta|-1)(1+\gamma)}{2}}\|u_1\|_{H^{|\beta|-1}} \quad \mbox{for} \quad |\beta|\geq 1, \\
\big\||D|^{|\beta|}u_t(t,\cdot)\big\|_{L^2} &\lesssim (1+t)^{-\frac{(|\beta|+1)(1+\gamma)}{2}}\|u_0\|_{H^{|\beta|+1}} + (1+t)^{-\frac{|\beta|(1+\gamma)}{2}}\|u_1\|_{H^{|\beta|}} \quad \mbox{for} \quad |\beta|\geq 0.
\end{align*}
The term $\exp\Big( -C\int_0^t\frac{1}{g(\tau)}d\tau \Big)$ implies a faster decay. \end{exam}
\begin{exam} \label{Exam2NonIntIncreasing}
Let us consider $g(t)=\dfrac{e+t}{\log(e+t)}$. We have
\begin{align*}
\int_0^tg(s)ds  = \int_0^t  \frac{e+s}{\log(e+s)} ds =\frac{1}{2}\frac{(e+t)^2}{\log(e+t)} - \frac{e^2}{2} + \frac{1}{2}\int_0^t  \frac{e+s}{(\log(e+s))^2} ds.
\end{align*}
So, we get the estimate
\[ \int_0^t  \frac{e+s}{\log(e+s)} ds \approx \frac{(e+t)^2}{\log(e+t)}.\]
On the other hand, we have
\[ \int_0^t\frac{1}{g(s)}ds = \int_0^t\frac{\log(e+s)}{e+s}ds = \frac{1}{2}\Big( \log(e+t) \Big)^2-\frac{1}{2}. \]
Then, the Sobolev solution $u=u(t,x)$ satisfies the following estimates:
\begin{align*}
\big\||D|^{|\beta|}u(t,\cdot)\big\|_{L^2} &\lesssim \bigg( \frac{(e+t)^2}{\log(e+t)} \bigg)^{-\frac{|\beta|}{2}} \|u_0\|_{H^{|\beta|}}
+ \bigg( \frac{(e+t)^2}{\log(e+t)} \bigg)^{-\frac{|\beta|-1}{2}} \|u_1\|_{H^{|\beta|-1}} \quad \mbox{for} \quad |\beta|\geq 1, \\
\big\||D|^{|\beta|}u_t(t,\cdot)\big\|_{L^2} &\lesssim \bigg( \frac{(e+t)^2}{\log(e+t)} \bigg)^{-\frac{|\beta|+1}{2}} \|u_0\|_{H^{|\beta|+1}}
+ \bigg( \frac{(e+t)^2}{\log(e+t)} \bigg)^{-\frac{|\beta|}{2}} \|u_1\|_{H^{|\beta|}} \quad \mbox{for} \quad |\beta|\geq 0.
\end{align*}
The term $\exp\Big( -C\int_0^t\frac{1}{g(\tau)}d\tau \Big)$ implies a faster decay.
\end{exam}
\begin{exam} \label{Exam3NonIntIncreasing}
We consider now $g(t)=\dfrac{1+t}{\nu(1+t)}$, where $\nu=\nu(s)$ is a positive strictly increasing function which tends to $\infty$ for $s \to \infty$ with
\[ |\nu^{(k)}(1+t)|\leq C_k \Big( \frac{1}{1+t} \Big)^{k}, \,\, k=1,2 \,\,\, \mbox{and} \,\,\, C_1, C_2 \,\,\, \mbox{are positive constants.} \]
Then, we have
\begin{align*}
\int_0^tg(s)ds = \int_0^t\frac{1+s}{\nu(1+s)}ds &= \frac{1}{2}\frac{(1+s)^2}{\nu(1+s)}\Big|_0^t + \frac{1}{2}\int_0^t(1+s)^2\frac{\nu'(1+s)}{\nu^2(1+s)}ds
\approx \frac{1}{2}\frac{(1+t)^2}{\nu(1+t)}.
\end{align*}
This implies the following estimate:
\begin{align*}
\int_0^t\frac{1+s}{\nu(1+s)}ds \approx \frac{(1+t)^2}{\nu(1+t)}.
\end{align*}
On the other hand, we have
\begin{align*}
\int_0^t\frac{1}{g(s)}ds = \int_0^t\frac{\nu(1+s)}{1+s}ds = C(t_0) + \int_{t_0}^t\frac{\nu(1+s)}{1+s} ds
\geq C(t_0) + \nu(1+t_0) \log \frac{1+t}{1+t_0}.
\end{align*}
Here we use the strictly increasing behavior of $\nu$. Then, for sufficiently large time $t$ and $t_0$ we arrive at the estimate
\begin{align*}
-\int_0^t\frac{\nu(1+s)}{1+s}ds \leq -C(t_0)-\nu(1+t_0)\log \frac{1+t}{1+t_0}.
\end{align*}
Using $\lim_{s \to \infty} \nu(s)=\infty$ we see that the following estimates for Sobolev solution $u=u(t,x)$ are valid:
\begin{align*}
\big\||D|^{|\beta|}u(t,\cdot)\big\|_{L^2} &\lesssim \bigg( \frac{(1+t)^2}{\nu(1+t)} \bigg)^{-\frac{|\beta|}{2}} \|u_0\|_{H^{|\beta|}} + \bigg( \frac{(1+t)^2}{\nu(1+t)} \bigg)^{-\frac{|\beta|-1}{2}}\|u_1\|_{H^{|\beta|-1}} \quad \mbox{for} \quad |\beta|\geq 1, \\
\big\||D|^{|\beta|}u_t(t,\cdot)\big\|_{L^2} &\lesssim \bigg( \frac{(1+t)^2}{\nu(1+t)} \bigg)^{-\frac{|\beta|+1}{2}} \|u_0\|_{H^{|\beta|+1}}
+ \bigg( \frac{(1+t)^2}{\nu(1+t)} \bigg)^{-\frac{|\beta|}{2}} \|u_1\|_{H^{|\beta|}} \quad \mbox{for} \quad |\beta|\geq 0.
\end{align*}
The term $\exp\Big( -C\int_0^t\frac{1}{g(\tau)}d\tau \Big)$ implies a faster decay.
\end{exam}
\begin{exam} \label{Exam4NonIntIncreasing}
We consider now $g(t)=\mu(1+t)$ with $\mu \in (0,1)$. For this example we have no parabolic effect because now the term
$\Big( 1+\int_0^tg(\tau)d\tau \Big)^{-\frac{|\beta|}{2}}$ decays faster with increasing $|\beta|$ than the term  $\exp\Big( -C\int_0^t\frac{1}{g(\tau)}d\tau \Big)$.
\end{exam}
\begin{proof}[Proof of Theorem \ref{ThmIncreasing}]
We write the equation in (\ref{modellinearauxiliary2}) in the following form:
\begin{equation*} \label{auxiliaryproblemNonInt1}
D_t^2 v + \frac{g(t)^2}{4}|\xi|^4 v - \bigg( 1-\frac{g'(t)}{2} \bigg)|\xi|^2v=0.
\end{equation*}
We introduce $h=h(t)=1-\dfrac{g'(t)}{2}$. Notice, that $\frac{1}{2}\leq h(t)\leq1$. Then, we consider
\begin{equation*} \label{auxiliaryproblemNonInt2}
D_t^2 v + \frac{g(t)^2}{4}|\xi|^4 v - h(t)|\xi|^2v=0.
\end{equation*}
Due to condition \textbf{(E1)} the function $f=f(t)=\dfrac{g(t)^2}{h(t)}$ is monotonically increasing.
Thus, we have a separating line $t_\xi$ as the solution of the implicit equation $\dfrac{g(t)^2}{4h(t)}|\xi|^2=1$.
\medskip

Let us divide the extended phase space into the following zones:
\begin{itemize}
\item hyperbolic zone:
\[ \Zhyp = \Big\{ (t,\xi)\in[0,\infty)\times\mathbb{R}^n : 1- \frac{g(t)^2|\xi|^2}{4h(t)}\geq \frac{1}{4} \Big\}, \]
\item reduced zone:
\[ \Zred = \Big\{ (t,\xi)\in[0,\infty)\times\mathbb{R}^n : -\frac{1}{4} \leq 1- \frac{g(t)^2|\xi|^2}{4h(t)}\leq \frac{1}{4} \Big\}. \]
We introduce now a part of the elliptic region $\Pi_{ell}$, namely the region
\[ R_{\text{ell}} = \Big\{ (t,\xi)\in[0,\infty)\times\mathbb{R}^n : 1- \frac{g(t)^2|\xi|^2}{4h(t)}\leq -\frac{1}{4} \Big\}, \]
and divide this region $R_{\text{ell}}$ into the following zones:
\item pseudo-differential zone:
\[ \Zpd(N) = \Big\{ (t,\xi)\in[0,\infty)\times\mathbb{R}^n : -N \leq 1- \frac{g(t)^2|\xi|^2}{4h(t)}\leq - \frac{1}{4} \Big\}, \]
\item elliptic zone:
\[ \Zell(N) = \Big\{ (t,\xi)\in[0,\infty)\times\mathbb{R}^n : 1- \frac{g(t)^2|\xi|^2}{4h(t)}\leq -N \Big\}, \]
\end{itemize}
where $N>0$ is sufficiently large.

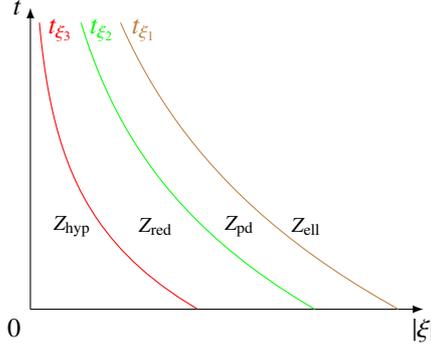
\begin{figure}[H]
\begin{center}
\begin{tikzpicture}[>=latex,xscale=1.1]
	\draw[->] (0,0) -- (4.7,0)node[below]{$|\xi|$};
	\draw[->] (0,0) -- (0,4)node[left]{$t$};
    \node[below left] at(0,0){$0$};
    \node[right]  at (0.1,3.7) {$\textcolor{red}{t_{\xi_3}}$};
    \node[right]  at (0.6,3.7) {$\textcolor{green}{t_{\xi_2}}$};
    \node[right]  at (1.1,3.7) {$\textcolor{brown}{t_{\xi_1}}$};
	\draw[domain=0:3.8,color=green,variable=\t] plot ({3.4*exp(-\t/2.2)},\t);
	\draw[domain=0:3.8,color=brown,variable=\t] plot ({4.4*exp(-\t/2.7)},\t);
	\node[color=black] at (3.3, 1.1){{\footnotesize $Z_{\text{ell}}$}};
	\node[color=black] at (2.5, 1.1){{\footnotesize $Z_{\text{pd}}$}};
	\node[color=black] at (1.5,1.1){{\footnotesize $Z_{\text{red}}$}};
	\draw[domain=0:3.8,color=red,variable=\t] plot ({2*exp(-\t/1.3)},\t);
	\node[color=black] at (.5,1.1){{\footnotesize $Z_{\text{hyp}}$}};
\end{tikzpicture}
\caption{Sketch of the zones for the case $g=g(t)$ is non-integrable and slowly increasing}
\label{fig.zone.increasing}
\end{center}
\end{figure}
We denote the separating line between elliptic and pseudo-differential zone as $t_{\xi_1}$, between pseudo-differential zone and reduced zone as $t_{\xi_2}$ and that between reduced zone and hyperbolic zone as $t_{\xi_3}$.

\subsection{Considerations in the elliptic zone $\Zell(N)$} \label{Section5.2}
\begin{proposition} \label{PropIncreasingEll}
The following estimates hold for all $t\in[t_{\xi_1},\infty)$:
\begin{align*}
|\xi|^{|\beta|}|\hat{u}(t,\xi)| &\lesssim \exp\bigg( -C\int_{t_{\xi_1}}^t\frac{1}{g(\tau)}d\tau \bigg)\Big( |\xi|^{|\beta|}|\hat{u}(t_{\xi_1},\xi)| + |\xi|^{|\beta|-1}|\hat{u}_t(t_{\xi_1},\xi)| \Big) \quad \mbox{for} \quad |\beta|\geq 1, \\
|\xi|^{|\beta|}|\hat{u}_t(t,\xi)| &\lesssim \exp\bigg( -C\int_{t_{\xi_1}}^t\frac{1}{g(\tau)}d\tau \bigg)\Big( |\xi|^{|\beta|+1}|\hat{u}(t_{\xi_1},\xi)| + |\xi|^{|\beta|}|\hat{u}_t(t_{\xi_1},\xi)| \Big) \\
& \quad + \exp\bigg( -\frac{|\xi|^2}{2}\int_{t_{\xi_1}}^tg(\tau)d\tau \bigg)|\xi|^{|\beta|}|\hat{u}_t(t_{\xi_1},\xi)| \quad \mbox{for} \quad |\beta|\geq 0.
\end{align*}
\end{proposition}
\begin{proof}

\noindent \textbf{Step 1.} \emph{A straight-forward estimate for the fundamental solution $E=E(t,s,\xi)$}

\begin{proposition} \label{propEstimateEjkIncreasing}
The fundamental solution $E=E(t,s,\xi)$ satisfies the following estimate:
\begin{align*} \label{EqPropEstEjkIncreasing}
\left(\begin{array}{cc}
|E^{(11)}(t,s,\xi)| & |E^{(12)}(t,s,\xi)| \\
|E^{(21)}(t,s,\xi)| & |E^{(22)}(t,s,\xi)|
\end{array}\right) \lesssim \exp\bigg( -C\int_s^t\frac{1}{g(\tau)}d\tau \bigg)\left(\begin{array}{cc}
1 & \dfrac{1}{g(s)|\xi|} \\
g(t)|\xi| & \dfrac{g(t)}{g(s)}
\end{array}\right),
\end{align*}
for all $t\geq s$ and $(t,\xi)$, $(s,\xi)\in\Zell(N)$, where the constant $C$ is independent of $(s,\xi), (t,\xi)\in\Zell(N)$.
\end{proposition}
\begin{proof} The proof coincides with the proof to Proposition \ref{propEstimateEjk}. In order to guarantee the invertibility of $N_1(t,\xi) = I + N^{(1)}(t,\xi)$, where $N^{(1)} = N^{(1)}(t,\xi)$ is defined in \eqref{EqNonIntegrableDecMatrN1}, we use the definition of the elliptic zone $g^2(t)|\xi|^2\geq 4(N+1)h(t)$ with sufficiently large $N$. In this way, we may conclude that the matrix $N^{(1)} = N^{(1)}(t,\xi)$ is invertible.
\end{proof}
\noindent\textbf{Step 2.} \emph{A refined estimate for the fundamental solution $E=E(t,s,\xi)$}
\begin{proposition} \label{Proposition5.4}
The fundamental solution $E=E(t,s,\xi)$ satisfies the following estimate:
\begin{align*}
& \left(\begin{array}{cc}
|E^{(11)}(t,s,\xi)| & |E^{(12)}(t,s,\xi)| \\
|E^{(21)}(t,s,\xi)| & |E^{(22)}(t,s,\xi)|
\end{array}\right) \lesssim \exp\bigg( -C\int_s^t\frac{1}{g(\tau)}d\tau \bigg)\left(\begin{array}{cc}
1 & \dfrac{1}{g(s)|\xi|} \\
\dfrac{1}{g(s)|\xi|} & \dfrac{1}{g^2(s)|\xi|^2}
\end{array}\right)
+ \exp\bigg( -|\xi|^2\int_s^tg(\tau)d\tau \bigg)\left(\begin{array}{cc}
0 & 0 \\
0 & 1
\end{array} \right)
\end{align*}
for all $(t,\xi)$, $(s,\xi)\in\Zell(N)$, where the constant $C$ is independent of $(s,\xi), (t,\xi)\in\Zell(N)$.
\end{proposition}
\begin{proof}
The proof is the same as the proof to Proposition \ref{LemmaRefinedEstimatesZellDecreasing} after using conditions \textbf{(E1)} and  \textbf{(E3)}. The only difference is the following: in the proof of Proposition \ref{LemmaRefinedEstimatesZellDecreasing} we used the decreasing behavior of $g=g(t)$ with $s\leq t$ to estimate
\[ \frac{1}{g(t)}+\frac{1}{g(s)}\exp\bigg( -C|\xi|^2\int_s^tg(\tau)d\tau \bigg)\lesssim \frac{1}{g(t)}. \]
But now we estimate
\[ \frac{1}{g(t)}+\frac{1}{g(s)}\exp\bigg( -C|\xi|^2\int_s^tg(\tau)d\tau \bigg)\lesssim \frac{1}{g(s)}. \]
For this reason the refined estimate for the entries of $E=E(t,s,\xi)$ differs to the estimate for the entries from Proposition \ref{LemmaRefinedEstimatesZellDecreasing}.
\end{proof}
This completes the proof to Proposition \ref{PropIncreasingEll}.
\end{proof}

\subsection{Considerations in the pseudo-differential zone $\Zpd(N)$} \label{SecNonIntegrableIncZpd}
\begin{proposition} \label{PropNonIntegrableIncZpd}
The following estimates hold for all $(t,\xi), (s,\xi) \in \Zpd(N)$ with $s \leq t$:
\begin{align*}
|\xi|^{|\beta|}|\hat{u}(t,\xi)| \lesssim \exp\bigg(  -\frac{1}{4(N+1)} |\xi|^2 \int_{s}^tg(\tau)d\tau \bigg)\Big( |\xi|^{|\beta|}|\hat{u}(s,\xi)| + |\xi|^{|\beta|-1}|\hat{u}_t(s,\xi)| \Big) \quad \mbox{for} \quad |\beta|\geq 1, \\
|\xi|^{|\beta|}|\hat{u}_t(t,\xi)| \lesssim \exp\bigg( -\frac{1}{4(N+1)} |\xi|^2 \int_{s}^tg(\tau)d\tau  \bigg)\Big( |\xi|^{|\beta|+1}|\hat{u}(s,\xi)| + |\xi|^{|\beta|}|\hat{u}_t(s,\xi)| \Big) \quad \mbox{for} \quad |\beta|\geq 0.
\end{align*}
\end{proposition}
\begin{proof}
We consider the transformed equation
\begin{equation*}
v_{tt} + \bigg( h(t)|\xi|^2-\dfrac{g^2(t)}{4}|\xi|^4\bigg)v=0.
\end{equation*}
If we define the micro-energy $V(t,\xi) = \Big( \dfrac{g(t)}{C_N}|\xi|^2v, D_tv \Big)^{\text{T}}$
then, we get the following system of first order:
\begin{equation} \label{NonIntegrableIncSystemZpd}
D_tV(t,\xi)=\left( \begin{array}{cc}
\dfrac{D_t g(t)}{g(t)} & \dfrac{g(t)}{C_N}|\xi|^2 \\
\dfrac{h(t)|\xi|^2-\frac{g^2(t)}{4}|\xi|^4}{\frac{g(t)}{C_N}|\xi|^2} & 0 \end{array} \right)V(t,\xi).
\end{equation}
We have the estimate
\[  \Big| h(t)|\xi|^2-\frac{g^2(t)}{4}|\xi|^4 \Big|\leq \frac{N}{N+1}\frac{g^2(t)}{4}|\xi|^4 \qquad \text{taking account of} \qquad \frac{g^2(t)}{4(N+1)}|\xi|^4 \leq h(t)|\xi|^2 \leq \frac{g^2(t)}{4}|\xi|^4. \]
Thus, to estimate the entries of the matrix \eqref{NonIntegrableIncSystemZpd} we will use
\[ \frac{\Big| h(t)|\xi|^2-\frac{g^2(t)}{4}|\xi|^4 \Big|}{\frac{g(t)}{C_N}|\xi|^2}\leq C_N\frac{N}{N+1}\frac{g(t)}{4}|\xi|^2. \]
To estimate $\dfrac{g(t)}{C_N}|\xi|^2$, let us choose $C_N=2+\kappa_N$. If we choose $\kappa_N \in (0,\frac{2}{N})$, then we get a suitable estimate. Let us choose $\kappa=\frac{1}{N}$, then $C_N=\frac{2N+1}{N}$. Thus, all entries of the matrix can be estimated by $\frac{4N+2}{4N+4} \frac{g(t)}{2}|\xi|^2$.
The entry $\dfrac{g'(t)}{g(t)}$ does not bring any additional term, it brings only a constant by using the definition of the pseudo-differential zone.

\begin{corollary} \label{CorollaryNonIntegrableIncreasing}
The fundamental solution $E=E(t,s,\xi)$ to \eqref{NonIntegrableIncSystemZpd} for all $t\geq s$ and $(t,\xi), (s,\xi)\in\Zpd(N)$ satisfies the following estimate:
\[ |E(t,s,\xi)|\leq \exp\bigg( \frac{2N+1}{4N+4} |\xi|^2 \int_s^t g(\tau)d\tau \bigg). \]
\end{corollary}
From the backward transformation and the equivalence $\dfrac{g(t)}{2}|\xi|^2\approx h(t)|\xi|$ in $\Zpd(N)$, we may conclude the desired statements of the proposition for all $t\geq s$ and $(t,\xi), (s,\xi)\in\Zpd(N)$.

\end{proof}
\subsection{Considerations in the hyperbolic zone $\Zhyp$ and in the reduced zone  $\Zred$} \label{Section5.1}
The treatment in the hyperbolic and reduced zone is the same as it was explained in Sections \ref{SecDecresingZhyp} and
\ref{SecDecresingZred}, respectively. We are able to extend the estimates from $\Zhyp$ to $\Zred$. For this reason, we obtain for $t\leq t_{\xi_2}$ the following estimates.
\begin{proposition} \label{NonIntegrableIncrCorZhyp}
The following estimates hold for all $t\in(0,t_{\xi_2}]$:
\begin{align*}
|\xi|^{|\beta|}|\hat{u}(t,\xi)| &\lesssim \exp\bigg( -\frac{|\xi|^2}{6}\int_{0}^tg(\tau)d\tau \bigg)\Big( |\xi|^{|\beta|}|\hat{u}_0(\xi)| + |\xi|^{|\beta|-1}|\hat{u}_1(\xi)| \Big) \quad \mbox{for} \quad |\beta|\geq 1, \\
|\xi|^{|\beta|}|\hat{u}_t(t,\xi)| &\lesssim \exp\bigg( -\frac{|\xi|^2}{6}\int_{0}^tg(\tau)d\tau \bigg)\Big( |\xi|^{|\beta|+1}|\hat{u}_0(\xi)| + |\xi|^{|\beta|}|\hat{u}_1(\xi)| \Big) \quad \mbox{for} \quad |\beta|\geq 0.
\end{align*}
\end{proposition}
\begin{proof}
The proof is the same as the proof to Proposition \ref{PropDecreasingHyp}.
\end{proof}
\subsection{Gluing procedure} \label{Section5.3}
For large frequencies we may use the estimates from Proposition \ref{PropIncreasingEll} because of $t_{\xi_1}=0$. We glue for small frequencies the estimates from Propositions \ref{PropNonIntegrableIncZpd} and \ref{NonIntegrableIncrCorZhyp} with the estimates from Proposition \ref{PropIncreasingEll}.
\begin{corollary} \label{CorIncreasingGlueHypRed}
The following estimates hold for all $t\in[t_{\xi_1},\infty)$ with a sufficiently large $N$:
\begin{align*}
|\xi|^{|\beta|}|\hat{u}(t,\xi)| &\lesssim \exp\bigg( -C\int_{t_{\xi_1}}^t\frac{1}{g(\tau)}d\tau \bigg)\exp\bigg( -\frac{1}{4(N+1)}|\xi|^2\int_{0}^{t_{\xi_1}}g(\tau)d\tau \bigg)\Big( |\xi|^{|\beta|}|\hat{u}_0(\xi)| + |\xi|^{|\beta|-1}|\hat{u}_1(\xi)| \Big) \quad \mbox{for} \quad |\beta|\geq 1, \\
|\xi|^{|\beta|}|\hat{u}_t(t,\xi)| &\lesssim \exp\bigg( -C\int_{t_{\xi_1}}^t\frac{1}{g(\tau)}d\tau \bigg)\exp\bigg( -\frac{1}{4(N+1)}|\xi|^2 \int_{t_{\xi_1}}^tg(\tau)d\tau \bigg)\Big( |\xi|^{|\beta|+1}|\hat{u}_0(\xi)| + |\xi|^{|\beta|}|\hat{u}_1(\xi)| \Big) \\
&\quad + \exp\bigg( -\frac{1}{4(N+1)}|\xi|^2\int_{0}^tg(\tau)d\tau \bigg)\Big( |\xi|^{|\beta|+1}|\hat{u}_0(\xi)|+|\xi|^{|\beta|}|\hat{u}_1(\xi)| \Big)
\quad \mbox{for} \quad |\beta|\geq 0.
\end{align*}
\end{corollary}
\begin{proof}
Let us begin to estimate $|\xi|^{|\beta|}\hat{u}(t,\xi)$. The statement of Proposition \ref{PropIncreasingEll} implies
\[ |\xi|^{|\beta|}|\hat{u}(t,\xi)| \lesssim \exp\bigg( -C\int_{t_{\xi_1}}^t\frac{1}{g(\tau)}d\tau \bigg)\Big( |\xi|^{|\beta|}|\hat{u}(t_{\xi_1},\xi)| + |\xi|^{|\beta|-1}|\hat{u}_t(t_{\xi_1},\xi)| \Big). \]
Using the estimates for $|\xi|^{|\beta|}|\hat{u}(t_{\xi_1},\xi)|$ and $|\xi|^{|\beta|}|\hat{u}_t(t_{\xi_1},\xi)|$ from Proposition \ref{PropNonIntegrableIncZpd} and \ref{NonIntegrableIncrCorZhyp} we have with a large $N$ the estimate
\begin{align*}
|\xi|^{|\beta|}|\hat{u}(t,\xi)| &\lesssim \exp\bigg( -C\int_{t_{\xi_1}}^t\frac{1}{g(\tau)}d\tau \bigg)\exp\Big( -\frac{1}{4(N+1)}\int_0^{t_{\xi_1}}g(\tau)d\tau \Big)\Big( |\xi|^{|\beta|}|\hat{u}_0(\xi)| + |\xi|^{|\beta|-1}|\hat{u}_1(\xi)| \Big).
\end{align*}
In the same way, we may conclude
\begin{align*}
|\xi|^{|\beta|}|\hat{u}_t(t,\xi)| &\lesssim \exp\bigg( -C\int_{t_{\xi_1}}^t\frac{1}{g(\tau)}d\tau \bigg)\exp\bigg( -\frac{1}{4(N+1)}|\xi|^2\int_0^{t_{\xi_1}}g(\tau)d\tau \bigg)\Big( |\xi|^{|\beta|+1}|\hat{u}_0(\xi)| + |\xi|^{|\beta|}|\hat{u}_1(\xi)| \Big)\\
& \quad + \exp\bigg( -\frac{1}{4(N+1)}|\xi|^2\int_{0}^tg(\tau)d\tau \bigg)\Big( |\xi|^{|\beta|+1}|\hat{u}_0(\xi)| + |\xi|^{|\beta|}|\hat{u}_1(\xi)| \Big).
\end{align*}
This completes the proof.
\end{proof}
\subsection{Energy estimates} \label{Section5.4}
For small frequencies we shall consider the interplay between two phase functions appearing in the estimates of Corollary \ref{CorIncreasingGlueHypRed}. For this reason we discuss the term $S_r(t,|\xi|)$ which is defined as follows:
\[ S_r(t,|\xi|) := |\xi|^r\exp\bigg( -C\int_{t_{\xi_1}}^t\frac{1}{g(\tau)}d\tau \bigg)\exp\bigg( -C_N|\xi|^2\int_0^{t_{\xi_1}}g(\tau)d\tau \bigg). \]
\begin{proposition} \label{LemmaIncreasingAuxiliary}
To a given positive constant $C$ there exists a small positive constant $C_N$ such that for $t>0$ it holds
\[ S_r(t,|\xi|) \lesssim \max_{\xi\in\mathbb{R}^n}\bigg\{ |\xi|^r\exp\bigg( -C_N|\xi|^2\int_0^{t}g(\tau)d\tau \bigg) \bigg\}\lesssim \bigg( 1+\int_0^{t}g(\tau)d\tau \bigg)^{-\frac{r}{2}} \quad \mbox{for} \quad r\geq0. \]
\end{proposition}
\begin{proof}
It is sufficient to verify the statement for large $t$, because for small  $t$ the set of admissible $\xi$ forms a compact set. To estimate the term $S_r(t,|\xi|)$ it is important that the first partial derivative $\partial_{|\xi|}S_r(t,|\xi|)$ is negative for $|\xi|\leq \varepsilon_r$. We have
\begin{align*}
\partial_{|\xi|} S_r(t,|\xi|) &= S_r(t,|\xi|)\bigg( \frac{r}{|\xi|} + C\frac{1}{g(t_{\xi_1})}d_{|\xi|}t_{\xi_1} - 2C_N|\xi|\int_0^{t_{\xi_1}}g(\tau)d\tau
- C_N|\xi|^2g(t_{\xi_1})d_{|\xi|}t_{\xi_1} \bigg) \\
& \leq  S_r(t,|\xi|)\bigg( \frac{r}{|\xi|} + \Big( C\frac{1}{g(t_{\xi_1})} - C_N|\xi|^2g(t_{\xi_1})  \Big)d_{|\xi|}t_{\xi_1} \bigg),
\end{align*}
Taking into account of $g^2(t_{\xi_1})|\xi|^2 = 4(N+1)h(t_{\xi_1})$ and $\frac{1}{2}\leq h(t)\leq 1$, we have
\[ |\xi|^2g(t_{\xi_1}) = \frac{|\xi|^2g^2(t_{\xi_1})}{g(t_{\xi_1})} = \frac{4(N+1)h(t_{\xi_1})}{g(t_{\xi_1})} \leq \frac{4(N+1)}{g(t_{\xi_1})}. \]
Therefore, if we choose the constant $C_N$ sufficiently small, then the term $C\dfrac{1}{g(t_{\xi_1})}$ dominates the term $C_N|\xi|^2g(t_{\xi_1})$. Moreover, after differentiation of $g^2(t_{\xi_1})|\xi|^2 = 4(N+1)h(t_{\xi_1})$, we get
\begin{align*}
\Big( 4(N+1)h'(t_{\xi_1}) - 2|\xi|^2g(t_{\xi_1})g'(t_{\xi_1}) \Big)d_{|\xi|}t_{\xi_1} = 2|\xi|g^2(t_{\xi_1}), \qquad d_{|\xi|}t_{\xi_1} = \frac{2|\xi|g^2(t_{\xi_1})}{4(N+1)h'(t_{\xi_1}) - 2|\xi|^2g(t_{\xi_1})g'(t_{\xi_1})}, \quad \mbox{respectively}.
\end{align*}
Employing $g(t)>0$, $0\leq g'(t)\leq 1$ and $g''(t)\leq0$ from condition \textbf{(E1)}, and taking account of $h'(t)=-\dfrac{g''(t)}{2}$, we find
\begin{align*}
4(N+1)h'(t_{\xi_1}) - 2|\xi|^2g(t_{\xi_1})g'(t_{\xi_1}) \geq -2(N+1)g''(t_{\xi_1}) - 2|\xi|^2g(t_{\xi_1}) \geq  - 2|\xi|^2g(t_{\xi_1}).
\end{align*}
We note that $4(N+1)h'(t_{\xi_1}) - 2|\xi|^2g(t_{\xi_1})g'(t_{\xi_1})<0$ using again $h'(t_{\xi_1})=-\dfrac{g''(t_{\xi_1})}{2}$, $|\xi|^2g(t_{\xi_1}) = 4(N+1)\dfrac{h(t_{\xi_1})}{g(t_{\xi_1})}$, $\frac{1}{2}\leq h(t_{\xi_1})\leq1$, and the condition \textbf{(E3)}, respectively, as follows:
\begin{align*}
4(N+1)h'(t_{\xi_1}) - 2|\xi|^2g(t_{\xi_1})g'(t_{\xi_1}) &= -2(N+1)g''(t_{\xi_1}) - 8(N+1)\frac{g'(t_{\xi_1})}{g(t_{\xi_1})}h(t_{\xi_1}) \\
& \leq 2(N+1)\frac{g'(t_{\xi_1})}{g(t_{\xi_1})} - 4(N+1)\frac{g'(t_{\xi_1})}{g(t_{\xi_1})}
= -2(N+1)\frac{g'(t_{\xi_1})}{g(t_{\xi_1})}<0.
\end{align*}
Then, we get
\begin{align*}
d_{|\xi|}t_{\xi_1} \leq \frac{2|\xi|g^2(t_{\xi_1})}{- 2|\xi|^2g(t_{\xi_1})} = -\frac{g(t_{\xi_1})}{|\xi|}.
\end{align*}
Moreover, for a fixed $r$ the term $\dfrac{r}{|\xi|}$ is dominated by the negative term
\[ |\xi|^2g(t_{\xi_1})d_{|\xi|}t_{\xi_1} \leq |\xi|^2g(t_{\xi_1})\Big( -\frac{g(t_{\xi_1})}{|\xi|} \Big) = -\frac{|\xi|^2g^2(t_{\xi_1})}{|\xi|} = -\frac{4(N+1)h(t_{\xi_1})}{|\xi|} \leq -2(N+1)|\xi|^{-1} \]
if we choose $N$ large enough.
In order to complete the proof it is sufficient to study small frequencies with $|\xi|\leq \varepsilon_r$. For $|\xi|\geq \varepsilon_r$ we  have an ``exponential decay'' from the elliptic zone. Let us now fix $t>0$. Then, the above term takes its maximum for the $|\tilde{\xi}|$ satisfying $t=t_{\tilde{\xi}_1}$. For $t=t_{\tilde{\xi}_1}$, the first integral vanishes in $S_r(t,|\xi|)$. Consequently, we get
\begin{align*}
& S_r(t,|\xi|) \leq S_r(t_{\tilde{\xi}_1},|\tilde{\xi}|) = |\tilde{\xi}|^r\exp\bigg( -C|\tilde{\xi}|^2\int_0^{t_{\tilde{\xi}_1}}g(\tau)d\tau \bigg) \\
& \lesssim \max_{\xi\in\mathbb{R}^n}\bigg\{ |\xi|^r\exp\bigg( -C|\xi|^2\int_0^{t}g(\tau)d\tau \bigg) \bigg\} \lesssim \bigg( 1+\int_0^{t}g(\tau)d\tau \bigg)^{-\frac{r}{2}}.
\end{align*}
The proof is completed. \end{proof}
Using Proposition \ref{LemmaIncreasingAuxiliary} we obtain the following statement.
\begin{corollary} \label{CorIncreasingGlueAllLarge}
The following estimates hold for all $t>0$ and small frequencies $0<|\xi|\leq 1$:
\begin{align*}
|\xi|^{|\beta|}|\hat{u}(t,\xi)| &\lesssim \bigg( 1+\int_0^{t}g(\tau)d\tau \bigg)^{-\frac{|\beta|}{2}}|\hat{u}_0(\xi)| + \bigg( 1+\int_0^{t}g(\tau)d\tau \bigg)^{-\frac{|\beta|-1}{2}}|\hat{u}_1(\xi)| \quad \mbox{for} \quad |\beta|\geq 1, \\
|\xi|^{|\beta|}|\hat{u}_t(t,\xi)| &\lesssim \bigg( 1+\int_0^{t}g(\tau)d\tau \bigg)^{-\frac{|\beta|+1}{2}}|\hat{u}_0(\xi)| + \bigg( 1+\int_0^{t}g(\tau)d\tau \bigg)^{-\frac{|\beta|}{2}}|\hat{u}_1(\xi)| \quad \mbox{for} \quad |\beta|\geq 0.
\end{align*}
\end{corollary}
For large frequencies we may use the estimates from Proposition \ref{PropIncreasingEll} because of $t_{\xi_1}=0$. These estimates imply an ``exponential type decay''.
\begin{corollary} \label{CorIncreasingEllLarge}
The following estimates hold for all $t>0$ and large frequencies $|\xi|\geq 1$:
\begin{align*}
|\xi|^{|\beta|}|\hat{u}(t,\xi)| &\lesssim \exp\bigg( -C\int_{0}^t\frac{1}{g(\tau)}d\tau \bigg)\Big( |\xi|^{|\beta|}|\hat{u}_0(\xi)| + |\xi|^{|\beta|-1}|\hat{u}_1(\xi)| \Big) \quad \mbox{for} \quad |\beta|\geq 1, \\
|\xi|^{|\beta|}|\hat{u}_t(t,\xi)| &\lesssim \exp\bigg( -C\int_{0}^t\frac{1}{g(\tau)}d\tau \bigg)\Big( |\xi|^{|\beta|+1}|\hat{u}_0(\xi)| + |\xi|^{|\beta|}|\hat{u}_1(\xi)| \Big) + \exp\bigg( -C\int_{0}^tg(\tau)d\tau \bigg)|\xi|^{|\beta|}|\hat{u}_1(\xi)| \quad \mbox{for} \quad |\beta|\geq 0.
\end{align*}
\end{corollary}
\subsection{Conclusion}\label{Section5.5}
Taking into consideration all these estimates and the fact, that the statements from Proposition \ref{PropNonIntegrableIncZpd}, Corollaries \ref{CorIncreasingGlueAllLarge} and \ref{CorIncreasingEllLarge} determine the decay estimates and regularity of the data, respectively, we may conclude the proof of Theorem \ref{ThmIncreasing}.
\end{proof}

\section{Concluding remarks} \label{Section6}
\begin{remark}\label{Remark6.1} {\it Scale-invariant models} \\
Let us turn to the scale-invariant case
\begin{equation} \label{modelscaleinvariant}
\begin{cases}
u_{tt}- \Delta u + \mu(1+t)(-\Delta)u_t=0, &(t,x) \in [0,\infty) \times \mathbb{R}^n, \\
u(0,x)= u_0(x),\quad u_t(0,x)= u_1(x), &x \in \mathbb{R}^n,
\end{cases}
\end{equation}
where $\mu>0$. Taking into consideration results from \cite{Reissig} one can arrive at the following estimates:
\begin{align*}
&& |\xi|^{|\beta|}|\hat u(t,\xi)|\leq C\Big( |\xi|^{|\beta|}|\hat u(0,\xi)|+ |\xi|^{|\beta|-1}|\hat u_t(0,\xi)| \Big) \,\,\,\mbox{for} \,\,\,|\beta|\geq 1,\\
&& |\xi|^{|\beta|}|\hat u_t(t,\xi)|\leq C(1+t) \Big( |\xi|^{|\beta|+1}|\hat u(0,\xi)|+ |\xi|^{|\beta|}|\hat u_t(0,\xi)| \Big)\,\,\,\mbox{for} \,\,\,|\beta|\geq 0,
\end{align*}
for all $(t,\xi) \in [0,\infty) \times \mathbb{R}^n$. This implies that multiplication of $|\xi|^{|\beta|}$ on the left-hand sides in the above estimates gives no faster decay depending on $|\beta|$ of higher order energies. Thus, we have no parabolic effect. In this way we get the following result.

\begin{theorem} \label{Theoremscaleinvariant}
Let us consider the Cauchy problem (\ref{modelscaleinvariant}), where the data $(u_0,u_1)$ is assumed to belong to
$\dot{H}^{|\beta|} \times \dot{H}^{|\beta|-1}$ with $|\beta|\geq 1$.
Then, we have the following estimates for Sobolev solutions:
\begin{align*}
\big\||D|^{|\beta|} u(t,\cdot)\big\|_{L^2} & \lesssim  \|u_0\|_{\dot{H}^{|\beta|}} + \|u_1\|_{\dot{H}^{|\beta|-1}},\\
\big\||D|^{|\beta|-1} u_t(t,\cdot)\big\|_{L^2} & \lesssim (1+t)\Big( \|u_0\|_{\dot{H}^{|\beta|}} + \|u_1\|_{\dot{H}^{|\beta|-1}} \Big).
\end{align*}
\end{theorem}
\end{remark}

\begin{remark} \label{Remark6.2} {\it Parabolic effect} \\
We have an almost complete picture on the validity of the ``parabolic effect'', that is, higher order energies have a faster decay with increasing order. In the case, that $g=g(t)$ is ``above
the scale-invariant case'' (see previous Remark \ref{Remark6.1}) we have shown in Section \ref{Section4.1} that there is no any parabolic effect.
In the case, that $g=g(t)$ is decreasing and integrable we have shown in Section \ref{Section4.2} that we do not have any parabolic effect.
Under the assumptions for $g=g(t)$ in Sections \ref{Section4.3} and \ref{Section4.4} we have, in general, the parabolic effect.
\end{remark}
\begin{remark} \label{Remark6.3} {\it Comparison of obtained results} \\
The estimates of Theorem \ref{Theorem1} are in some sense related to the estimates of Theorem \ref{Theoremscaleinvariant} although $g(t)=\mu(1+t),\,\mu>0$
does not satisfy the assumption \textbf{(A2)}. The difference between both results is the difference of order of regularity of the data which is $2$ in
Theorem \ref{Theorem1} and $1$ in Theorem \ref{Theoremscaleinvariant}.
The estimates of Theorem \ref{ThmIncreasing} are in some sense related to the estimates of Theorem \ref{Theoremscaleinvariant}. Theorem \ref{ThmIncreasing} gives us at least the information that in the scale-invariant case we cannot expect any parabolic effect.
The estimates of Theorem \ref{ThmDecreasing} are compatible with the estimates of Theorem \ref{ThmIncreasing}. If we formally put $g(t)\equiv 1$, then both results coincide. In the case $g(t) \equiv 1$ we apply straight-forward Fourier analysis.
The estimates of Theorem \ref{TheoremIntegrableCase} are in some sense related to the estimates of Theorem \ref{ThmDecreasing}. If $g \in L^1(0,\infty)$, then from Theorem \ref{ThmDecreasing} we get formally no decay estimates and no parabolic effect anymore.
The difference between both results is the difference of order of regularity of the data which is $2$ in
Theorem \ref{TheoremIntegrableCase} and $1$ in Theorem \ref{ThmDecreasing}.
\end{remark}

\section*{Acknowledgments}
The discussions on this paper began during the research stay of the first author at the Technical University Bergakademie Freiberg, within the period August to October 2022. This stay was supported by Funda\c c\~ao de Amparo \`a Pesquisa do Estado de S\~ao Paulo (FAPESP), Grant 2021/01743-3.


\end{document}